\documentclass[11pt]{article}
\usepackage[margin=1.0in]{geometry}
\usepackage[utf8]{inputenc}
\usepackage{mathrsfs,amssymb,amsfonts,amsthm,amsmath,amsbsy}
\usepackage{mathtools}
\usepackage{mathabx}
\usepackage{graphicx}
\usepackage{wrapfig}
\usepackage{caption}
\usepackage{subfig}
\usepackage{dsfont}
\usepackage{float}
\usepackage{upgreek}
\usepackage{listings}
\usepackage{lipsum}
\usepackage[export]{adjustbox}
\usepackage[
backend=bibtex,
style=alphabetic,
citestyle=ieee-alphabetic,
maxnames=10
]{biblatex}
\usepackage{titlesec}
\usepackage{hyperref}
\hypersetup{
	colorlinks,
	citecolor=blue,
	linkcolor=blue,
	urlcolor=green}

\newtheorem{theorem}{Theorem}[section]
\newtheorem{proposition}[theorem]{Proposition}

\newtheorem{lemma}[theorem]{Lemma}

\newtheorem{assumption}[theorem]{Assumption}

\newtheorem{definition1}[theorem]{Definition}
\newtheorem{condition}{Condition}

\newtheorem{remark1}[theorem]{Remark}

\DeclareMathOperator*{\argmax}{argmax}
\DeclarePairedDelimiter\floor{\lfloor}{\rfloor}
\DeclarePairedDelimiter\ceil{\lceil}{\rceil}
\usepackage[ruled,vlined]{algorithm2e}
\usepackage[font=small]{caption}

\newcommand{\Prob}{\mathbb{P}}
\newcommand{\EVal}{\mathbb{E}}
\newcommand{\R}{\mathbb{R}}

\newcommand{\Z}{\mathbb{Z}}
\newcommand{\N}{\mathbb{N}}
\usepackage[ruled,vlined]{algorithm2e}
\usepackage[font=small]{caption}
\usepackage[usenames,dvipsnames]{xcolor}
\usepackage{tikz}
\usepackage{subfig}
\usepackage{dsfont,fancyhdr}
\usetikzlibrary{matrix}

\usepackage{color}

\title{Aging and sub-aging\\for one-dimensional random walks\\amongst random conductances}
\author{D.\ A.\ Croydon\footnote{Research Institute for Mathematical Sciences, Kyoto University, croydon@kurims.kyoto-u.ac.jp}, D.\ Kious\footnote{Department of Mathematical Sciences, University of Bath, d.kious@bath.ac.uk} and C.\ Scali\footnote{Department of Mathematical Sciences, University of Bath, cs2458@bath.ac.uk}}

\begin{document}

\maketitle

\begin{abstract}
\noindent
We consider random walks amongst random conductances in the cases where the conductances can be arbitrarily small, with a heavy-tailed distribution at 0, and where the conductances may or may not have a heavy-tailed distribution at infinity. We study the long time behaviour of these processes and prove aging statements. When the heavy tail is only at 0, we prove that aging can be observed for the maximum of the process, i.e.~the same maximal value is attained repeatedly over long time-scales. When there are also heavy tails at infinity, we prove a classical aging result for the position of the walker, as well as a sub-aging result that occurs on a shorter time-scale.
\end{abstract}

\section{Introduction}

In this paper, we study the aging phenomenon for random walks amongst random conductances in dimension one. It is now well understood that these random walks can exhibit atypical behaviour, in the sense that they can be sub-diffusive due to the presence of atypical areas in the environment. There are two ways to create a slow-down for the walk. First, the environment can have very small conductances, acting as walls into which the walker will collide for a long time before overcoming them. Second, the environment can contain large conductances that the walker, at each visit, will cross back and forth many times before exiting them, hence these conductances act like traps that the walker has to escape. These atypical areas appear when the law of the conductances are chosen such that they have a heavy tail at $0$, for the walls, or at infinity, for the traps.

Let us first discuss the trapping mechanism corresponding to large conductances. This effect is reminiscent of Bouchaud's trap model, which was introduced by the physicist Jean-Philippe Bouchaud \cite{Bou92}. This model consists of a continuous-time random walk on $\mathbb{Z}^d$ such that, at each vertex, a trap is placed with exponential waiting times whose average is random, independent and heavy-tailed. The behaviour of this model is strikingly different in dimension one and in dimension two and above. In dimension two and above, it has been proved by Ben Arous and \v Cern\'y \cite{BenArousZd}, that the properly rescaled random walk converges to a fractional-kinetics process, which is a Brownian motion time-changed by the inverse of an independent stable subordinator. In \cite{BarlowCernyZD, Cerny2d}, it has been proved that random walks in random conductances with heavy tails (at infinity) also converge in two or more dimensions to a fractional-kinetics. In dimension one, the behaviour of the Bouchaud trap model is radically different, with the limiting process being a singular diffusion called the FIN diffusion,  first defined by Fontes, Isopi and Newman \cite{FIN}. This diffusion falls into the large class of spatially-subordinated Brownian motions later defined  in \cite{BenArousCabezasCernyRoyfman}, and can be described as follows. Consider a degenerate Poisson point process on $\mathbb{R}\times\mathbb{R}^+$, corresponding to the limit of the positions of the traps $\mathbb{Z}$ together with their rescaled depth (or average waiting time). The FIN diffusion is a Brownian motion time-changed by the inverse of its own local time on this degenerate Poisson point process. A major difference between FIN diffusion and fractional-kinetics is that the time-change is \emph{not} independent of the Brownian motion itself: indeed, in dimension one, the delay that is being accumulated at time $t$ depends on the trajectory up to this time. Naturally, one can expect that a random walk among heavy-tailed random conductances in one dimension converges has a FIN diffusion scaling limit, and indeed this was proved in \cite{Cerny2d}.

The slow-down created by walls is of different nature. Indeed, the walk will not stay put at the same place a long time but its maximum (and minimum) will not change for long periods, as the walker will collide against the high walls in the environment many times. This phenomenon is only observed in dimension one because, in higher dimension, the walker will easily go around these walls without any major slow-down. In dimension one, the scaling limit of this walk was derived by Kawazu and Kesten \cite{KawazuKesten} and is in yet another class of processes: it roughly resembles FIN diffusion, except that the Brownian motion is not time-changed by the inverse of its local time on a Poisson process, but is spatially-deformed by the inverse of a subordinator, where the atoms in the Poisson process corresponding to the jumps in the subordinator represent the scaling limit of the walls of high resistance.

In this paper, we study environments either containing walls, or containing both walls and traps. (See the end of Section \ref{sec:mainres} for a discussion of possible further results to ours.) In the second case, the rescaled environment can be seen in the limit as the superimposition of two independent Poisson processes: one for the walls and one for the traps. We use a unifying approach to prove scaling limits of such random walks, that is, the theory of stochastic processes associated with resistance forms, see \cite{CroydonHamblyKumagai,CroydonResistanceForm} (see also \cite{ALW} for a related work on trees). We would further like to mention the works \cite{ESZ1,ESZ2,ESZ3} that are conceptually important. Let us also acknowledge the works \cite{BergerSalviShort,BergerSalvi2020,ZindyAging} and \cite{Frib_conduc,SubBallisticZD,QuenchedBiasedRWRC}, which deal with a biased version of the random walks on random conductances in dimension one and in higher dimensions, respectively.

Now, as noted above, the main goal of our paper is to study the aging phenomenon for these random walks, when in the presence of walls, or of walls and traps, in dimension one. Bouchaud's trap model was introduced by Bouchaud  as a toy model to understand aging for the dynamics of spin glasses, which tend to stay around states with atypically low energy for long periods. More generally, aging of a system is the phenomenon that the time it takes to observe a change in the state of the system is of the order of the age of the system. As explained by Ben Arous and \v Cern\'y \cite{BenArousZd}, proving an aging result involves finding a two-point function $F(t,t\cdot h)$, with $h>1$, measuring the state of the system after it has aged for a further time $(h-1)t$ after time $t$ that exhibits a non-trivial limit $F(h)$ as $t\rightarrow\infty$.

The choice of the two-point function is important and depends on the details of the model. For instance, in the case of Bouchaud's trap model, Rinn, Maass and Bouchaud \cite{RinnMaassBouchaud} considered $F$ to be the probability that the random walk is at the same location at times $t$ and $t\cdot h$. In that case, the random walk is likely to visit the same few places for long periods. A more precise statement was conjectured in \cite{RinnMaassBouchaud} and proved in \cite{BenArousCerny2AgingRegimes}, corresponding to the phenomenon called {\it sub-aging}. In that case the two-point function is the probability that the random walk stays in the same position the whole time, from time $t$ to time $t+t^{\gamma}$, where $\gamma<1$ is related to the tails of the averages of the exponential waiting times in Bouchaud's trap model.

For random walks amongst random conductances, the aging results in the case where the environment has walls, but no traps, is different. Indeed, the random walker will not stay around the same locations for a very long time but its maximum will. Indeed, the probability that the maximum of the random walk at time $t$ is equal to the maximum obtained between times $t$ and $t\cdot h$ will have a non-trivial limit for $h>1$. We state this result in Theorem \ref{theoremAgingNEW}. (See also the comment following Theorem \ref{theoremAgingNEW} concerning the more detailed statement that will be given later in the article.)

When the environment has both walls and traps, because of the similarities explained above, one can expect that the corresponding random walk amongst random conductances will show aging and sub-aging similar to that of Bouchaud's trap model. We prove this result is indeed true, with the exception that we have an additional slow-down effect due to the presence of walls, see Theorem \ref{theoremSubAging0}.

Finally, let us emphasize that all the aging and sub-aging results outlined above are in dimension one and are proved under the annealed measures, that is the measures that average over all possible environments. For Bouchaud's trap model, quenched results were proved in dimension two in \cite{BenArousCernyMountford} and in dimension three and above \cite{CernyThesis}, with different slow-downs due to the difference in the Green function of simple random walk. Such quenched results do not hold in dimension one, see \cite{CroydonMuirhead}, because the environment seen from the walker is not mixing enough. It is reasonable to believe that the same is true for one-dimensional random walks amongst random conductances.

\paragraph{Acknowledgements.}

DK was partially supported by the EPSRC grant EP/V00929X/1. DC was supported by JSPS Grant-in-Aid for Scientific Research (C) 19K03540 and the Research Institute for Mathematical Sciences, an International Joint Usage/Research Center located in Kyoto University. CS was supported by a scholarship from the EPSRC Centre for Doctoral Training in Statistical Applied Mathematics at Bath (SAMBa), under the project EP/S022945/1.

\subsection{Model(s)}

As already set out above, the model that we consider is the random walk amongst random conductances. In particular, we will consider two different versions of this model, one with heavy-tailed resistances, and one with both heavy-tailed resistances and heavy-tailed conductances. Let $E=\{\{i, i+1\}:\:i \in \Z\}$ indicate the nearest-neighbour links on $\Z$, and let $(c(\{i, i+1\}))_{i \in \Z}$ be a family of positive weights associated with those edges. Moreover, for each $x \in \Z$, define $c(x) = c(\{x, x-1\}) + c(\{x, x+1\})$. Then, one can naturally define a random walk on this lattice starting from the origin by considering the continuous-time Markov chain $X_t$ with state space $\Z$ and generator
\begin{equation}\label{eqn:Generator}
(Lf)(x) \coloneqq \sum_{y : |x - y|\le1} \frac{c(\{x, y\})}{c(x)} \left( f(y) - f(x) \right).
\end{equation}
Note that, as this Markov chain has exponential holding times of mean $1$ at each site, the long-term behaviour of this model closely resembles the one of the discrete-time Markov chain with jump probabilities
\begin{equation}\label{RandomWalkConductances}
P\left( X_{k + 1} = y | X_{k} = x \right) = P_x\left( X_{1} = y \right) \coloneqq \frac{c(\{x, y\})}{c(x)}.
\end{equation}
We define the random walk amongst random conductances by selecting the weights $(c(\{i, i+1\}))_{i \in \Z}$ to be independent and identically-distributed (i.i.d.)\ under some probability measure $\mathbf{P}$ on a probability space $\Omega$. For a fixed realisation $\omega \in \Omega$ of the environment, it is possible to define a walk as in \eqref{eqn:Generator} and such that $X_{0} = x, x \in \Z$ almost surely; we call the distribution of such a walk its \textbf{quenched law}, and denote it by $P_{x}^\omega(\cdot)$. Moreover, the \textbf{annealed law} of this random walk is obtained by integrating out the environment:
\[    \Prob_{x}(\cdot) \coloneqq \mathbf{E}\left[ P_{x}^\omega( \cdot) \right] = \int_\Omega P_{x}^\omega( \cdot)\mathbf{P}\left(d\omega\right).\]
We also use the notation $P^\omega(\cdot) = P_{0}^\omega(\cdot)$ and $\Prob(\cdot) = \Prob_{0}(\cdot)$.

Let  $(r(\{i, i+1\}))_{i \in \Z}$ be the family of associated resistances, where, for all $i \in \Z$, $r(\{i, i+1\}) = 1/c(\{i, i+1\})$. Let us state the two fundamental assumptions on the distribution of the environment under which we will work. To distinguish between the two cases more clearly, in the second of the cases, we will denote the probability measure on the probability space on which the environment is built by $\widetilde{\mathbf{P}}$.

\begin{assumption}\label{AssumptionTail}
Fix $\alpha_0$ and $\alpha_\infty$ to be two constants in $(0, 1)$.
\begin{description}
\item [Random walk amongst random walls.] The family $(c(\{i, i+1\}))_{i \in \Z}$ satisfies:
    \begin{equation}\label{RW}\tag{RW}
    \mathbf{E}\left[ c(\{0, 1\}) \right] < \infty \quad\quad \textnormal{and} \quad\quad \mathbf{P}\left( r(\{0, 1\}) > t \right) = L_0(t) t^{-\alpha_0},\:\forall t>1,
    \end{equation}
  where $L_0(t)$ is slowly varying at infinity. We recall that a function $L$ is slowly varying at infinity if $\lim_{x \to \infty} L(ax)/L(x) = 1$ for all $a > 0$.

\item [Random walk amongst random walls and traps.] The family $(c(\{i, i+1\}))_{i \in \Z}$ satisfies:
    \begin{equation}\label{RWT}\tag{RWT}
    \widetilde{\mathbf{P}}\left[ c(\{0, 1\}) > t \right] = L_\infty(t) t^{-\alpha_\infty} \quad\quad \textnormal{and} \quad\quad \widetilde{\mathbf{P}}\left( r(\{0, 1\}) > t \right) = L_0(t) t^{-\alpha_0},\:\forall t>1,
    \end{equation}
    where both $L_0(t)$ and $L_\infty(t)$ are slowly varying at infinity.
\end{description}
\end{assumption}

In order to make the model more general, we also include a vanishing bias as in \cite{MottModel, GantertMortersWachtelAging}. We can recover the unbiased model by setting the bias parameter to $0$.

\begin{assumption}\label{AssumptionWeakBias}
For both \eqref{RW} and \eqref{RWT} as in Assumption~\ref{AssumptionTail}, and for $\lambda\in\mathbb{R}$, the $n$-scale weakly-biased random walk is the continuous-time Markov process with generator as in \eqref{eqn:Generator} with conductances and resistances deterministically-tilted in the following way:
\[c^{\lambda/n}(\{i, i+1\}) = c(\{i, i+1\}) e^{2 \lambda i/n},\qquad r^{\lambda/n}(\{i, i+1\})=\frac{1}{c^{\lambda/n}(\{i, i+1\})}.\]
\end{assumption}

We will denote by $(X_{t})_{t \ge 0}$ the random walk under \eqref{RW} and Assumption~\ref{AssumptionWeakBias}, whilst $(\widetilde{X}_{t})_{t \ge 0}$ will denote the one under \eqref{RWT} and Assumption~\ref{AssumptionWeakBias}. For $x \in \mathbb{Z}$, will denote by $P^{\omega, \lambda/n}_{x}$ and $\mathbb{P}_{x}^{\lambda/n}$ respectively the quenched and annealed laws of $(X_{t})_{t \ge 0}, X_0 = x$; note that we do not need to change the notation for $\mathbf{P}$. The notation corresponding to $(\widetilde{X}_{t})_{t \ge 0}$ is set to $\widetilde{P}_{x}^{\omega, \lambda/n}$ and $\widetilde{\mathbb{P}}_{x}^{\lambda/n}$. We drop the subscript and write $P^{\omega, \lambda/n}$ and $\mathbb{P}^{\lambda/n}$ when $x = 0$, the same conventions is used for $\widetilde{P}^{\omega, \lambda/n}$ and $\widetilde{\mathbb{P}}^{\lambda/n}$.

Let us also introduce the scaling terms
\begin{equation} \label{eqn:RightScalingTerms}
	d_{n, \infty} \coloneqq \inf \left\{ t > 0 : \widetilde{\mathbf{P}} \left( c(\{0, 1\}) > t\right) \le \frac{1}{n} \right\}, \quad d_{n, 0} \coloneqq \inf \left\{ t > 0 : \widetilde{\mathbf{P}} \left( r(\{0, 1\}) > t\right) \le \frac{1}{n} \right\}.
\end{equation}
It will perhaps be useful to the reader to indicate that $d_{n,\infty}$ should be thought of as $n^{1/\alpha_\infty}$ and $d_{n,0}$ should be thought of as $n^{1/\alpha_0}$. These quantities are different in general (by a factor of a slowly-varying function), but of the same order when the tails of the distributions above are polynomial.

\subsection{Main results}\label{sec:mainres}

In this section, we provide a first statement of our main results. These results will be restated in a more precise manner in Section \ref{sec:more-results}. Towards these ends, let us start by introducing the time scales at which we will look at the processes $X$ and $\widetilde{X}$, respectively:
\begin{equation}\label{eqn:RightTimeScales}
a_n \coloneqq n d_{n, 0} \quad\quad \text{and} \quad\quad b_n \coloneqq d_{n, \infty} d_{n, 0}.
\end{equation}
Let us explain why these are the relevant time scales for the scaling of our processes. Under the assumption \eqref{RW}, the natural time scale (corresponding to a distance scale of $n$) is given by $a_n=n d_{n,0}$, which is larger than the $n^2$ time scaling seen for a usual symmetric random walk on $\mathbb{Z}$. The factor $d_{n,0}$ represents the size of the largest resistances (walls) met by the random walk on the relevant scale, and, taking into account the excursions away from these walls, one can check that the time accumulated by the random walk bouncing against these before being able to overcome them is precisely of the order $n d_{n,0}$. Under the assumption \eqref{RWT}, the natural scale becomes $b_n$, because now the motion of the random walk is also perturbed by the large conductances, which act in the limit as large exponential waiting times, similar to what happens in the case of Bouchaud's trap model.

For our first main result, we consider the supremum of $X$ over a time interval, i.e.\
\[	\widebar{X}_{[a, b]} \coloneqq \sup_{a \le s \le b} X_{s},\]
furthermore, we write $\widebar{X}_{t} \coloneqq \widebar{X}_{[0, t]}$ for the running supremum of $X$.

\begin{theorem}\label{theoremAgingNEW}
	Under \eqref{RW} and Assumption~\ref{AssumptionWeakBias}, for all $0<\alpha_0<1$,  the following aging statement holds. There exists an explicit function $\theta:(1,\infty)\to(0,1)$ such that, for all $h>1$,
	\begin{equation*}
		\lim_{n \to \infty} \Prob^{\lambda/n}\left( \widebar{X}_{a_n} = \widebar{X}_{[a_n,h a_n]} \right) = \theta(h).
	\end{equation*}
\end{theorem}

The function $\theta$ above depends on the law of the environment, i.e.~on the law of the conductances, and on the bias parameter $\lambda$.
We remark that, for this model, it can further be checked from the arguments of this article that $\Prob^{\lambda/n}\left( \widebar{X}_{a_n} = \widebar{X}_{ha_n} \right)$ converges to a non-trivial limit for all $h>0$. We highlight, however, that such a result is hardly unique to the current model. Indeed, it will hold for the usual simple symmetric random walk on $\mathbb{Z}$, with the limiting expression being given by the corresponding probability for the standard Brownian motion. What is distinct to this setting, and will be made precise in Theorem \ref{TheoremGap} below, is that, with high probability, the location of the running supremum of $X$ has a particular feature, namely being to the left of an edge of large resistance, i.e.\ one of scale $d_{n,0}$. In particular, this clarifies that, under the assumption \eqref{RW}, the main trapping mechanism is that of the large resistances that act like walls, preventing the random walk from progressing towards the right (or left), hence its maximum will stay still for a long time before jumping quickly to a new value. This phenomenon is what enables us to prove the above theorem, which is certainly not true for the usual simple symmetric random walk on $\mathbb{Z}$. We further note that the function $\theta$ above will be given using the law of the scaling limit of $(X_t)$, defined in Section \ref{sec:limits} below.

\begin{remark1}
We expect that the analogous result should be true for one-dimensional Mott variable-range hopping in the regime studied in \cite{MottModel}. Indeed, as was demonstrated in that article, the behaviour of the Mott model of \cite{MottModel} is very closely related to the nearest-neighbour random conductance model studied here, and similar arguments work in the analysis of each. The extra complication in the Mott model is that one has to show the effect of long-range jumps is asymptotically negligible.
\end{remark1}

Our second main result concerns the situation when the conductances have heavy tails at infinity.

\begin{theorem}\label{theoremSubAging0}
Under \eqref{RWT} and Assumption~\ref{AssumptionWeakBias}, for all $\alpha_0, \alpha_\infty \in (0, 1)$, the following aging statement holds. There exists an explicit function $\widetilde{\theta}:(1,\infty)\to(0,1)$ such that, for all $h>1$,
\begin{equation*}
    \lim_{n \to \infty} \widetilde{\Prob}^{\lambda/n}\left( \left| \widetilde{X}_{b_n} - \widetilde{X}_{h b_n} \right| \le 1 \right) = \widetilde{\theta}(h).
\end{equation*}
Furthermore, the following sub-aging statement holds. There exists an explicit function $\widebar{\theta}:(0, \infty)\to(0,1)$ such that, for all $h>0$,
\begin{equation*}
    \lim_{n \to \infty} \widetilde{\Prob}^{\lambda/n}\left( \left| \widetilde{X}_{b_n + s_1 d_{n, \infty}} - \widetilde{X}_{b_n + s_2 d_{n, \infty}} \right| \le 1, \,\, \forall s_1, s_2 \in [0, h] \right) = \widebar{\theta}(h).
\end{equation*}
\end{theorem}

Again, the limiting functions, $\widetilde{\theta}$ and $\widebar{\theta}$ in this case, will be made explicit in Section \ref{sec:more-results}, once we have defined the scaling limit of $\widetilde{X}$. Note that this second theorem incorporates a different way for a random walk to undergo aging. Under \eqref{RWT}, the random walk is now also trapped by large conductances (i.e.\ those of scale $d_{n,\infty}$), over which it will cross many times before escaping. Moreover, the walker will come back to the same large conductance with good probability (depending on the tail decay of the resistance distribution) many times. The aging statement in this case corresponds to the fact that, after a time of the order of the age of the system, the walker will be likely to be on a large conductance and come back to it after a multiple of that time. The sub-aging statement provides finer information: at time $b_n$, i.e.\ the age of the system, the walker is likely to be on a large conductance and to stay adjacent to it for a time of order $d_{n,\infty}$. We highlight that the tail decay of the distribution of the resistances does not affect the length of the sub-aging timescale.

Let us make two comments on our main statements above. First, we expect that, following a similar strategy to the one presented in this paper, one could recover an aging statement for the maximum of the walk under the assumption \eqref{RWT}, similar to that of Theorem \ref{theoremAgingNEW}, but at a different time-scale. We choose to present the result for the \eqref{RW} model only, as it is the edges of large resistance that capture the aging phenomenon under consideration in that result. Second, the reader may wonder why we consider the case with walls only and the case with walls and traps, but not the case with traps only: the reason is that it seems clear to us that the statements and proofs would be very similar to those for Bouchaud's trap model in \cite{FIN,CroydonMuirhead}.

\subsection{Outline of the proof}

In this section, we discuss the organisation of the paper and outline the proof of the main results stated above. As explained in the introduction, under the assumptions \eqref{RWT} or \eqref{RW}, the random walk will cross back and forth edges with atypically large conductances many times, and collide with edges of atypically large resistance (i.e.~small conductance). Accordingly, the rescaled random walk will converge towards a diffusion in random environment that will localize on some points, and whose maximum will also localize on some points. In order to study the limit of random walks amongst random conductances, it is useful to consider environments as empirical point processes of normalised conductances, or resistances (i.e.~inverse conductances), that encode the positions and the values of large conductances, or resistances. Then, under appropriate assumptions on the conductances or resistances, one can prove that the environments converge to degenerate, dense Poisson point processes on $\mathbb{R}\times\mathbb{R}^+$. The Poisson process corresponding to large conductances locates points where the limiting diffusion will localize, while the Poisson process corresponding to large resistances corresponds to points where the maximum will stagnate.

As for the organisation of the article, in Section \ref{sec:limits}, we will define the point processes and associated subordinators that will encode the limiting environments. In the same section, we define the limiting processes, which are diffusions on these limiting environments. In Section \ref{sec:more-results}, we state more results and in particular restate Theorem \ref{theoremAgingNEW} and Theorem \ref{theoremSubAging0} with refined details. In Section \ref{sec:coupling}, we prove the convergence of the empirical point processes towards the limiting Poisson point processes and give an explicit construction of a crucial coupling between the discrete empirical point processes and their limits. This coupling is used throughout the rest of the paper. Section \ref{sec:randomwalksestimates} provides technical tools and estimates for random walks that will useful in proving the main results. We prove for instance that, under \eqref{RWT}, the probability that the random walk is located on a given large conductance converges towards the probability that the limiting diffusion is located on the corresponding atom of the limiting Poisson point process. Finally, we prove the aging statements in Section \ref{sec:agingproof} and the sub-aging statement in Section \ref{sec:subagingproof}. In Section \ref{sec:reslimitprocess}, we provide some useful estimates on the limit processes. The article also contains an appendix, which contains some notes on $J_1$ convergence.

\section{Limit processes and refined statements of main results}

\subsection{Limit processes and limit environments}\label{sec:limits}

In this section, we recall the definitions of two processes $Z^\lambda$ and $\widetilde{Z}^\lambda$ that were considered in \cite{MottModel}. In particular, these processes are the scaling limits of the random walks we consider. The process $Z^\lambda$ is a (generalized) diffusion in a random environment (given by a two-sided subordinator). We enlarge our probability space so that the environment is defined under the measure $\mathbf{P}$, and write $P^{\omega, \lambda}$ for the quenched law of the process and $\Prob^{\lambda}$ for its annealed law. We do the same (adding a tilde on top of the measures) for $\widetilde{Z}^\lambda$. For the processes with vanishing bias $\lambda=0$, $\Prob^{0}$ and $\widetilde{\Prob}^{0}$, we will drop the superscript $\lambda$ in the notation. We remark that $\lambda$ is a positive parameter that is present in the limit due to Assumption~\ref{AssumptionWeakBias}. We may write $\mathbb{P}^{\lambda}(Z \in \cdot)$ in place of $\mathbb{P}^{\lambda}(Z^\lambda \in \cdot)$ to ease the notation.

Let us start by defining $Z^\lambda$. Consider a standard Brownian motion $B=(B_t)_{t \ge 0}$ (started from 0) and an independent two-sided L\'evy process $S^{\alpha_0}$ with L\'evy measure
\[	\alpha_0 x^{-1-\alpha_0}\mathds{1}_{\{x>0\}}dx.\]
Furthermore, for $\lambda>0$, define an exponentially-tilted version of the L\'evy process by setting
\begin{equation}\label{eqn:eqnTwoSidedLevy}
S^{\alpha_0, \lambda}(u) \coloneqq \int_{0}^{u} e^{-2\lambda v} dS^{\alpha_0}(v).
\end{equation}
Note that $S^{\alpha_0, 0}= S^{\alpha_0}$. Furthermore, let us define the measure $\mu^{\lambda}$, whose support is $\overline{S^{\alpha_0,\lambda}(\R)}$, i.e.~the closure of the image of the L\'evy process defined in \eqref{eqn:eqnTwoSidedLevy},  by
\begin{equation}\label{SmoothSpeedMeasure}
	\mu^{\lambda} \left( (a, b] \right) \coloneqq 2\mathbf{E}\left[c({0, 1})\right] \int_{(S^{\alpha_0,\lambda})^{-1}(a)}^{(S^{\alpha_0, \lambda})^{-1}(b)} e^{2\lambda v} dv,
\end{equation}
where $(S^{\alpha_0, \lambda})^{-1}$ denotes the right-continuous inverse of $S^{\alpha_0, \lambda}$. Writing $(L_t^B(x))_{t \ge 0, x \in \R}$ for the local times of $B$, we further define
\begin{equation}\label{eqn:eqnDefinitionH}
H^{\lambda}_t \coloneqq \inf\left\{ s \ge 0: \, \int_{\R} L_s^{B} (x) \mu^{\lambda} (dx) > t \right\}.
\end{equation}
Finally, we construct $(Z_t)_{t \ge 0}$ by setting
\begin{equation}\label{eqn:eqnDefinitionZ}
Z^\lambda_t \coloneqq (S^{\alpha_0, \lambda})^{-1} \left( B_{H^{\lambda}_t} \right).
\end{equation}
Note that, for the process $Z^\lambda$, $P^{\omega, \lambda}(\cdot) = \Prob^{\lambda}(\cdot \, | \, S^{\alpha_0, \lambda})$, we may use both notations.

The definition of $(\widetilde{Z}^\lambda_t)_{t \ge 0}$ is similar. Consider, independent of $B$ and $S^{\alpha_0}$, a two-sided L\'evy process $S^{\alpha_\infty}$ with intensity $\alpha_\infty x^{- 1 -\alpha_\infty}\mathds{1}_{\{x>0\}}dx$ and, similarly to \eqref{eqn:eqnTwoSidedLevy}, we define its tilted version
\begin{equation*}
S^{\alpha_\infty, \lambda}(u) \coloneqq \int_{0}^{u} e^{2\lambda v} dS^{\alpha_\infty}(v);
\end{equation*}
we highlight that the difference in the sign of $2\lambda v$ between the above expression and \eqref{eqn:eqnTwoSidedLevy} is intentional. We define an associated measure and time-change by supposing
\begin{equation}\label{eqn:HeavyTailedSpeedMeasure}
	\widetilde{\mu}^{\lambda} \left( (a, b] \right) \coloneqq \int_{(S^{\alpha_0, \lambda})^{-1}(a)}^{(S^{\alpha_0, \lambda})^{-1}(b)} e^{2\lambda v} dS^{\alpha_\infty}(v)\text{ and } \widetilde{H}^{\lambda}_t \coloneqq \inf\left\{ s \ge 0: \, \int_{\R} L_s^{B} (x) \widetilde{\mu}^{\lambda} (dx) > t \right\},
\end{equation}
and then set
\begin{equation}\label{eqn:eqnDefinitionZTilde}
\widetilde{Z}^\lambda_t \coloneqq (S^{\alpha_0, \lambda})^{-1} \left( B_{\widetilde{H}^{\lambda}_t} \right).
\end{equation}
Note that, for the process $\widetilde{Z}^\lambda$, $\widetilde{P}^{\omega, \lambda}(\cdot) = \widetilde{\Prob}^{\lambda}(\cdot \, | \, S^{\alpha_0, \lambda}, S^{\alpha_\infty, \lambda})$, we may use both notations.

For later purposes, it will be useful to recall a well-known representation of the subordinators considered above. In particular, let us introduce the measures
\begin{equation}\label{eqn:PoissonPointMeasures}
	 \nu^{\alpha_\infty}(dz) \coloneqq \sum_{(x, w) \in \mathcal{P}^{\alpha_\infty}} w \delta_x(dz), \quad\quad \textnormal{and} \quad\quad \nu^{\alpha_0}(dz) \coloneqq \sum_{(y, v) \in \mathcal{P}^{\alpha_0}} v \delta_y(dz),
\end{equation}
where $\mathcal{P}^{\alpha_\infty}$ is a Poisson point process on $\R \times \R_+$ with intensity $dx\alpha_\infty w^{-1-\alpha_\infty}dw$ and $\mathcal{P}^{\alpha_0}$ is a Poisson point process on $\R \times \R_+$ with intensity $dy\alpha_0v^{-1-\alpha_0}dv$, and we suppose these two Poisson processes are independent. We can then write the two-sided L\'evy processes above as
\begin{equation} \label{eqn:eqnDefAlpha0Sub}
	S^{\alpha_0, \lambda}(t) = \int_{0}^{t} e^{-2\lambda s} \nu^{\alpha_0}(ds) \,\, \textnormal{for} \,\, t \ge 0, \quad S^{\alpha_0, \lambda}(t) = -\int_{t}^{0} e^{-2\lambda s} \nu^{\alpha_0}(ds) \,\, \textnormal{for} \,\, t < 0,
\end{equation}
and
\begin{equation} \label{eqn:eqnDefAlphaInftySub}
	S^{\alpha_\infty, \lambda}(t) = \int_{0}^{t} e^{2\lambda s} \nu^{\alpha_\infty}(ds) \,\, \textnormal{for} \,\, t \ge 0, \quad S^{\alpha_\infty, \lambda}(t) = -\int_{t}^{0} e^{2\lambda s} \nu^{\alpha_\infty}(ds) \,\, \textnormal{for} \,\, t < 0.
\end{equation}
It is also convenient to introduce at this point the discrete counterparts of these subordinators. For this purpose, let us define  $R^{\lambda/n}(i, j)$ to be the \textbf{effective resistance} between indices $i$ and $j$ on $\mathbb{Z}$ in the electrical network associated with $(c^{\lambda/n}(\{i, i+1\}))_{i \in \Z}$, i.e.\ for $i<j$, we set $R^{\lambda/n}(i,i) \coloneqq 0$ and
\[R^{\lambda/n}(i, j)\equiv R^{\lambda/n}(j,i)\coloneqq \sum_{k=i}^{j-1}r^{\lambda/n}(\{k, k+1\}).\]
When $A$ and $B$ are sets of indices, we denote $R^{\lambda/n}(A,B)$ the effective resistance between two sets. As noted in the introduction, the general intuition is that the scaling limits of the random walks are impacted by both the large resistances and, in the case of \eqref{RWT}, the large conductances. Due to the heavy-tailed distributions, when observing the environment on an interval of length of order $n$, the sum of the resistances will be of the same order as the largest resistance encountered, that is $d_{n,0}$. Similarly, under \eqref{RWT}, the sum of the conductances will be of the same order as the largest conductance encountered, that is $d_{n,\infty}$. We incorporate these scaling factors into the following definitions. For the resistances, we define
\begin{equation}\label{EquationEffctiveResistance}
S^{\alpha_0, \lambda/n, (n)}(t):=\left\{
                                 \begin{array}{ll}
                                   \frac{1}{d_{n, 0}} R^{\lambda/n}\left( 0, \floor{nt} \right), & \textnormal{for} \,\, t \ge 0,\\
 - \frac{1}{d_{n, 0}} R^{\lambda/n}\left( \ceil{nt}, 0 \right), & \textnormal{for} \,\, t <0.\\
                                 \end{array}\right.
\end{equation}
Similarly, for the conductances,
\[S^{\alpha_\infty, \lambda/n, (n)}(t) :=\left\{
                                       \begin{array}{ll}
                                         \frac{1}{d_{n, \infty}} \sum_{i = 0}^{\floor{nt}-1} c^{\lambda/n}(\{ i , i + 1 \}), & \textnormal{for} \,\, t \ge 0,\\
                                         -\frac{1}{d_{n, \infty}} \sum_{i = \ceil{nt}}^{-1} c^{\lambda/n}(\{ i , i + 1 \}), & \textnormal{for} \,\, t< 0.
                                       \end{array}
                                     \right.\]

\subsection{Restatement of the main results}\label{sec:more-results}

In this section, we restate the results of Section \ref{sec:mainres} with some more detail, and also present some further statements. In particular, the results of this section include those of Section \ref{sec:mainres}.

Towards stating the first result of the section, we recall the definitions \eqref{eqn:RightScalingTerms} and \eqref{eqn:RightTimeScales} of the scaling terms and define
\begin{equation}\label{EquationDefGaps}
	\begin{split}
		\mathrm{Gap}^{\lambda}_{n}(t) &\coloneqq \frac{1}{d_{n, 0}} r^{\lambda/n}\left( \widebar{X}_{t a_n}, \widebar{X}_{t a_n} + 1 \right) = S^{\alpha_0, \lambda/n, (n)}\left(n^{-1}\left(\widebar{X}_{t a_n} + 1\right)\right) - S^{\alpha_0, \lambda/n, (n)}\left(n^{-1}\widebar{X}_{t a_n}\right), \\
		\mathrm{Gap}^{\lambda}(t) &\coloneqq S^{\alpha_0, \lambda}\left(\widebar{Z}^{\lambda}_{t} \right) - S^{\alpha_0, \lambda}\left(\widebar{Z}^{\lambda}_{t^{-}} \right).
	\end{split}
\end{equation}
The following theorem describes the scaling limit of the size of the wall seen by the maximum of the walker after time $a_n$, under the assumption \eqref{RW}.

\begin{theorem} \label{TheoremGap}
Under \eqref{RW} and Assumption~\ref{AssumptionWeakBias}, for all $0<\alpha_0<1$ it holds that, under the annealed law $\Prob^{\lambda/n}$,
	\[	\textnormal{Gap}^{\lambda}_n(1) \stackrel{(\mathrm{d})}{\to} \textnormal{Gap}^{\lambda}(1), \quad\quad \textnormal{as} \,\, n \to \infty,	\]
where $\textnormal{Gap}^\lambda(1)$ is a non-trivial random variable taking values in $(0,\infty)$.
 \end{theorem}

The following result will later be shown to be a consequence of the construction needed to prove Theorem~\ref{TheoremGap}, and it implies Theorem \ref{theoremAgingNEW}.

\begin{proposition}\label{theoremAging}
	Under \eqref{RW} and Assumption~\ref{AssumptionWeakBias}, for all $0<\alpha_0<1$, the following aging statement holds. For all $h > 1$, we have
	\begin{equation*}
		\lim_{n \to \infty} \Prob^{\lambda/n}\left( \widebar{X}_{a_n} = \widebar{X}_{[a_n,h a_n]} \right) = \theta(h) \coloneqq \Prob^\lambda \left(\widebar{Z}_{1} = \widebar{Z}_{[1, h]} \right),
	\end{equation*}
	where the right-hand side takes values in $(0,1)$.
\end{proposition}

Finally, the subsequent result implies Theorem \ref{theoremSubAging0}, providing an explicit form for the aging and sub-aging functions.

\begin{proposition}\label{theoremSubAging}
Under the hypothesis \eqref{RWT} and Assumption~\ref{AssumptionWeakBias}, for all $\alpha_0, \alpha_\infty \in (0, 1)$, the following aging statement holds. For all  $h > 1$, we have
\begin{equation*}
    \lim_{n \to \infty} \widetilde{\Prob}^{\lambda/n}\left( \left| \widetilde{X}_{b_n} - \widetilde{X}_{h b_n} \right| \le 1 \right) = \widetilde{\theta}(h) \coloneqq \widetilde{\Prob}^\lambda \left(\widetilde{Z}_{1} = \widetilde{Z}_{h} \right)  ,
\end{equation*}
where the right-hand side takes values in $(0,1)$. Furthermore, the following sub-aging statement holds. For all $h>0$, we have
\begin{equation*}
    \lim_{n \to \infty} \widetilde{\Prob}^{\lambda/n}\left( \left| \widetilde{X}_{b_n + s_1 d_{n, \infty}} - \widetilde{X}_{b_n + s_2 d_{n, \infty}} \right| \le 1, \,\, \forall s_1, s_2 \in [0, h] \right) = \widebar{\theta}(h) \coloneqq \widetilde{\EVal}^\lambda\left[ e^{- h\frac{A^0 + A^2}{2A^1}}\right],
\end{equation*}
where $A^0, A^1, A^2$ are such that $A^1 \stackrel{(\mathrm{d})}{=} \nu^{\alpha_\infty}\left(\widetilde{Z}^{\lambda}_1\right)$ and $A^0, A^2$ are distributed as independent conductances under $\widetilde{\mathbf{P}}$, not tilted, and  independent of $\widetilde{Z}^{\lambda}$.
\end{proposition}

As will become clear in the proof, the key to these conclusions is showing that, with high probability, at time $b_n$, the process $\widetilde{X}$ is in a trap whose depth is of order $d_{n,\infty}$  (i.e.\ $\widetilde{X}$ is adjacent to a conductance of this scale), and also the limiting process $\widetilde{Z}$ is in a non-trivial trap at time $1$. For the second claim in particular, the limiting expression arises from the observation that the conductance environment around the large conductance is asymptotically close (up to a multiplicative constant) in distribution to that of $(A^0, d_{n,\infty}A^1,A^2)$, from which it follows that the time to escape from the edge in question is approximately exponential with mean $2d_{n,\infty}A^1/(A^0+A^2)$.

\section{Coupling and convergence of the environment}\label{sec:coupling}

The goal of this section is to prove in Proposition \ref{PropositionPoissonPoint} (see also Propositions \ref{PropositionRWEnvConvergence} and \ref{PropositionRWTEnvConvergence}) that the environment, under an explicit coupling, converges to its limiting counterpart in a precise sense. Before stating the main result of this section let us recall some useful notions of convergence for measures.

\subsection{Convergence of point processes}

In this section, we recall notions of convergence of measures.  (For further background, see \cite[Section 2]{BenArousCerny2AgingRegimes}, for example.) Let $\mathcal{M}$ denote the family of locally finite Borel measures on $\R$.

\begin{definition1}
	Consider  $\nu\in \mathcal{M}$ and a family $(\nu^{(n)}; n \in \N)$ in $\mathcal{M}$. We say that $\nu^{(n)}$ converges vaguely to $\nu$, and write $\nu^{(n)} \stackrel{v}{\to} \nu$ as $n \to \infty$, if for all continuous real-valued functions $f$ on $\R$ with bounded support
\[	\int_{\R} f(y) \nu^{(n)}(dy) \to \int_{\R} f(y) \nu(dy), \quad \textnormal{as} \,\, n \to \infty.\]
\end{definition1}

\begin{definition1}
Consider  $\nu\in \mathcal{M}$ and a family $(\nu^{(n)}; n \in \N)$ in $\mathcal{M}$. We say that $\nu^{(n)}$ converges in point-process sense to $\nu$, and write $\nu^{(n)} \stackrel{pp}{\to} \nu$ as $n \to \infty$, if the following holds. If the atoms of $\nu$ and $\nu^{(n)}$ are, respectively, at locations $y_i$ and $y^{(n)}_{i}$ in $\R$ with weights $w_i$ and $w^{(n)}_{i}$ in $(0, \infty)$, then the set $V^{(n)} \coloneqq \bigcup_{i} \{ (y^{(n)}_{i}, w^{(n)}_{i}) \}$ converges to the set $V \coloneqq \bigcup_{i} \{ (y_{i}, w_{i}) \}$ in the following sense: for any open set $U \subset \R\times (0, \infty)$ whose closure is a compact subset of $\R\times (0, \infty)$ and is such that the boundary does not contain any point of $V$, the number of points $|U \cap V^{(n)}|$ is finite and equals the number of points $|U \cap V|$ for all $n$ large enough.
\end{definition1}

Furthermore, we introduce a condition that relates to the above two notions of convergence of measures.

\begin{condition}\label{Condition1}
Consider  $\nu\in \mathcal{M}$ with atoms $(x_\ell,w_\ell)$ and a family $(\nu^{(n)}; n \in \N)$ in $\mathcal{M}$ with atoms $(x^{(n)}_\ell,w^{(n)}_\ell)$. For each $\ell \ge 0$ there exists a sequence $j_\ell(n)$ such that
\[\left(x^{(n)}_{j_\ell(n)}, w^{(n)}_{j_\ell(n)} \right) \to \left( x_\ell, w_\ell \right), \quad \textnormal{as} \,\, n \to \infty.\]
\end{condition}

\begin{lemma}\cite[Proposition 2.1]{FIN}\label{LemmaCond1PointCv}
Consider  $\nu\in \mathcal{M}$ and a family $(\nu^{(n)}; n \in \N)$ in $\mathcal{M}$. If $\nu^{(n)} \stackrel{pp}{\to} \nu$ as $n \to \infty$, then Condition~\ref{Condition1} holds. If Condition~\ref{Condition1} holds and $\nu^{(n)} \stackrel{v}{\to} \nu$, then $\nu^{(n)} \stackrel{pp}{\to} \nu$.
\end{lemma}

\subsection{Coupling and convergence of the discrete environment}\label{SectionCopuledSpaces}

The goal of this section is to prove the convergence of the environment we consider in this paper. Recall the hypotheses \eqref{RW} and \eqref{RWT} given in Assumption \ref{AssumptionTail}. Recall also the definitions \eqref{eqn:PoissonPointMeasures} of the independent measures $\nu^{\alpha_\infty}$ and $\nu^{\alpha_0}$.

We will see the discrete environment as the superposition of two empirical measures. For this purpose, for fixed $K\in\mathbb{N}$, let us define the measures
\begin{equation}\label{eqn:eqn2DiscetePointMeasures}
	\begin{split}
	\nu^{\alpha_\infty, (n)} &\coloneqq \frac{1}{d_{n, \infty}} \sum_{x \in \Z, \, |x| \le Kn} \delta_{x/n} c\left(\{x, x + 1\}\right),\\
	\nu^{\alpha_0, (n)} &\coloneqq \frac{1}{d_{n, 0}} \sum_{x \in \Z, \, |x| \le Kn} \delta_{x/n} r\left(\{x, x + 1\}\right).
	\end{split}
\end{equation}
Note that we chose to not emphasize the dependence on $K$ in the notation. The result below holds for all $K\in\mathbb{N}$.

\begin{proposition}(Vague and point-process convergence of the environment).\label{PropositionPoissonPoint} First, under the assumption \eqref{RW}, there exists an explicit coupling under which $\nu^{\alpha_0, (n)}$ converges almost surely, in both the vague and the point process sense, to $\nu^{\alpha_0}$ restricted to $[-K, K]$ and, moreover, Condition \ref{Condition1} is satisfied.

Second, under the assumption \eqref{RWT}, there exists an explicit coupling under which $\nu^{\alpha_\infty, (n)}$ and $\nu^{\alpha_0, (n)}$ converge almost surely, in both the vague and the point process sense, to the independent measures $\nu^{\alpha_\infty}$ and $\nu^{\alpha_0}$ restricted to $[-K, K]$ and, moreover Condition \ref{Condition1} is satisfied by both sequences.
\end{proposition}

In order to prove this result we will use a coupling technique developed in \cite{FIN} and further used in \cite{BenArousCerny2AgingRegimes}. Additionally to their strategy, when we work under the assumption \eqref{RWT},  we need to take care of the dependence between the large conductances and the large  resistances, so as to show that the measures described in \eqref{eqn:eqn2DiscetePointMeasures} are asymptotically independent. We will detail the coupling only for $\nu^{\alpha_\infty,(n)}$  under the assumption \eqref{RWT}, and justify the asymptotic independence of $\nu^{\alpha_\infty,(n)}$ and $\nu^{\alpha_0,(n)}$. The coupling for $\nu^{\alpha_0,(n)}$ under the assumption \eqref{RWT} or \eqref{RW} follows from similar arguments.

Let us next present a result that is key to justifying the asymptotic independence of $\nu^{\alpha_\infty,(n)}$ and $\nu^{\alpha_0,(n)}$. To do this, we need the following notation: for any $0< \widehat{\delta} < 1$ (to be chosen later),
\begin{itemize}
    \item the set of $n$-walls $\mathcal{J}_n^{\alpha_0} \coloneqq \left\{ j \in \Z: r(\{j, j+1\}) > d_{n, 0}^{1 - \widehat{\delta}}  \right\}$;
    \item the set of $n$-traps $\mathcal{J}_n^{\alpha_\infty} \coloneqq \left\{ j \in \Z: c(\{j, j+1\}) > d_{n, \infty}^{1 - \widehat{\delta}}\right\}$.
\end{itemize}
The following lemma states that, under \eqref{RWT}, these two sets are well-separated with high probability. By a simpler argument, a similar result holds for the set of $n$-walls under \eqref{RW}.

\begin{lemma}\label{LemmaGoodSeparationIndep} Assume \eqref{RWT}. Let us consider the sets $T^{\alpha_0}_n = (\mathcal{J}_n^{\alpha_0}) \cap [-Kn, Kn]$ and $T^{\alpha_\infty}_n = (\mathcal{J}_n^{\alpha_\infty}) \cap [-Kn, Kn]$. Define the event
\begin{equation}\label{def:Tn}
\mathcal{T}_n \coloneqq \left\{ |i - j| > n^{1/4} \,\, \textnormal{for all distinct} \, i, j \in T^{\alpha_0}_n \cup T^{\alpha_\infty}_n \right\}.
\end{equation}
Then, for all $\widehat{\delta} = \widehat{\delta}(\alpha_0, \alpha_\infty)$ small enough, almost surely there exists $n_0 = n_0(\omega,K, \widehat{\delta}) >0$ such that $\mathcal{T}_n$ occurs for all $n\ge n_0$.
\end{lemma}

\begin{proof} Let us start by noticing that
	\[\left\{ |i - j| > n^{1/4} \,\, \textnormal{for all distinct} \, i, j \in T^{\alpha_0}_n \cup T^{\alpha_\infty}_n \right\}^c  \subseteq \bigcup_{m = -Kn}^{Kn} A(m) \cap \{m  \in T^{\alpha_0}_n \cup T^{\alpha_\infty}_n\},\]
	where
	\[	A(m) \coloneqq \left\{ \exists j \in \{ m - n^{1/4} , \dots, m + n^{1/4} \} \backslash \{m\}  \,\, \text{such that } j \in T^{\alpha_0}_n \cup T^{\alpha_\infty}_n \right\}.\]
Using the fact that slowly varying functions grow slower than any polynomial asymptotically, we have that, for all $\varepsilon>0$ and for $n$ large enough,
\begin{equation}\label{eq:slowly}
n^{\tfrac{1}{\alpha_0}-\varepsilon}\le d_{n,0}\le n^{\tfrac{1}{\alpha_0}+\varepsilon} \text{ and }n^{\tfrac{1}{\alpha_\infty}-\varepsilon}\le d_{n,\infty}\le n^{\tfrac{1}{\alpha_\infty}+\varepsilon}.
\end{equation}
Using the previous estimates, one can prove that, for $n$ large enough,
	\begin{equation*}
	\begin{split}
	\widetilde{\mathbf{P}}\left(r(\{j,j+1\})>d_{n,0}^{1-\widehat{\delta}}\right)&\le n^{-1+3\widehat{\delta}}\text{ and }
	\widetilde{\mathbf{P}}\left(c(\{j,j+1\})>d_{n,\infty}^{1-\widehat{\delta}}\right)\le n^{-1+3\widehat{\delta}}.
	\end{split}
	\end{equation*}
Using the above, the independence of $A(m)$ and  $\{m \in {T}^{\alpha_0}_n\cup {T}^{\alpha_\infty}_n\}$, and a union bound, we obtain that, if $\widehat{\delta}$ is chosen suitably small, then, for all $n$ large enough,
\begin{equation*}
	\widetilde{\mathbf{P}}\left( \left\{ |i - j| > b_n \,\, \textnormal{for all distinct} \, i, j \in T^{\alpha_0}_n \cup T^{\alpha_\infty}_n \right\}^c \right) \le 3Kn\cdot 3n^{1/4}\cdot  2n^{-2+ 6 \widehat{\delta}} \le  n^{-1/2}.
\end{equation*}

To complete the proof, we need to improve this result to an almost sure one. We take inspiration from \cite[Appendix C]{BergerSalvi2020}. Let us define the following event
\begin{equation*}
    \widebar{\mathcal{T}_{\ell}}^{c} \coloneqq \bigcup_{m = -2K\ell}^{2K\ell} \widebar{A}(m) \cap  \{m  \in \widebar{T}^{\alpha_0}_{\ell} \cup \widebar{T}^{\alpha_\infty}_{\ell}\},
\end{equation*}
where $\widebar{T}^{\alpha_i}_{\ell} = (\mathcal{J}_\ell^{\alpha_i}) \cap [-2K\ell, 2K\ell]$, $i=0,\infty$, and
\begin{equation*}
	\widebar{A}(m) \coloneqq \left\{ \exists j \in \{ m - 2\ell^{1/4} , \dots, m + 2\ell^{1/4} \} \backslash \{m\}  \,\, \text{such that } j \in \widebar{T}^{\alpha_0}_\ell \cup \widebar{T}^{\alpha_\infty}_\ell \right\}.
\end{equation*}
It is crucial to note that $\forall n \in\{ \ell, \dots, 2\ell\}$, it holds that $\mathcal{T}_{n}^{c} \subseteq \widebar{\mathcal{T}_{\ell}}^{c}$. Moreover, let us consider the subsequence $n_{\ell} =\exp\{(\log(\ell))^2\}$. Arguing as in the first paragraph of the proof, we then have that
\[\sum_{\ell} \widetilde{\mathbf{P}}(\widebar{\mathcal{T}}_{n_{\ell}}^{c})\leq \sum_{\ell} Cn_\ell\cdot n_\ell^{1/4}\cdot n_\ell^{-2+ 6 \widehat{\delta}}\leq  \sum_{\ell} Cn_\ell^{-1/2} < \infty.\]
Thus, by Borel-Cantelli, there almost-surely exists $\ell_0$ such that $\widebar{\mathcal{T}}_{n_{\ell}}^{c}$ does not happen for all $\ell \ge \ell_0$. Furthermore, we observe that $\lim_{\ell} n_{\ell + 1} / n_{\ell} = 1$, so there exists $\ell_1$ such that, for all $\ell \ge \ell_1$, we have that $n_{\ell + 1} \le 2 n_{\ell}$. We are able to conclude by setting $\widetilde{\ell} = \max \{ \ell_0, \ell_1\}$ and noting that, for all $\ell \ge \widetilde{\ell}$, the events $\widebar{\mathcal{T}}_{n_{\ell}}$ occur and that $\widebar{\mathcal{T}}_{n_{\ell}}\subseteq\mathcal{T}_{n}$ for all $n_{\ell} \le n \le n_{\ell+1}$.
\end{proof}

From now on, we assume that the sets $\mathcal{J}_n^{\alpha_0}$ and $\mathcal{J}_n^{\alpha_\infty}$ are chosen with $\widehat{\delta}$ small enough such that $\mathcal{T}_n$ holds almost surely as in Lemma~\ref{LemmaGoodSeparationIndep}. Our next step is to build an explicit coupling measure $\widebar{\widetilde{\mathbf{P}}}$ between the limit measures $\nu^{\alpha_\infty}$ and $\nu^{\alpha_0}$ and the discrete measures $\nu^{\alpha_\infty, (n)}$, $\nu^{\alpha_0, (n)}$. Following \cite{FIN}, we will couple conditioned sequences of conductances and resistances.
 In order to do so we need several ingredients. First, let us define the quantity $p \coloneqq \widetilde{\mathbf{P}}(c(\{0, 1\}) \ge 1)$ (and, as a by-product, $(1-p) \coloneqq \widetilde{\mathbf{P}}(r(\{0, 1\}) > 1)$). Consider:
\begin{enumerate}
	\item A sequence of i.i.d.\ $\mathrm{Ber}(p)$ random variables, $\{b_i\}_{i \in \mathbb{Z}}$.
	\item Two independent two-sided stable subordinators $S^{\alpha_\infty}, S^{\alpha_0}$, that are formally defined in equations~\eqref{eqn:eqnDefAlpha0Sub}-\eqref{eqn:eqnDefAlphaInftySub}.
	\item Two independent sequences of i.i.d.\ random variables $\{\widehat{c}(\{x, x+1\})\}_{x\in \Z}$ and $\{\widehat{r}(\{x, x+1\})\}_{x\in \Z}$, where $\widehat{c}(\{0, 1\})$ is distributed like $c(\{0, 1\})$ conditional on $\{c(\{0, 1\}) \ge 1\}$ and $\widehat{r}(\{0, 1\})$ is distributed like $r(\{0, 1\})$ conditional on $\{r(\{0, 1\}) > 1\}$.
\end{enumerate}

Let us build the coupling in the $\alpha_\infty$-case, and note the other can be constructed in the same way. From the subordinator $S^{\alpha_\infty}$, one can define a measure $\widebar{\nu}^{\alpha_\infty}$ such that,  for all $a<b$,
\begin{equation*}
	\widebar{\nu}^{\alpha_\infty} \left( (a, b]\right) = p^{-1/\alpha_\infty}\left(S^{\alpha_\infty}(pb) - S^{\alpha_\infty}(pa)\right).
\end{equation*}
Using \cite[(5.40)]{ResnickHeavyTail}, we have that the term above has is distributed like $S^{\alpha_\infty}(b) - S^{\alpha_\infty}(a)$, i.e.~$\widebar{\nu}^{\alpha_\infty}$ has the distribution of $\nu^{\alpha_\infty}$. Let $\widebar{\mathcal{P}}^{\alpha_\infty}$ be the associated point process. Define the function $G^{\alpha_\infty} : [0, \infty) \to [0, \infty)$ through the formula
\begin{equation}\label{eqn:GFunctionCoupling}
	\widebar{\widetilde{\mathbf{P}}} \left( S^{\alpha_\infty}(1) > G_{\alpha_\infty}(y) \right) = \widetilde{\mathbf{P}} \left( \widehat{c}(\{x, x+1\}) > y \right),
\end{equation}
note that $G_{\alpha_\infty}(y)$ is well defined, non decreasing and right-continuous by the continuity of the distribution of $S^{\alpha_\infty}(1)$. Thus, one can also define its generalised right-continuous inverse $G_{\alpha_\infty}^{-1}$. Moreover, it is also possible to define the function $g^{\alpha_\infty}_n$ as
\[g^{\alpha_\infty}_n(y) \coloneqq \frac{1}{d^*_{n, \infty}} G_{\alpha_\infty}^{-1}\left(n^{1/\alpha_\infty} y \right),\]
where
\begin{equation}\label{def:d*}
d^*_{n, \infty} \coloneqq \inf \{ t > 0 : \widetilde{\mathbf{P}} \left( \widehat{c}(\{0, 1\}) > t\right) \le 1/n \}.
\end{equation}
It is not hard to check that $d^*_{n, \infty}/d_{n, \infty} \to  p^{-1/\alpha_\infty}$ as $n \to \infty$, see \cite[Proposition~2.6]{ResnickHeavyTail}. The next lemma explains how to use these objects to build a copies of $\{\widehat{c}(\{x, x+1\})\}_{x\in \Z}$ and $\{\widehat{r}(\{x, x+1\})\}_{x\in \Z}$ from the subordinators.

\begin{lemma}
	Consider the two independent families $\{\widebar{\widehat{c}}(\{x, x+1\})\}_{x\in \Z}$ and $\{\widebar{\widehat{r}}(\{x, x+1\})\}_{x\in \Z}$ defined for all $x \in \Z$ by setting
\[	\widebar{\widehat{c}}(\{x, x + 1\}) := d^*_{n, \infty} g_{n}^{\alpha_\infty} \left( S^{\alpha_\infty}\left(\frac{1}{n}(x+1) \right) - S^{\alpha_\infty}\left(\frac{1}{n}(x) \right) \right),\]
	and
	\[	\widebar{\widehat{r}}(\{x, x + 1\}) := d^*_{n, 0} g_{n}^{\alpha_0} \left( S^{\alpha_0}\left(\frac{1}{n}(x+1) \right) - S^{\alpha_0}\left(\frac{1}{n}(x) \right) \right).	\]
	These define i.i.d.\ copies of the random conditioned conductances $\{\widehat{c}(\{x, x+1\})\}_{x \in \Z}$ and associated resistances $\{\widebar{\widehat{r}}(\{x, x+1\})\}_{x\in \Z}$.
\end{lemma}
\begin{proof}
We give the proof in the $\alpha_\infty$-case; the $\alpha_0$-case follows in the same way. Using the stationarity and independence of the increments of $S^{\alpha_\infty}$, we only need to prove that $\widebar{\widetilde{\mathbf{P}}}[\widebar{\widehat{c}}(\{0, 1\}) > t] = \widetilde{\mathbf{P}}[\widehat{c}(\{0, 1\}) > t]$. By substituting one gets
	\begin{equation*}
	\begin{split}
	\widebar{\widetilde{\mathbf{P}}}\left(\widebar{\widehat{c}}(\{0, 1\}) > t\right) &= \widebar{\widetilde{\mathbf{P}}}\left(S^{\alpha_\infty}\left( \frac{1}{n} \right) > G_{\alpha_\infty}(t) n^{-1/\alpha_\infty}\right)\\
	&= \widebar{\widetilde{\mathbf{P}}}\left(S^{\alpha_\infty}\left( 1 \right) > G_{\alpha_\infty}(t) \right)\\
	&=\widetilde{\mathbf{P}}(\widehat{c}(\{0, 1\}) > t),
	\end{split}
	\end{equation*}
	where the second equality is due to the self-similarity relation of $S^{\alpha_\infty}$, and the third equality comes from \eqref{eqn:GFunctionCoupling}. This concludes the proof.
\end{proof}

We are now ready to present our explicit coupling. We start by defining the set of conductances
\[	\widebar{c}(\{x, x+1\}) :=\left\{
                           \begin{array}{ll}
                             \widebar{\widehat{c}}(\{x^*, x^*+1\}) \mathds{1}_{\{b_x = 1\}} + \widebar{\widehat{r}}(\{x-x^*, x-x^*+1\})^{-1} \mathds{1}_{\{b_x = 0\}},&\text{for }x\ge0,\\
                              \widebar{\widehat{c}}(\{x^*-1, x^*\}) \mathds{1}_{\{b_x = 1\}} + \widebar{\widehat{r}}(\{x-x^*, x-x^*+1\})^{-1} \mathds{1}_{\{b_x = 0\}},&\text{for }x<0,
                           \end{array}
                         \right.\]
where
\begin{equation}\label{def:x*}
\begin{split}
	x^* &\coloneqq \sum_{j=0}^{x-1} \mathds{1}_{\{b_j = 1\}} \text{, for }x\ge0\text{, and }
	x^* \coloneqq -\sum_{j=x+1}^{-1} \mathds{1}_{\{b_j = 1\}} \text{, for }x<0.
	\end{split}
\end{equation}
We also define the resistance $\widebar{r}(\{x, x+1\})=1/\widebar{c}(\{x, x+1\})$. The fact that $\{\widebar{c}(\{x, x+1\})\}_{x \in \Z} \stackrel{(\mathrm{d})}{=}\{ c(\{x, x+1\})\}_{x \in \Z}$ is a straightforward application of conditioning. We are now able to define the coupled version of the two measures of Proposition~\ref{PropositionPoissonPoint},
\begin{equation*}
	\widebar{\nu}^{\alpha_\infty, (n)} \coloneqq \frac{1}{d_{n, \infty}} \sum_{x \in \Z, \, |x| \le Kn} \delta_{x/n} \widebar{c}\left(\{x, x + 1\}\right), \quad \quad \widebar{\nu}^{\alpha_0, (n)} \coloneqq \frac{1}{d_{n, 0}} \sum_{x \in \Z, \, |x| \le Kn} \delta_{x/n} \widebar{r}\left(\{x, x + 1\}\right).
\end{equation*}
Before going to the proof of Proposition~\ref{PropositionPoissonPoint}, we state two lemmas from \cite{FIN} that are useful for the analysis of the coupled measures.

\begin{lemma}\cite[Lemma 3.1]{FIN}\label{LemmaFIN1}
	For any fixed $y > 0$, $g_n^{\alpha_0}(y) \to y$ and $g_n^{\alpha_\infty}(y) \to y$ as $n \to \infty$.
\end{lemma}

We note that, using the monotonicity of $g_n^{\alpha_0}$, this lemma readily implies $g_n^{\alpha_0}(y_n) \to y$ whenever $y_n\rightarrow y>0$. A similar comment applies to $g_n^{\alpha_\infty}$.

\begin{lemma}\cite[Lemma 3.2]{FIN}\label{LemmaFIN2}
	For any $\delta' > 0$, there exist positive constants $C_1, C_2, C_3$ and $C_4$ such that
	\begin{align*}
		&g_n^{\alpha_0}(x) \le C_1 x^{1 - \delta'}, \quad \quad \textnormal{for} \,\, n^{-1/\alpha_0} \le x \le 1 \,\, \textnormal{whenever} \,\, n^{-1} \le C_2, \\
		&g_n^{\alpha_\infty}(x) \le C_3 x^{1 - \delta'}, \quad \quad \textnormal{for} \,\, n^{-1/\alpha_\infty} \le x \le 1 \,\, \textnormal{whenever} \,\, n^{-1} \le C_4.
	\end{align*}
\end{lemma}

\begin{proof}[Proof of Proposition~\ref{PropositionPoissonPoint}]
We restrict ourselves for simplicity to the box $[0, 1]$; extending to $[-K, K]$ does not change the proof. Moreover, we will only detail the proof of the convergence of $\widebar{\nu}^{\alpha_\infty, (n)}$ under the assumption \eqref{RWT}, as the proof of the convergence of $\widebar{\nu}^{\alpha_0, (n)}$ under the assumption \eqref{RWT} or \eqref{RW} follows in the same manner. Concerning notation, let us set, for $x \in [0, 1]$, $N^{(n)}(x) \coloneqq \sum_{i = 0}^{\floor{xn} - 1} b_i$ and use the shorthand $N^{(n)} = N^{(n)}(1)$. Furthermore let us introduce the function $h_n(x) \coloneqq N^{(n)}(x)/n$ and its right continuous inverse $h^{-1}_n$. One can check that, for $i$ such that $b_i=1$, $h_n(i/n) = i^*/n$ and $h^{-1}_n(i^*/n)-\tfrac1n  = i/n$, where $i^*$ is defined in \eqref{def:x*}.

First we prove almost-sure vague convergence for $\widebar{\nu}^{\alpha_\infty, (n)}$. We highlight the fact that, in the following, equalities and limits will hold almost-surely by the coupling. Let us consider a bounded continuous function $f$ of compact support $I = [0, 1]$. We use the notation
\begin{equation}\label{eqn:SetLargeIncrements}
I^{(n), \alpha_\infty}_y \coloneqq \left\{ z \in \mathbb{Z}: \, b_z = 1, \, \frac{z}{n} \in I ,\, S^{\alpha_\infty}\left(\frac{z^*+1}{n} \right) - S^{\alpha_\infty}\left(\frac{z^*}{n}\right) \ge y\right\}.
\end{equation}
Then, we have
\begin{align*}
\int_I f(y) \widebar{\nu}^{\alpha_\infty, (n)}(dy) &=\frac{1}{d_{n, \infty}}\sum_{i/n \in I} f\left(\frac{i}{n}\right) \widebar{c}(\{i, i+1\}) \\ &= \frac{1}{d_{n, \infty}}\sum_{i \in I^{(n), \alpha_\infty}_0} f\left(\frac{i}{n}\right) \widebar{c}(\{i, i+1\}) + \frac{1}{d_{n, \infty}}\sum_{i/n \in I, \, b_i = 0} f\left(\frac{i}{n}\right) \widebar{c}(\{i, i+1\}).
\end{align*}
Using that $f$ is uniformly bounded, there exists a constant $C>0$ such that
\[\frac{1}{d_{n, \infty}}\sum_{i/n \in I, \, b_i = 0} f\left(\frac{i}{n}\right) \widebar{c}(\{i, i+1\}) \le C \frac{1}{d_{n, \infty}} \sum_{i/n \in I, \, b_i = 0} 1 \le C d_{n, \infty}^{-1} n \to 0,\]
as $n \to \infty$, where we used \eqref{eq:slowly} and the fact that $\alpha_\infty<1$ to deduce the convergence to zero of the upper bound. It remains to deal with the other term. We split the sum into three parts. In particular, we fix $\delta>0$, and then consider summing over the three sets $I^{(n), \alpha_\infty}_\delta$, $I^{(n), \alpha_\infty}_{n^{-1/\alpha_\infty}} \setminus I^{(n), \alpha_\infty}_{\delta}$ and $I^{(n), \alpha_\infty}_0 \setminus I^{(n), \alpha_\infty}_{n^{-1/\alpha_\infty}}$ separately. For the first term,
\begin{align*}
\frac{1}{d_{n, \infty}}\sum_{i \in I^{(n), \alpha_\infty}_\delta} f\left(\frac{i}{n}\right) \widebar{c}(\{i, i+1\}) = \frac{d^*_{n, \infty}}{d_{n, \infty}} \sum_{i \in I^{(n), \alpha_\infty}_\delta} f\left( \frac{i}{n} \right) g_{n}^{\alpha_\infty} \left( S^{\alpha_\infty}\left(\frac{i^*+1}{n} \right) - S^{\alpha_\infty}\left(\frac{i^*}{n} \right) \right),
\end{align*}
where $d^*_{n,\infty}$ is defined in \eqref{def:d*}. Notice that we can re-write the right-hand side as
\begin{flalign*}
	&\lim_{n \to \infty} \frac{d^*_{n, \infty}}{d_{n, \infty}} \sum_{i \in I^{(n), \alpha_\infty}_\delta} f\left( \frac{i}{n} \right) g_{n}^{\alpha_\infty} \left( S^{\alpha_\infty}\left(\frac{i^*+1}{n} \right) - S^{\alpha_\infty}\left(\frac{i^*}{n} \right) \right) \\
	&\quad \quad = \lim_{n \to \infty} \frac{d^*_{n, \infty}}{d_{n, \infty}} \sum_{i \in I^{(n), \alpha_\infty}_\delta} f\left(h_n^{-1}\left(\frac{i^*}{n} \right) -\frac1n \right) g_{n}^{\alpha_\infty} \left( S^{\alpha_\infty}\left(\frac{i^*+1}{n} \right) - S^{\alpha_\infty}\left(\frac{i^*}{n} \right) \right).
\end{flalign*}
By the functional law of large numbers we have that $h_n (x) \to p x$ uniformly in $[0, 1]$ and consequently $h_n^{-1} (s) \to s/p$.
By Lemma~\ref{LemmaFIN1} and using that $f$ is bounded, that $g_n$ is monotone, that $I^{(n), \alpha_\infty}_\delta$ has almost-surely finitely many terms, and that the number of atoms of $\widebar{\mathcal{P}}^{\alpha_\infty}$ such that $w_j \ge \delta$ and $x_j \in [0, p]$ is $\widebar{\widetilde{\mathbf{P}}}$-a.s.~finite (each with distinct $x_j$, and for no atoms does $w_j=\delta$), this implies that
\begin{align*}
&\lim_{n \to \infty} \frac{d^*_{n, \infty}}{d_{n, \infty}} \sum_{j \in I^{(n), \alpha_\infty, p}_\delta} f\left(h_n^{-1}\left(\frac{j}{n} \right) -\frac1n\right) g_{n}^{\alpha_\infty} \left( S^{\alpha_\infty}\left(\frac{j+1}{n} \right) - S^{\alpha_\infty}\left(\frac{j}{n} \right) \right)\\
= & p^{-1/\alpha_\infty} \sum_{\substack{(x_j, w_j) \in \widebar{\mathcal{P}}^{\alpha_\infty} :\\ w_j \ge \delta, x_j \in [0, p]}} f\left(\frac{1}{p}x_j\right) w_j,
\end{align*}
where $I^{(n), \alpha_\infty, p}_y \coloneqq \left\{ z \in \N: \, z\le N^{(n)},\, S^{\alpha_\infty}\left(\frac{z+1}{n} \right) - S^{\alpha_\infty}\left(\frac{z}{n}\right) \ge y\right\}$. Let us define, for $\delta'>0$,
\[ H_\delta \coloneqq \sum_{\substack{(x_j, w_j) \in \widebar{\mathcal{P}}^{\alpha_\infty} :\\ w_j \le \delta, x_j \in I}} w_j^{1 - \delta'}.\]
Using Lemma~\ref{LemmaFIN2}, one can prove that, for $\delta'$ small enough and a positive constant $C$,
\begin{equation}\label{eqn:SmallDiscrete}
	\begin{split}
	\limsup_{n\to\infty}	\frac{d^*_{n, \infty}}{d_{n, \infty}}\sum_{i \in I^{(n), \alpha_\infty}_{n^{-1/\alpha_\infty}} \setminus I^{(n), \alpha_\infty}_{\delta}} & \left|f\left( \frac{i}{n} \right)\right| g_{n}^{\alpha_\infty} \left( S^{\alpha_\infty}\left(\frac{i^*+1}{n} \right) - S^{\alpha_\infty}\left(\frac{i^*}{n} \right) \right) \\ &\le C \limsup_{n\to\infty}\sum_{i \in I^{(n), \alpha_\infty}_{n^{-1/\alpha_\infty}} \setminus I^{(n), \alpha_\infty}_{\delta}} \left( S^{\alpha_\infty}\left(\frac{i^*+1}{n} \right) - S^{\alpha_\infty}\left(\frac{i^*}{n} \right) \right)^{1-\delta'} \\
		&\le C H_\delta.
	\end{split}
\end{equation}
We also claim that, $\widebar{\widetilde{\mathbf{P}}}$-a.s., $H_{\delta} \to 0$. Indeed, as $H_\delta$ is  positive and monotone its almost-sure limit is well-defined and
\begin{equation}\label{eqn:SmallAtomsAreSmall}
\widebar{\widetilde{\mathbf{E}}}\left[H_\delta\right] \le \alpha_\infty |I| \int_0^\delta w^{1-\delta'} w^{-1-\alpha_\infty}dw \to 0 \quad\quad \textnormal{as} \,\, \delta \to 0,
\end{equation}
as we can choose $\delta'$ such that $\delta' + \alpha_\infty < 1$. Finally, one can notice that, for all $x \le n^{-1/\alpha_\infty}$, by monotonicity of $g_n^{\alpha_\infty}$, one gets $g_n^{\alpha_\infty}(x) \le g_n^{\alpha_\infty}(n^{-1/\alpha_\infty}) \le C/d_{n, \alpha_\infty}$ for some finite positive $C$.
Then, using \eqref{eq:slowly}, we obtain that
\begin{equation}\label{eqn:VerySmallDiscrete}
	\begin{split}
		\lim_{n \to \infty} \frac{d^*_{n, \infty}}{d_{n, \infty}} \sum_{i \in I^{(n), \alpha_\infty}_0 \setminus I^{(n), \alpha_\infty}_{n^{-1/\alpha_\infty}}} f\left( \frac{i}{n} \right) &g_{n}^{\alpha_\infty} \left( S^{\alpha_\infty}\left(\frac{i^*+1}{n} \right) - S^{\alpha_\infty}\left(\frac{i^*}{n} \right) \right) \\
		&\le \frac{C'}{d_{n, \infty}} \sum_{i \in I^{(n), \alpha_\infty}_0} 1 \le \frac{C''}{d_{n, \infty}} n \to 0 \quad\quad \textnormal{as} \,\, n \to \infty.
	\end{split}
\end{equation}
Putting everything together, we have that
\begin{equation*}
	\lim_{n \to \infty} \frac{1}{d_{n, \infty}}\sum_{i/n \in I} f\left(\frac{i}{n}\right) \widebar{c}(\{i, i+1\}) = \lim_{\delta \to 0} p^{-1/\alpha_\infty} \sum_{\substack{j : w_j \ge \delta \\ x_j \in [0, p]}} f\left(\frac{1}{p} x_j\right) w_j = \int_I f(y) \widebar{\nu}^{\alpha_\infty}(dy).
\end{equation*}
This implies almost-sure vague convergence of the coupled measures. As stated at the beginning of the proof, the argument for the $\alpha_0$-process is identical. Additionally,  the independence of the limits is guaranteed by construction.

Now let us deal with point process convergence. We aim to prove Condition~\ref{Condition1} and apply Lemma \ref{LemmaCond1PointCv}. For any atom $(x_\ell, w_\ell)$ of $\widebar{\nu}^{\alpha_\infty}$ we need to find a sequence $j_\ell(n)$ such that
\begin{equation}\label{eqn:eqnProofPointProcess}
	\frac{j_\ell(n)}{n} \to x_{\ell}, \quad\quad \textnormal{and} \quad\quad  \frac{1}{d_{n, \infty}}\widebar{c}\left( j_{\ell}(n), j_{\ell}(n) + 1 \right) \to w_\ell.
\end{equation}
Note that, using the definition of $\widebar{\nu}^{\alpha_\infty}$, for all atoms $x_\ell\in[0, 1]$ of $\widebar{\nu}^{\alpha_\infty}$, there must exist an atom $(x_\ell^*, w^{*}_{\ell})$ of ${\nu}^{\alpha_\infty}$ with $x_\ell^*\in[0, p]$ such that $x^*_\ell/p = x_\ell$ and $ p^{-1/\alpha_\infty} w^{*}_{\ell} = w_{\ell}$. Then, we pick $j^*_\ell(n)$ to be such that $nx^*_\ell \in (j^*_\ell(n), j^*_\ell(n)+1]$, so that Lemma~\ref{LemmaFIN1} and the comment below \eqref{def:d*} guarantee that
\[\frac{d^*_{n, \infty}}{d_{n, \infty}}g_{n}^{\alpha_\infty}\left( S^{\alpha_\infty}\left(\frac{j^*_\ell(n)+1}{n} \right) - S^{\alpha_\infty}\left(\frac{j^*_\ell(n)}{n} \right) \right) \to p^{-1/\alpha_\infty} w^{*}_{\ell}.\]
Clearly $j^*_\ell(n)/n \to x^*_\ell$, but we also can find an index $j_{\ell}(n)$ such that $h^{-1}_n(j^*_\ell(n)/n) = j_\ell(n)/n$, and this index is such that $j_\ell(n)/n \to x_{\ell} = x^*_\ell/p $ and $\widebar{c}\left( j_{\ell}(n), j_{\ell}(n) + 1 \right)/d_{n, \infty} \to w_\ell$.

We still have to take care of the fact that the interval $[0, N^{(n)}/n]$ may not contain all the atoms (or that it may contain too many of them). However, notice that the event $E_n \coloneqq \{ np + n^{2/3} \ge N^{(n)} \ge np - n^{2/3}\}$ will happen eventually almost-surely by a standard application of Chernoff's bound (see e.g.\ \cite[Theorem A.1.4]{Chernoff}) and the Borel-Cantelli Lemma. On $E_n$ and using that subordinators do not have jumps almost-surely at deterministic times (i.e.\ there is no atom such that $x^*_\ell = p$), we get that, for all atoms and $n$ large enough $x_\ell < N^{(n)}/n$, completing the construction of the coupling and proving the convergence stated in \eqref{eqn:eqnProofPointProcess}.
\end{proof}

The following result is a corollary of Proposition~\ref{PropositionPoissonPoint} and its proof.

\begin{lemma}\label{LemmaCorollaryEnvironment}
	Assume the notation and coupling of Section~\ref{SectionCopuledSpaces}, and let $J_1$ denote the classical Skorohod topology (see \eqref{eqn:Skorohod}), then the following statement holds almost surely:
	\[S^{\alpha_0, \lambda/n, (n)}(t) \stackrel{J_1}{\to}S^{\alpha_0, \lambda}(t).\]
	Moreover, under \eqref{RW}, we have that
	\[
	\frac{1}{2n} \sum_{i \in [-Kn, Kn]} \delta_{i/n} \widebar{c}^{\lambda/n}(i)\rightarrow \mathbf{E}\left[c(0,1)\right]e^{2\lambda v}\mathds{1}_{[-K,K]}(v)dv,
	\]
	weakly as a finite measure on $\mathbb{R}$. And, under
	\eqref{RWT}, we have that
	\[
	\frac{1}{2d_{n, \infty}} \sum_{i \in [-Kn, Kn]} \delta_{i/n} \widebar{c}^{\lambda/n}(i) \rightarrow e^{2\lambda v}dS^{\alpha_\infty}(dv),
	\]
	weakly as a finite measure on $\mathbb{R}$.
\end{lemma}

\begin{proof}
	The first statement is justified in Appendix~\ref{sec:AppendixJ_1}. The second one is an immediate consequence of the assumption \eqref{RW} and the functional law of large numbers. The third statement follows directly from the convergence of $\nu^{\alpha_\infty}$ towards $dS^{\alpha_\infty}(dv)$ stated in \eqref{eqn:eqn2DiscetePointMeasures}.
\end{proof}

The following results, in the spirit of \cite[Theorem 4.1]{MottModel}, give the convergence of the environment as a compact metric measure space under \eqref{RW} and \eqref{RWT}, respectively.

\begin{proposition}\label{PropositionRWEnvConvergence}
Consider, for $n\ge1$, $\mathcal{X}_n = [-Kn, Kn] \cap \Z$, $m^{\lambda/n}_n(a, b) = d_{n, 0}^{-1} \widebar{R}^{\lambda/n}(a, b) \coloneqq d_{n, 0}^{-1} \sum_{k=a}^{b-1}\widebar{r}^{\lambda/n}(\{k, k+1\})$, $\mu^{\lambda/n}_n(dx) = \frac{1}{2n} \sum_{i \in \mathcal{X}_n} \delta_i \widebar{c}^{\lambda/n}(i)$, and $\Phi_n (\cdot) = \frac{1}{n} (\cdot) \colon \mathcal{X}_n \to \mathbb{R}$, as well as $\mathcal{X} = \overline{S^{\alpha_0, \lambda}([-K, K])}$, $d$ the Euclidean metric, the speed measure $\mu^\lambda$ defined in \eqref{SmoothSpeedMeasure}, and $\Phi(\cdot) = (S^{\alpha_0, \lambda})^{-1}(\cdot)\colon \mathcal{X} \to \mathbb{R}$. Moreover consider a sequence $(\beta_n)$ in $\mathcal{X}_n$ such that $\lim_nn^{-1}\beta_n=\beta$, where $\beta$ is a continuity point of $S^{\alpha_0, \lambda}$. Under the hypothesis of \eqref{RW} and under the coupling of Proposition~\ref{PropositionPoissonPoint}, explicitly constructed above, the quintuplet
\begin{equation}\label{tildexn}
\left( \mathcal{X}_n, m^{\lambda/n}_n, \mu^{\lambda/n}_n, \beta_n, \Phi_n \right),
\end{equation}
converges $\widebar{\mathbf{P}}$-a.s.\ to its continuous counterpart
\begin{equation*}
\left( \mathcal{X}, d, \mu^{\lambda}, S^{\alpha_0, \lambda}(\beta), \Phi \right).
\end{equation*}
in the spatial Gromov-Hausdorff-Prohorov topology (see \cite[Section 7]{CroydonResistanceForm}).
\end{proposition}

\begin{proof}
The proof of this result is a relatively straightforward adaptation of \cite[Theorem 4.1]{MottModel}. In fact, it is slightly easier, since all the relevant spaces can be isometrically embedded into $\mathbb{R}$. Hence we will be brief with the details. First, define a map $\zeta_n:\mathcal{X}_n\rightarrow\mathbb{R}$ by setting
\begin{equation}\label{pindef}
\zeta_n(x)\coloneqq d_{n,0}^{-1}\mathrm{sign}(x)\widebar{R}^{\lambda/n}(0,x),\qquad \forall x\in \mathcal{X}_n,
\end{equation}
where we denote by $\mathrm{sign}$ the sign of $x$. (Note that it is not necessary to define $\mathrm{sign}(0)$.) Let $\mathcal{Y}_n\coloneqq\zeta_n(\mathcal{X}_n)$, so that $\zeta_n$ is a bijection from $\mathcal{X}_n$ to $\mathcal{Y}_n$, and that the quintuplet at \eqref{tildexn} is isometrically equivalent to
\begin{equation}\label{doijd}
\left( \mathcal{Y}_n, d, \mu^{\lambda/n}_n\circ\zeta_n^{-1}, \zeta_n(\beta_n), \Phi_n\circ\zeta_n^{-1}  \right),
\end{equation}
where again we use $d$ to denote the Euclidean metric (restricted to the relevant space). Consequently, to check the claim of the lemma, it suffices to show that: $\mathcal{Y}_n$ converges to $\mathcal{X}$ with respect to the usual convergence of compact subsets of $\mathbb{R}$ (with respect to the Hausdorff metric);  $\mu^{\lambda/n}_n\circ\zeta_n^{-1}$ converges to $\mu^{\lambda}$ weakly; $\zeta_n(\beta_n)$ converges to $S^{\alpha_0, \lambda}(\beta)$; and there exist correspondences $\mathcal{C}_n$ between $\mathcal{X}$ and $\mathcal{Y}_n$ (i.e.\ subsets of $\mathcal{X}\times \mathcal{Y}_n$ such that each $x\in\mathcal{X}$ is paired with at least one element $y\in\mathcal{Y}_n$, and vice versa) for which $\sup_{(x,y)\in\mathcal{C}_n}(|x-y|+| \Phi(x)-\Phi_n\circ\zeta_n^{-1}(y)|)\rightarrow 0$.

Towards checking these requirements, we start by noting that
\[\zeta_n(x)=S^{\alpha_0,\lambda/n,(n)}\circ\Phi_n(x).\]
Hence
\[\mathcal{Y}_n=S^{\alpha_0,\lambda/n,(n)}([-K,K]\cap(\mathbb{Z}/n))\left\{
  \begin{array}{l}
    \subseteq S^{\alpha_0,\lambda/n,(n)}([-K,K]),\\
    \supseteq S^{\alpha_0,\lambda/n,(n)}([-K+1/n,K]).
  \end{array}
\right.\]
Since we may assume that both $-K$ and $K$ are continuity points of $S^{\alpha_0, \lambda}$ with $\widebar{\mathbf{P}}$-probability one, it readily follows from this and the almost-sure $J_1$ convergence of $S^{\alpha_0,\lambda/n,(n)}$ to $S^{\alpha_0, \lambda}$, see Lemma~\ref{LemmaCorollaryEnvironment}, that $\mathcal{Y}_n$ converges almost-surely to $\mathcal{X}$ as compact subsets of $\mathbb{R}$. Moreover, since we have assumed that $\beta$ is a continuity point of $S^{\alpha_0, \lambda}$, it is moreover clear that $\zeta_n(\beta_n)=S^{\alpha_0,\lambda/n,(n)}(\beta_n/n)\rightarrow S^{\alpha_0, \lambda}(\beta)$. Next, for the measure convergence, we start by observing that, by Lemma~\ref{LemmaCorollaryEnvironment}
\[\mu_n^{\lambda/n}\circ\Phi_n^{-1}\rightarrow \mathbf{E}\left[c(0,1)\right]e^{-2\lambda v}\mathds{1}_{[-K,K]}(v)dv\]
weakly as finite measures on $\mathbb{R}$. Hence, again applying the $J_1$ convergence of $S^{\alpha_0,\lambda/n,(n)}$ to $S^{\alpha_0, \lambda}$, and using that the limiting measure here does not have any atoms, it follows that
\[\mu^{\lambda/n}_n\circ\zeta_n^{-1}=\mu^{\lambda/n}_n\circ\Phi_n^{-1}\circ( S^{\alpha_0,\lambda/n,(n)})^{-1}\rightarrow \mu^{\lambda}\]
weakly as finite measures on $\mathbb{R}$. Finally, to construct an appropriate correspondence, we can again use the $J_1$ convergence of $S^{\alpha_0,\lambda/n,(n)}$ to $S^{\alpha_0, \lambda}$ to proceed exactly as in the proof of \cite[(65)]{MottModel}. In particular, the construction of a suitable correspondence is given below \cite[(70)]{MottModel}. Roughly, each point $x\in \mathcal{X}$ is matched to a nearby point $y\in \mathcal{Y}_n$ (and vice versa), which can be done as a result of the Hausdorff convergence of the sets in question. Since the inverse of the limiting subordinator is continuous, it follows that we also have that $\Phi_n\circ\zeta_n^{-1}(y)=(S^{\alpha_0,\lambda/n,(n)})^{-1}(y)$ is close to $(S^{\alpha_0,\lambda})^{-1}(x)$. Since they are identical to the argument of \cite{MottModel}, we omit the details.
\end{proof}

\begin{proposition}\label{PropositionRWTEnvConvergence}
Consider  $\mathcal{X}_n = [-Kn, Kn] \cap \Z$, $m_n^{\lambda/n}(a, b) = d_{n, 0}^{-1} \widebar{R}^{\lambda/n}(a, b)$, $\widetilde{\mu}_n^{\lambda/n}(dx) = 1/(2d_{n, \infty}) \sum_{i \in \mathcal{X}_n} \delta_i \widebar{c}^{\lambda/n}(i)$, and $\Phi_n (\cdot) = 1/n (\cdot) \colon \mathcal{X}_n \to \mathbb{R}$, as well as $\mathcal{X} = \overline{S^{\alpha_0, \lambda}([-K, K])}$, $d$ the Euclidean metric, $\widetilde{\mu}^{\lambda}$ the speed measure defined in \eqref{eqn:HeavyTailedSpeedMeasure}, and $\Phi(\cdot) = (S^{\alpha_0, \lambda})^{-1}(\cdot)\colon \mathcal{X} \to \mathbb{R}$. Moreover, consider a sequence $(\beta_n)$ in $\mathcal{X}_n$ such that $\lim_nn^{-1}\beta_n=\beta$, where $\beta$ is a continuity point of $S^{\alpha_0, \lambda}$. Under the hypothesis of \eqref{RWT} and under the coupling of Proposition~\ref{PropositionPoissonPoint}, explicitly constructed above, the quintuplet
\begin{equation*}
\left( \mathcal{X}_n, m_n^{\lambda/n}, \widetilde{\mu}_n^{\lambda/n}, \beta_n, \Phi_n \right),
\end{equation*}
converges $\widebar{\widetilde{\mathbf{P}}}$-a.s.\ to its continuous counterpart
\[\left( \mathcal{X}, d, \widetilde{\mu}^\lambda, S^{\alpha_0, \lambda}(\beta), \Phi \right)\]
in the spatial Gromov-Hausdorff-Prohorov topology.
\end{proposition}
\begin{proof}
The proof is entirely similar to the one of Proposition~\ref{PropositionRWEnvConvergence}, apart from some additional care is needed to handle the measure component. In particular, to check that $\widetilde{\mu}_n^{\lambda/n}\circ\zeta_n^{-1}$ converges weakly to $\widetilde{\mu}^{\lambda}$, one can combine the convergence of $\widetilde{\mu}_n^{\lambda/n}\circ\Phi_n^{-1}$ to $e^{2\lambda v}dS^{\alpha_\infty}(dv)$, see Lemma~\ref{LemmaCorollaryEnvironment}, and the $J_1$ convergence of $S^{\alpha_0,\lambda/n,(n)}$ to $S^{\alpha_0, \lambda}$. The one subtlety in doing this is resolved by observing that, because the limiting subordinators are independent, their discontinuities are almost-surely disjoint.
\end{proof}

In all that follows we will drop the bar on top of the probability measures  $\mathbf{P}$ and $\widetilde{\mathbf{P}}$, the reader should assume the coupling to be in place from now on unless stated otherwise.

\section{Random walk estimates}\label{sec:randomwalksestimates}

\subsection{Couplings}

In the next sections we follow a classical general strategy to prove aging. That is, we exploit the coupling of Proposition~\ref{PropositionPoissonPoint} and the consequent Proposition~\ref{PropositionRWEnvConvergence} and Proposition~\ref{PropositionRWTEnvConvergence}. This guarantees almost-sure convergence of the \textbf{quenched} distribution of the process. In our proof, it is crucial that the convergence of Proposition~\ref{PropositionPoissonPoint} holds almost-surely on the coupling. Then, if one can prove the aging statement for the quenched law (that has nicer properties than the annealed one, such as the strong Markov property), the annealed aging theorems follow by applying the dominated convergence theorem. We would like to stress that, as stated in \cite[Corollary 1.10]{MottModel}, quenched convergence is true only after coupling, and not in the original probability space. Below, we work with the quintuplets defined in Proposition~\ref{PropositionRWEnvConvergence} and Proposition~\ref{PropositionRWTEnvConvergence}.

We introduce the random walks and the diffusions in the resistance space. Under Assumption~\eqref{RW}, define
\begin{equation}\label{RandomWalkResistanceSpace}
	Y^{(n)}_t \coloneqq \zeta_n\left(X_{t a_n}\right),
\end{equation}
where we recall the definition of $\zeta_n$ from \eqref{pindef}. Let  $B^{\beta}$ be a standard Brownian motion started at $S^{\alpha_0, \lambda}(\beta)$, and let $H^{\lambda, \beta}$ be the associated time-change defined as in \eqref{eqn:eqnDefinitionH}. Define a process in the resistance space by setting
\[	Y^{\lambda}_t \coloneqq  B^{\beta}_{H^{\lambda, \beta}_t}.\]
Similarly, under Assumption~\eqref{RWT},
\[	\widetilde{Y}^{(n)}_t \coloneqq \zeta_n\left(\widetilde{X}_{t b_n}\right),\]
and
\[	\widetilde{Y}^{\lambda}_t \coloneqq  B^{\beta}_{\widetilde{H}^{\lambda, \beta}_t},\]
where $\widetilde{H}^{\lambda, \beta}_t$ is defined in terms of \eqref{eqn:HeavyTailedSpeedMeasure} with $B^\beta$ instead of $B$.
In this section the reader should consider all the walks above to be built on the coupled version of the spaces. We avoid introducing specific notation to avoid unnecessary complication.

\begin{proposition}\label{prop:QuenchedWeakConvergence}
Assume \eqref{RW}. Under the coupling constructed in Section~\ref{SectionCopuledSpaces}, we have that $\mathbf{P}$-a.s., if $n^{-1}\beta_n\rightarrow \beta$, where $\beta$ is a continuity point of $S^{\alpha_0,\lambda}$, then
\[P_{\beta_n}^{\omega, \lambda/n, K}\left( \left(n^{-1} X_{t a_n} \right)_{t \ge 0} \in \cdot\right) \quad \mathrm{and} \quad P_{\beta_n}^{\omega, \lambda/n, K}\left( (Y_t^{(n)})_{t \ge 0} \in \cdot \right)	\]
	converge respectively, weakly as probability measures on $D([0, \infty), \mathbb{R})$, to the laws of $(Z^{\lambda}_t)_{t \ge 0}$ started at $\beta$ and $(Y^{\lambda}_t)_{t \ge 0}$ started at $S^{\alpha_0,\lambda}(\beta)$. Analogously, under \eqref{RWT} we have that $\widetilde{\mathbf{P}}$-a.s., if $n^{-1}\beta_n\rightarrow \beta$, where $\beta$ is a continuity point of $S^{\alpha_0,\lambda}$, then
	\[	\widetilde{P}_{\beta_n}^{\omega, \lambda/n ,K}\left( \left(n^{-1} \widetilde{X}_{t b_n} \right)_{t \ge 0} \in \cdot \right) \quad \mathrm{and} \quad \widetilde{P}_{\beta_n}^{\omega, \lambda/n, K}\left( (\widetilde{Y}_t^{(n)})_{t \ge 0} \in \cdot \right)	\]
	converge respectively, weakly as probability measures on $D([0, \infty), \mathbb{R})$, to the laws of $(\widetilde{Z}^{\lambda}_t)_{t \ge 0}$ started at $\beta$ and $(\widetilde{Y}^{\lambda}_t)_{t \ge 0}$ started at $S^{\alpha_0,\lambda}(\beta)$. In particular, both the convergence statements above hold with $\beta_n = \beta = 0$.
\end{proposition}
\begin{proof}
The results for $Z^\lambda$ and $\widetilde{Z}^\lambda$ are straightforward consequences of Propositions~\ref{PropositionRWEnvConvergence} and \ref{PropositionRWTEnvConvergence} and \cite[Theorem 7.1]{CroydonResistanceForm}. It is fundamental that $\beta$ is a continuity point of $S^{\alpha_0, \lambda}$. As for the $Y^\lambda$ statement, one can proceed along the same lines, replacing the map $\Phi_n\circ\zeta_n^{-1}$ in \eqref{doijd} with the identity map. A similar argument also gives the result for $\widetilde{Y}^\lambda$.
\end{proof}

\begin{remark1}
	Note that Proposition~\ref{prop:QuenchedWeakConvergence} and the dominated convergence theorem imply the weak convergence of the processes mentioned there under the annealed law. In particular, this confirms the predictions of \cite[Remark~1.9]{MottModel} on these scaling limits.
\end{remark1}

\subsection{Aging estimates under assumption \eqref{RW}}

Throughout this section, we work under the assumption \eqref{RW}, and under the coupling constructed in Section~\ref{SectionCopuledSpaces}.
We will prove the aging statement. We restrict the space to the box $[-K, K]$, but we will drop this in the notation for brevity's sake. The processes are started at $\beta_n = \beta = 0$ unless stated otherwise.

Let us define $\widebar{\rho}^{(n)}(t)$, respectively $\widebar{\rho}(t)$,  the quenched marginal distribution of the maximum of the process $(n^{-1} X_{s a_n})_{s \le t}$, respectively $(Z^{\lambda}_s)_{s \le t}$. By Lemma \ref{lemmaMaximumDisc}, we know that $\widebar{\rho}(t)$ is purely atomic, so that we can define  $\text{supp}(\widebar{\rho}(t)) \subseteq \text{supp}(\nu^{\alpha_0})$ to be its support, i.e.~the set of its atoms. Recall that, by  Proposition~\ref{PropositionPoissonPoint},  Condition~\ref{Condition1} holds and for any atom $(x_{\ell}, v_{\ell})$ of the measure $\nu^{\alpha_0} $, there exists  $j_{\ell}(n)\in\mathbb{Z}$ such that $x_\ell^{(n)}:=j_{\ell}(n)/n\to x_\ell$ and $v_{\ell}^{(n)} \to v_{\ell}$. Moreover, recall that, under the coupling of Section~\ref{SectionCopuledSpaces}, the {\it quenched} convergence of Proposition~\ref{prop:QuenchedWeakConvergence} holds. The aim of the subsection is to prove the following result.

\begin{proposition} \label{PropostionWallsPointProcess}
Under the coupling of Section~\ref{SectionCopuledSpaces}, we have that, for all $t \ge 0$,
	\[	\widebar{\rho}^{(n)}(t) \stackrel{n \to \infty}{\to} \widebar{\rho}(t),	\]
vaguely and in the point process sense. More precisely, for every atom $(x_{\ell}, v_{\ell})$ of $\widebar{\rho}(t)$, there exists $j_{\ell}(n) \in [0, Kn]$ such that
	\begin{equation}\label{eqn:WallsPP}
		P^{\omega, \lambda/n}\left(\widebar{X}_{a_nt} = j_{\ell}(n) \right) \stackrel{n \to \infty}{\to} P^{\omega, \lambda}\left(\widebar{Z}_t = x_\ell \right).
	\end{equation}
For the atom such that $x_{\ell} = K$, we have $j_{\ell}(n) = Kn$. For all atoms $x_\ell \neq K$ of $\widebar{\rho}$, there exists $w_\ell>0$ such that $(x_\ell,w_\ell)$ is an atom of $\nu^{\alpha_0}$. Moreover, $(j_{\ell}(n)/n, w_\ell^{(n)})$, where $w_\ell^{(n)}:=d_{n,0}^{-1}r(\{j_{\ell}(n), j_{\ell}(n) + 1\})$ is the atom of $\nu^{\alpha_0, (n)}$ that converges to $(x_\ell, w_\ell)$, as provided by Proposition~\ref{PropositionPoissonPoint}.

For $\beta_n/n \to \beta$ as $n \to \infty$, the convergence above stays valid if  $X$ is started from $\beta_n$ and $\widetilde{Z}$ from $\beta$, as long as $\beta$ is a continuity point of $S^{\alpha_0}$, i.e.\ $\beta$ is not an atom of $\nu^{\alpha_0}$.
\end{proposition}

Before proving the proposition above, we will state and prove two useful lemmas.

\begin{lemma}\label{lemmaEstimate1RWRW} For all $x_\ell$ such that $(x_\ell, w_\ell) \in \text{supp}(\widebar{\rho}(1))$
\[\lim_{\delta \to 0} \limsup_{n \to \infty} P^{\omega, \lambda/n}\left( \widebar{X}_{a_n} \in [j_{\ell}(n) + 1, j_{\ell}(n) + \delta n] \right) = 0.\]	
\end{lemma}
\begin{proof}
By our construction of the measures $\nu^{\alpha_0, (n)}$ and $\nu^{\alpha_0}$ and the fact that $S^{\alpha_0, \lambda/n, (n)}$ converges in the $J_1$ topology to $S^{\alpha_0, \lambda}$ (see Appendix~\ref{sec:AppendixJ_1}), it is true that
\begin{equation*}
	S^{\alpha_0, \lambda/n, (n)}\left(x_{\ell}^{(n)} + 1/n \right) \to S^{\alpha_0}\left(x_\ell \right), \quad \mathrm{and} \quad \limsup_{n \to \infty} S^{\alpha_0, \lambda/n, (n)}\left(x_{\ell}^{(n)} + \delta \right) \le S^{\alpha_0}\left(x_\ell + \delta \right).
\end{equation*}
If we consider $\widebar{Y}_1^{(n)}$ to be the maximum of the walk in the resistance space defined in \eqref{RandomWalkResistanceSpace}, we then obtain
\begin{align*}
	\lim_{\delta \to 0} \limsup_{n \to \infty}& \,\, P^{\omega, \lambda/n}\left( \widebar{X}_{a_n} \in [j_{\ell}(n) + 1, j_{\ell}(n) + \delta n] \right) \\
	&= \lim_{\delta \to 0} \limsup_{n \to \infty} P^{\omega, \lambda/n}\left( \widebar{Y}^{(n)}_1 \in \left[S^{\alpha_0, \lambda/n, (n)}\left(x_{\ell}^{(n)} + 1/n \right), S^{\alpha_0, \lambda/n, (n)}\left(x_{\ell}^{(n)}+ \delta \right)\right] \right) \\
	& \le  \lim_{\delta \to 0} P^{\omega, \lambda} \left(\widebar{Y}_1 \in \left[S^{\alpha_0, \lambda}\left(x_\ell \right), S^{\alpha_0, \lambda}\left(x_\ell + \delta \right)\right] \right) \\
	&= P^{\omega, \lambda}\left(\widebar{Y}_1 \equiv \widebar{ (B_{H^{\lambda}})}_1 =S^{\alpha_0, \lambda}\left(x_\ell \right) \right) \\
	&= 0,
\end{align*}
where the second equality holds by the almost-sure right-continuity of subordinators, and the third equality is due to Lemma \ref{lemmaMaximumDisc}. (For checking this final claim, it is useful to note that, almost-surely, $S^{\alpha_0,\lambda}(x_\ell)$ can not be equal to $S^{\alpha_0,\lambda}(v^{-})$ for any $v\in D$, where $D$ is the set of discontinuities of $S^{\alpha_0}$. Indeed, if $S^{\alpha_0,\lambda}(x_\ell)=S^{\alpha_0,\lambda}(v^{-})$ for such a $v$, then it must hold that $v>x_\ell$ and moreover $S^{\alpha_0,\lambda}$ is constant on $[x_\ell,v)$, which can not be the case.)
\end{proof}

Let us introduce the quantities
\begin{equation}\label{TimeOfMaxima}
	T^{Z^{\lambda}}_t \coloneqq \inf\left\{ s \le t: Z^{\lambda}_s = \widebar{Z}^{\lambda}_t \right\} \quad \mathrm{and} \quad \widebar{T}^{Z^{\lambda}}_t \coloneqq \sup\left\{ s \le t: Z^{\lambda}_s = \widebar{Z}^{\lambda}_t \right\}.
\end{equation}
One can similarly define the same quantities for the ``discrete" process $X^{(n)} \coloneqq (n^{-1} X_{t a_n})_{t\geq 0}$. In the following, we will write $\tau^X_a$ for the hitting time of a point $a$ by a process $X$.

\begin{lemma}\label{lemmaEstimate2RWRW} For all $x_\ell$ such that $(x_\ell, w_\ell) \in \text{supp}(\widebar{\rho}(1))$
	\[	\lim_{\delta \to 0} \limsup_{n \to \infty} P^{\omega, \lambda/n}\left( \widebar{X}_{a_n} \in [j_{\ell}(n) - \delta n,j_{\ell}(n) - 1] \right) = 0.	\]	
\end{lemma}
\begin{proof}
Let us rephrase the problem in the following way
\begin{align*}
	P^{\omega, \lambda/n}\left( \widebar{X}_{a_n} \in [j_{\ell}(n) - \delta n,j_{\ell}(n) - 1] \right) = P^{\omega, \lambda/n}\left( \tau^{X^{(n)}}_{x_\ell^{(n)} - \delta } \le 1, \, \tau^{X^{(n)}}_{x_\ell^{(n)}} > 1 \right).
\end{align*}
For any $\eta > 0$, we get by the law of total probability and the strong Markov property of the quenched law
\begin{align}\label{SplittingKeyQuantityRWRW}
\lefteqn{	P^{\omega, \lambda/n}\left(\tau^{X^{(n)}}_{x_{\ell}^{(n)} - \delta } \le 1, \, \tau^{X^{(n)}}_{x_\ell^{(n)}} > 1 \right) }\nonumber\\
&= P^{\omega, \lambda/n}\left( \tau^{X^{(n)}}_{x_\ell^{(n)} - \delta } < 1 - \eta, \, \tau^{X^{(n)}}_{x_\ell^{(n)}} > 1 \right) + P^{\omega, \lambda/n}\left( \tau^{X^{(n)}}_{x_\ell^{(n)} - \delta } \in [1 - \eta, 1], \, \tau^{X^{(n)}}_{x_\ell^{(n)}} > 1 \right) \nonumber \\
	&\le P^{\omega, \lambda/n}_{x_\ell^{(n)} - \delta}\left(\tau^{X^{(n)}}_{x_\ell^{(n)}} \ge \eta \right) + P_{0}^{\omega, \lambda/n}\left(\tau^{X^{(n)}}_{x_\ell^{(n)} - \delta } \in [1 - \eta, 1] \right).
\end{align}
The rest of the proof will focus on finding bounds from above for the last two quantities. Let us deal with the first one. For  $\delta' > 100 \delta$, a union bound gives that
\begin{align*}
	P^{\omega, \lambda/n}_{x_\ell^{(n)} - \delta}\left(\tau^{X^{(n)}}_{x_\ell^{(n)}} \ge \eta \right) \le P^{\omega, \lambda/n}_{x_\ell^{(n)} - \delta}\left(\tau^{X^{(n)}}_{x_\ell^{(n)} - \delta', x_\ell^{(n)}} \ge \eta \right) + P^{\omega, \lambda/n}_{x_\ell^{(n)} - \delta}\left(\tau^{X^{(n)}}_{x_\ell^{(n)}} > \tau^{X^{(n)}}_{x_\ell^{(n)} - \delta'}\right),
\end{align*}
where $\tau^X_{{a, b}} \coloneqq \min\{\tau^X_{a}, \tau^X_{b} \}$. Using a well-known electric networks formula (see, for example, \cite[Equation~(A.1)]{BergerSalvi2020}) we get, for the second term in the sum,
\begin{equation*}
	P^{\omega, \lambda/n}_{x_\ell^{(n)} - \delta} \left( \tau^X_{x_\ell^{(n)} - \delta'} < \tau^X_{x_\ell^{(n)}} \right) = \frac{R^{\lambda/n}\left(j_{\ell}(n) - \delta n, j_{\ell}(n)\right)}{R^{\lambda/n}\left(j_{\ell}(n) - \delta' n, j_{\ell}(n)\right)}.
\end{equation*}
From the almost sure convergence of the rescaled effective resistance we also get that
\begin{equation}\label{2HittingTimesBoundRWRW}
	\limsup_{n \to \infty} \frac{R^{\lambda/n}\left(j_{\ell}(n) - \delta n, j_{\ell}(n)\right)}{R^{\lambda/n}\left(j_{\ell}(n) - \delta' n, j_{\ell}(n)\right)} \le \frac{S^{\alpha_0, \lambda}(x_\ell^{-}) - S^{\alpha_0, \lambda}(x_\ell - 2\delta)}{S^{\alpha_0, \lambda}(x_\ell^{-}) - S^{\alpha_0, \lambda}(x_\ell - \delta')}.
\end{equation}
Let us now deal with the other term. Thanks to the commute time identity (\cite[Proposition~10.7]{LPBook}) and Markov's inequality, we get
\begin{align*}
	P^{\omega, \lambda/n}_{x_\ell^{(n)} - \delta}\left(\tau^{X^{(n)}}_{x_\ell^{(n)} - \delta', x_\ell^{(n)}} \ge \eta \right) &\le \frac{R^{\lambda/n}\left(j_{\ell}(n) - \delta n, \{j_{\ell}(n), j_{\ell}(n) - \delta' n \}\right)}{\eta n d_{n, 0} }  2\sum_{i = j_{\ell}(n) - \delta' n}^{j_{\ell}(n) - 1} c^{\lambda/n}(\{i, i+1\}), \\
	&\le \frac{R^{\lambda/n}\left(j_{\ell}(n) - \delta n, j_{\ell}(n)\right)}{\eta n d_{n, 0} } 2\sum_{i = j_{\ell}(n) - \delta' n}^{j_{\ell}(n) - 1} c^{\lambda/n}(\{i, i+1\}).
\end{align*}
Because we work under the coupling, by Proposition \ref{PropositionRWEnvConvergence} and using the strong law of large numbers, we obtain
\begin{align}
	\lefteqn{\limsup_{n \to \infty} \frac{R^{\lambda/n}\left(j_{\ell}(n) - \delta n, j_{\ell}(n)\right)}{\eta n d_{n, 0} } 2\sum_{i = j_{\ell}(n) - \delta' n}^{j_{\ell}(n) - 1} c^{\lambda/n}(\{i, i+1\})}\nonumber\\
& \le 4 \delta' e^{2\lambda K} \mathbf{E}[c(\{0, 1\})] \frac{S^{\alpha_0, \lambda}(x_\ell^{-}) - S^{\alpha_0, \lambda}(x_\ell - 2\delta)}{\eta}.\label{CoverTimeBoundRWRW}
\end{align}
Let us now focus on the second quantity appearing in \eqref{SplittingKeyQuantityRWRW}. Recall the definitions of \eqref{TimeOfMaxima}, we claim that
\begin{equation}\label{LastMaximumBoundRWRW}
	\limsup_{n \to \infty} P^{\omega, \lambda/n}\left(\tau^{X^{(n)}}_{x_\ell^{(n)} - \delta } \in [1-\eta, 1] \right) \le \limsup_{n \to \infty} P^{\omega, \lambda/n}\left( T^{X^{(n)}}_1 \ge 1-\eta \right) \le P^{\omega, \lambda}\left( \widebar{T}^{Z}_1 \ge 1 - 2\eta \right).
\end{equation}
We just need to justify the second inequality. Let us consider the coupled version on which $(X^{(n)}_t)_{t \in [0, 1]}$ converges $P^{\omega, \lambda/n}$ almost-surely path by path in the uniform topology towards $(Z^{\lambda}_t)_{t \in [0, 1]}$ (which is true on bounded intervals by the continuity of the limit $Z^{\lambda}$, see \cite[Lemma~5.4]{MottModel}). Note that the two probabilities above still make sense, but we can now compare the two events $\{T^{X^{(n)}}_1 \ge 1 -\eta\}$ and $\{\widebar{T}^{Z^{\lambda}}_1 \ge 1 - 2\eta\}$ on the same probability space. In particular, we want to show that, for all $n$ large enough
\begin{equation}\label{aim}
	\left\{ T^{X^{(n)}}_1 \ge 1 -\eta \right\} \subseteq \left\{ \widebar{T}^{Z^{\lambda}}_1 \ge 1 - 2\eta \right\}.
\end{equation}
One can prove this statement by showing that $\{ \widebar{T}^{Z^{\lambda}}_1 < 1 - 2\eta \} \subseteq \{ T^{X^{(n)}}_1 < 1 -\eta \}$ for all $n$ large enough. Let us assume $\widebar{T}^{Z^{\lambda}}_1 < 1 - 2\eta $, then by the continuity of the process $Z^{\lambda}$, we get
\begin{equation*}
	\inf_{t \in [1 - 2\eta, 1]} \left|Z^{\lambda}_t - \widebar{Z}^{\lambda}_1\right| > 0.
\end{equation*}
Thanks to the uniform convergence of $X^{(n)}$ towards $Z^{\lambda}$ we also get that
\begin{equation*}
	\lim_{n \to \infty} \inf_{t \in [1 - \eta, 1]} \left|X^{(n)}_t - \widebar{Z}^{\lambda}_1\right| > 0 \quad \mathrm{and} \quad \lim_{n \to \infty} \inf_{t \in [0, 1]} \left|X^{(n)}_t - \widebar{Z}^{\lambda}_1\right| = 0.
\end{equation*}
But then we can always choose $n$ large enough such that $T^{X^{(n)}}_1 < 1 -\eta$, since we know that at time $T^{X^{(n)}}_1$ the process $X^{(n)}$ is close to $\widebar{Z}^{\lambda}_1$. Hence we have shown that \eqref{aim} holds for large $n$.

We can now finish the proof of the lemma by noticing that the limit superior (as $n \to \infty$) of the two quantities of \eqref{SplittingKeyQuantityRWRW} is bounded, for all $\eta > 0$, by the sum of the quantities of \eqref{2HittingTimesBoundRWRW}, \eqref{CoverTimeBoundRWRW} and \eqref{LastMaximumBoundRWRW}. By continuity of probability and Lemma \ref{lemmaZNotAtMaximum}, we get that
\begin{equation*}
	\lim_{\eta \to 0} P^{\omega, \lambda}\left( \widebar{T}^{Z}_1 \ge 1 - 2\eta \right) = P^{\omega, \lambda}\left( \widebar{T}^{Z}_1 = 1 \right) \le P^{\omega, \lambda}\left( Z_1 = \widebar{Z}_1 \right) = 0.
\end{equation*}
Consequently, for all $\varepsilon>0$ we can choose $\eta^* > 0$ such that $P^{\omega, \lambda}\left( \widebar{T}^{Z}_1 \ge 1 - 2\eta^* \right) \le \varepsilon/3$. Fixing this quantity and using the existence of the left limits of the process $S^{\alpha_0}$, we can find $\delta^*$ small enough (we can fix $\delta'$ to any finite value) such that \eqref{2HittingTimesBoundRWRW} and \eqref{CoverTimeBoundRWRW} are respectively less than $\varepsilon/3$. Overall, we have proved that, for all $\varepsilon>0$, there exists $\delta^*>0$ such that, for all $\delta \le \delta^*$,
\begin{equation*}
	\limsup_{n \to \infty} P^{\omega, \lambda/n}\left(\widebar{X}_{a_n} \in [j_{\ell}(n) - \delta n, j_{\ell}(n) - 1] \right) \le \varepsilon,
\end{equation*}
which is enough to conclude the proof.
\end{proof}

We are now ready to prove Proposition \ref{PropostionWallsPointProcess}.

\begin{proof}[Proof of Proposition \ref{PropostionWallsPointProcess}]
We will prove the result for $t=1$ for notational simplicity, but the same proof holds for general $t>0$. Using the $J_1$-convergence of processes, the continuity of the limiting process and \cite[Theorem 13.4.1]{whitt}, we have that the quenched distribution of the maximum converges in $J_1$ to the quenched distribution of the maximum, i.e.\
\begin{equation*}
	\left(n^{-1} \widebar{X}_{t a_n} \right)_{t \in [0, 1]} \stackrel{(\mathrm{d})}{\to} (\widebar{Z}^{\lambda}_t)_{t \in [0, 1]}.
\end{equation*}
This implies the vague convergence statement (it is actually stronger since it involves the whole process and not just the marginal). So we only need to prove Condition~\ref{Condition1} for $\widebar{\rho}^{(n)}(1)$, $n\ge1$, and $\widebar{\rho}(1)$. Let us fix any atom $x_\ell \in \text{supp}(\widebar{\rho}(1))$ (including the special atom $x_\ell = K$ in the analysis), we need to show that there exists $j_\ell^{(n)}/n \to x_\ell$ such that
\begin{equation*}
   	P^{\omega, \lambda/n}\left(n^{-1} \widebar{X}_{a_n} = j_\ell^{(n)}/n \right) \stackrel{n \to \infty}{\to} P^{\omega, \lambda}\left(\widebar{Z}_1 = x_\ell \right).
\end{equation*}
We claimed in the statement that the only good candidate for $j_\ell^{(n)}$ is the index such that $j_\ell^{(n)}/n \to x_\ell$ and $ d_{n, 0}^{-1}r(\{j_\ell^{(n)}, j_\ell^{(n)} +1 \}) \to \nu^{\alpha_0}(x_{\ell})$. That a point satisfying these conditions exists is guaranteed by the almost-sure point process convergence of $\nu^{(n), \alpha_0}$ towards $\nu^{\alpha_0}$, and by the fact that $\widebar{\rho}(1)$ is absolutely continuous with respect to $\nu^{\alpha_0}$, by  {Lemma~\ref{lemmaMaximumDisc}}. Note that it is immediate to get that $d_{n, 0}^{-1}r^{\lambda/n}(\{j_\ell^{(n)}, j_\ell^{(n)} +1 \})$ also converges to $\nu^{\alpha_0}(x_{\ell})e^{-\lambda x_\ell}$.

Let us recall the notation $x_{\ell}^{(n)} = j_{\ell}(n)/n$ and note that
\begin{equation*}
   	\limsup_{n \to \infty} P^{\omega, \lambda/n}\left(n^{-1} \widebar{X}_{a_n} = x_{\ell}^{(n)} \right) \le P^{\omega, \lambda}\left(\widebar{Z}_1 = x_{\ell} \right),
\end{equation*}
otherwise the vague convergence statement would be violated. We are left with the task of proving that
\begin{equation}\label{LimiInfMaxProb}
	\liminf_{n \to \infty} P^{\omega, \lambda/n}\left(n^{-1} \widebar{X}_{a_n} = x_\ell^{(n)} \right) \ge P^{\omega, \lambda}\left(\widebar{Z}_1 = x_\ell \right).
\end{equation}
For convenience, let us fix the notation $X^{(n)}_t \coloneqq n^{-1} X_{t a_n}$. By convergence in distribution in the $J_1$ topology we get that
\begin{align*}
	\lefteqn{P^{\omega, \lambda}\left(\widebar{Z}_1 = x_\ell \right)}\\ =& \lim_{\delta\to 0} P^{\omega, \lambda}\left(\widebar{Z}_1 \in [x_\ell - \delta, x_\ell] \right) \nonumber \\
	 \le& \lim_{\delta\to 0} \liminf_{n \to \infty} P^{\omega, \lambda/n}\left( \widebar{X}^{(n)}_1 \in [x^{(n)}_\ell - \delta, x^{(n)}_\ell + \delta] \right) \nonumber \\
	\le&  \liminf_{n \to \infty} P^{\omega, \lambda/n}\left(\widebar{X}^{(n)}_1 = x_\ell^{(n)} \right)  \nonumber \\
	&+ \lim_{\delta\to 0} \limsup_{n \to \infty} \left\{ P^{\omega, \lambda/n} \left(\widebar{X}^{(n)}_1 \in \left[x^{(n)}_\ell + 1/n, x^{(n)}_\ell + \delta\right] \right)+P^{\omega, \lambda/n} \left( \widebar{X}^{(n)}_1  \in \left[x^{(n)}_\ell - \delta, x^{(n)}_\ell\right) \right) \right\}\\
	=& \liminf_{n \to \infty} P^{\omega, \lambda/n}\left(\widebar{X}^{(n)}_1 = x_\ell^{(n)} \right),
\end{align*}
where we applied Lemma~\ref{lemmaEstimate1RWRW} and Lemma~\ref{lemmaEstimate2RWRW} in the last equality. This completes the proof of \eqref{LimiInfMaxProb}, and the extension to a generic starting point $\beta_n$ is straightforward. Thus we conclude the proof of the proposition.
\end{proof}

\subsection{Sub-aging estimates under assumption \eqref{RWT}}

Throughout this section, we work under the assumption \eqref{RWT}, and under the coupling constructed in Section~\ref{SectionCopuledSpaces}. We will prove the sub-aging statement. We restrict the space to the box $[-K, K]$, but we will drop this in the notation for simplicity. The processes are started at $\beta_n = \beta = 0$ unless stated otherwise. Let us recall the definition of the following quantity for all $i \in \Z$
\begin{equation}\label{InvariantMeasure}
    c^{\lambda/n}(i) \coloneqq c^{\lambda/n}(\{i, i-1 \}) + c^{\lambda/n}(\{i, i+1 \}),
\end{equation}
which is a quenched invariant measure associated with the random walk $\widetilde{X}$ under $P^{\omega, \lambda/n}$.

\begin{proposition}\label{PropositionMarginalWithTraps}
	For a fixed $t>0$, let $\widetilde{\rho}^{(n)}(t)$ be the quenched marginal distribution of $n^{-1} \widetilde{X}_{t b_n}$ and let $\widetilde{\rho}(t)$ denote the one of $\widetilde{Z}_{t}$, then on the coupling described in Section~\ref{SectionCopuledSpaces} we have that, for any fixed $t \ge 0$,
	\[
	\widetilde{\rho}^{(n)}(t) \stackrel{n \to \infty}{\to} \widetilde{\rho}(t),
	\]
	vaguely. Moreover, for any atom in $(x_\ell, v_{\ell}) \in \text{supp}(\widetilde{\rho}(t))$ we have that there exists $ j_{\ell}(n)\in\mathbb{Z}$ such that
	\begin{equation*}
		\widetilde{P}^{\omega, \lambda/n}\left(\widetilde{X}_{t b_n} \in \{ j_{\ell}(n), j_{\ell}(n) + 1 \} \right) \to \widetilde{P}^{\omega, \lambda}\left( \widetilde{Z}_{t} \in x_\ell \right).
	\end{equation*}
More precisely, there exists $w_\ell$ such that $(x_\ell,w_\ell)$ is an atom of  $\nu^{\alpha_\infty}$ and $j_{\ell}(n)$ is the index such that $j_{\ell}(n)/n \to x_{\ell}$ and $d_{n,\infty}^{-1}c(\{j_{\ell}(n), j_{\ell}(n) + 1\}) = w_\ell^{(n)} \to w_\ell$, where $(j_{\ell}(n)/n, w_\ell^{(n)})$ is the atom of $\nu^{\alpha_\infty, (n)}$ that converges to $(x_\ell, w_\ell)$, as provided by Proposition~\ref{PropositionPoissonPoint}.

For $\beta_n/n \to \beta$ as $n \to \infty$, the convergence above stays valid if  $\widetilde{X}^{(n)}$ is started from $\beta_n/n $ and $\widetilde{Z}$ from $\beta$, as long as $\beta$ is a continuity point of $S^{\alpha_0}$, i.e.\ $\beta$ is not an atom of $\nu^{\alpha_0}$.
\end{proposition}

To prove the proposition above, we will need the following result. In the following we denote $x_\ell^{(n)} \coloneqq j_{\ell}(n)/n$, where $j_{\ell}(n)$ is in the sense of the statement of Proposition~\ref{PropositionMarginalWithTraps}.

\begin{lemma}\label{Lemma:ConditionalProbHitting}
Let $(x_\ell, w_\ell)$ be an atom of $\nu^{\alpha_\infty}$ and $(x_\ell^{(n)},w_\ell^{(n)})$ be the sequence converging to $(x_\ell, w_\ell)$, as provided by Proposition~\ref{PropositionPoissonPoint}. There exists a sequence $\eta=\eta(\delta) \to 0$ as $\delta \to 0$ such that
	\begin{equation*}
		\lim_{\substack{\delta \to 0 \\ \eta(\delta) \to 0}} \liminf_{n \to \infty} \widetilde{P}^{\omega, \lambda/n}\left(\widetilde{X}^{(n)}_1 \in \left\{x_\ell^{(n)}, x_\ell^{(n)} + 1/n\right\} \Big| \big| \widetilde{X}^{(n)}_{1 - \eta} - x_{\ell}^{(n)} \big| < \delta \right) = 1.
	\end{equation*}
\end{lemma}

We will prove the above lemma at the end of this section. We first proceed to the proof of the proposition.

\begin{proof}[Proof of Proposition \ref{PropositionMarginalWithTraps}] We prove the statement for $t=1$, but the proof is identical for arbitrary $t\geq0$. First, the fact that the support of $\widetilde{\rho}(1)$ is a subset of $\nu^{\alpha_\infty}$ (for the first coordinate) is a consequence of Lemma \ref{lem73}. Hence, let us fix an atom $(x_\ell, w_\ell)$ of $\nu^{\alpha_\infty}$, and let $(x_\ell^{(n)},w_\ell^{(n)})$ be the sequence converging to $(x_\ell, w_\ell)$, as provided by Proposition~\ref{PropositionPoissonPoint}. To shorten the notation, let us write $\widetilde{X}^{(n)}_1 \coloneqq n^{-1} \widetilde{X}_{b_n}$. Let us start by recalling that vague convergence follows from $J_1$ process convergence in distribution and the continuity of the limiting process. Moreover, vague convergence implies that
\begin{equation*}
	\limsup_{n \to \infty} \widetilde{P}^{\omega, \lambda/n}\left(\widetilde{X}^{(n)}_1\in \left\{x_\ell^{(n)}, x_\ell^{(n)} + 1/n\right\} \right) \le \widetilde{P}^{\omega, \lambda}\left(\widetilde{Z}_1 = x_\ell \right).
\end{equation*}
We are left with the task of proving that if we fix any atom $(x_\ell, v_\ell)\in \text{supp}(\widetilde{\rho}(t))$, then
\[	\liminf_{n \to \infty} \widetilde{P}^{\omega, \lambda/n}\left(\widetilde{X}^{(n)}_1 \in \left\{x_\ell^{(n)}, x_\ell^{(n)} + 1/n\right\} \right) \ge \widetilde{P}^{\omega, \lambda}\left(\widetilde{Z}_1 = x_\ell \right).\]
We notice that, for all $\delta, \eta > 0$,
\begin{align*}
	\widetilde{P}^{\omega, \lambda/n}&\left(\widetilde{X}^{(n)}_1 \in \left\{x_\ell^{(n)}, x_\ell^{(n)} + 1/n\right\} \right)\\
	&\ge \widetilde{P}^{\omega, \lambda/n}\left(\widetilde{X}^{(n)}_1 \in \left\{x_\ell^{(n)}, x_\ell^{(n)} + 1/n\right\} \Big| \big| \widetilde{X}^{(n)}_{1 - \eta} - x_{\ell}^{(n)} \big| < \delta \right) \widetilde{P}^{\omega, \lambda/n}\left(\big| \widetilde{X}^{(n)}_{1 - \eta} - x_{\ell}^{(n)} \big| < \delta \right),
\end{align*}
so that
\begin{align}
	\lefteqn{\liminf_{n \to \infty} \widetilde{P}^{\omega, \lambda/n}\left(\widetilde{X}^{(n)}_1 \in \{x_\ell^{(n)}, x_\ell^{(n)} + 1/n\} \right) \ge} \nonumber \\ & \lim_{\substack{\delta \to 0 \\ \eta(\delta) \to 0}} \liminf_{n \to \infty} \widetilde{P}^{\omega, \lambda/n}\left(\widetilde{X}^{(n)}_1 \in \{x_\ell^{(n)}, x_\ell^{(n)} + 1/n\} \Big| \big| \widetilde{X}^{(n)}_{1 - \eta} - x_{\ell}^{(n)} \big| < \delta \right) \widetilde{P}^{\omega, \lambda/n}\left(\big| \widetilde{X}^{(n)}_{1 - \eta} - x_{\ell}^{(n)} \big| < \delta \right), \label{eqn:FirstBoundBelowRWRWT}
\end{align}
for all possible $\delta \to 0$ and $\eta(\delta) \to 0$. Let us focus on the second term in \eqref{eqn:FirstBoundBelowRWRWT}. We observe that
\begin{align*}
	\lefteqn{\left\{ \left| \widetilde{X}^{(n)}_{1 - \eta} - x_{\ell}^{(n)} \right| < \delta \right\} \supseteq }\\
&\qquad\left\{ \left|\widetilde{X}^{(n)}_{1-\eta} - \widetilde{Z}^{\lambda}_{1-\eta}\right| < \frac{\delta}{3} \right\} \cap \left\{ \left|x_\ell - x_\ell^{(n)}\right| < \frac{\delta}{3} \right\} \cap \left\{ \left|\widetilde{Z}^{\lambda}_{1} - \widetilde{Z}^{\lambda}_{1-\eta}\right| < \frac{\delta}{3} \right\} \cap \left\{ \widetilde{Z}^{\lambda}_{1} = x_\ell \right\}.
\end{align*}
Note that, in the last expression, the event is well-defined since we consider the coupled version of $\widetilde{X}^{(n)}$ and $\widetilde{Z}^\lambda$, and we can use Proposition~\ref{prop:QuenchedWeakConvergence} and Skorohod's representation theorem, together with the fact that $\widetilde{Z}^{\lambda}$ is continuous almost-surely.  Moreover, the event $\{ |x_\ell - x_\ell^{(n)}| < \delta/3 \}$ is measurable with respect to the environment and is thus deterministically true for all $n$ large enough by Proposition~\ref{PropositionPoissonPoint} and the construction of Section~\ref{SectionCopuledSpaces}. Using a union bound we obtain
\begin{align*}
	&\widetilde{P}^{\omega, \lambda/n}\left(\left| \widetilde{X}^{(n)}_{1 - \eta} - x_{\ell}^{(n)} \right| < \delta \right) \\
	&\ge \widetilde{P}^{\omega, \lambda}\left(\widetilde{Z}_{1} = x_\ell \right) -  \widetilde{P}^{\omega, \lambda/n}\left(\left|\widetilde{X}^{(n)}_{1-\eta} - \widetilde{Z}^{\lambda}_{1-\eta}\right| \ge \delta/3 \right)- \widetilde{P}^{\omega, \lambda}\left(\left|\widetilde{Z}_{1} - \widetilde{Z}_{1-\eta}\right| \ge \delta/3\right) .
\end{align*}
Plugging the last estimate back into \eqref{eqn:FirstBoundBelowRWRWT} and choosing  $\eta(\delta) \to 0$,  we observe that
\begin{align}
	\liminf_{n \to \infty} &\widetilde{P}^{\omega, \lambda/n}\left(\widetilde{X}^{(n)}_1 \in \left\{x_\ell^{(n)}, x_\ell^{(n)} + 1/n\right\} \right) \ge \nonumber \\ & \widetilde{P}^{\omega, \lambda}\left(\widetilde{Z}_{1} = x_0 \right) \lim_{\substack{\delta \to 0 \\ \eta(\delta) \to 0}} \liminf_{n \to \infty} \widetilde{P}^{\omega, \lambda/n}\left(\widetilde{X}^{(n)}_1 \in \left\{x_\ell^{(n)}, x_\ell^{(n)} + 1/n\right\} \Big| \big| \widetilde{X}^{(n)}_{1 - \eta} - x_{\ell}^{(n)} \big| < \delta \right) \nonumber \\
	&- \lim_{\delta \to 0} \limsup_{n \to \infty} \widetilde{P}^{\omega, \lambda/n}\left(\left|\widetilde{X}^{(n)}_{1-\eta} - \widetilde{Z}^{\lambda}_{1-\eta}\right| \ge \delta/3 \right) \label{eqn:ThirdBoundBelowRWRWT} \\
	&- \lim_{\substack{\delta \to 0 \\ \eta(\delta) \to 0}} \widetilde{P}^{\omega, \lambda}\left(\left|\widetilde{Z}_{1} - \widetilde{Z}_{1-\eta}\right| \ge \delta/3\right) \label{eqn:FourthBoundBelowRWRWT}.
\end{align}
The term in \eqref{eqn:ThirdBoundBelowRWRWT} is $0$  using Proposition~\ref{prop:QuenchedWeakConvergence} and the coupling given by Skorohod's representation theorem. The  term \eqref{eqn:FourthBoundBelowRWRWT} is $0$ by Lemma \ref{lemmaleftcont}. The result then follows from Lemma~\ref{Lemma:ConditionalProbHitting}. The extension to a generic starting point $\beta_n$ is straightforward, and we thus conclude the proof.
\end{proof}

Our goal is now to prove Lemma \ref{Lemma:ConditionalProbHitting}. For this purpose, we will first state and prove three lemmas providing random walks estimates on the interval we define below. We will work with an atom $(x_\ell, w_\ell)$ of $\nu^{\alpha_\infty}$ and $(j_\ell^{(n)}/n,w_\ell^{(n)})$ the sequence converging to $(x_\ell, w_\ell)$ provided by Proposition~\ref{PropositionPoissonPoint}. Let us fix $\delta'> \delta >0$ and the following intervals:
\begin{itemize}
	\item $I_n^\delta \coloneqq \left[ j_{\ell}(n) - \delta n, j_{\ell}(n) + 1 + \delta n \right]$;
	\item $I_n^{\delta'} \coloneqq \left[ j_{\ell}(n) - \delta' n, j_{\ell}(n) + 1 + \delta' n \right]$.
\end{itemize}
See Figure \ref{fig:figIntervalAroundLnV2}.

\begin{figure}[H]
	\centering
	\includegraphics[width=.7\linewidth]{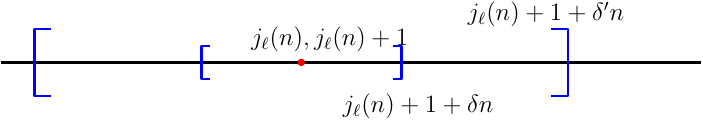}
	\captionof{figure}{Visualisation the interval around $\{j_{\ell}(n), j_{\ell}(n) + 1\}$. The two points are visualised as collapsed for simplicity.}
	\label{fig:figIntervalAroundLnV2}
\end{figure}

For a subset $A\subset\mathbb{Z}$, let us define $\tau^{\widetilde{X}^{(n)}}(A) \coloneqq \inf \{s \ge 0 : \, \widetilde{X}_s^{(n)} \in A\}$ to be the first hitting time of the set $A$ for $\widetilde{X}^{(n)}$. For any $\vartheta>0$ and $y \in \R$, define
\begin{equation*}
	\Delta^{+, \alpha_0}_{y}(\vartheta) \coloneqq S^{\alpha_0, \lambda}(y + \vartheta) - S^{\alpha_0, \lambda}(y), \quad \mathrm{and} \quad \Delta^{-, \alpha_0}_{y}(\vartheta) \coloneqq S^{\alpha_0, \lambda}(y^-) - S^{\alpha_0, \lambda}(y - \vartheta).
\end{equation*}
Note that by the definition of $j_{\ell}(n)$ given below \eqref{eqn:eqnProofPointProcess}, $I_{n}^{\delta}/n$ contains the interval $[x_\ell - \delta, x_\ell + \delta]$ almost-surely by construction (respectively the same works with $\delta'$).

\begin{lemma}\label{LemmaHittingTheMiddle}
The walk, started inside $I_n^\delta$, hits $ \{j_{\ell}(n), j_{\ell}(n) + 1\}$ before exiting $I_n^{\delta'}$, with high probability. More precisely, there exists a positive constant $C_1 > 0$ such that
\[\limsup_{n \to \infty} \sup_{x \in I_n^\delta} \widetilde{P}_x^{\omega, \lambda/n} \left( \tau^{\widetilde{X}^{(n)}}\left(\{j_{\ell}(n), j_{\ell}(n) + 1\}\right) > \tau^{\widetilde{X}}\left((I^{\delta'}_n)^c\right) \right) \le C_1 \left(\frac{\Delta^{+, \alpha_0}_{x_\ell}(\delta)}{\Delta^{+, \alpha_0}_{x_\ell}(\delta'/2)} + \frac{\Delta^{-, \alpha_0}_{x_\ell}(2\delta)}{\Delta^{-, \alpha_0}_{x_\ell}(\delta')} \right).\]
\end{lemma}

\begin{proof}
Let us assume that $x \in \left[j_n(\ell) + 1, j_n(\ell) + n \delta\right]$. For all $x \in I_n^\delta$ we have by well known electrical networks formulas (for instance, \cite[Eqn.~(9.13)]{LPBook}) that
\begin{equation*}
	 \widetilde{P}_x^{\omega, \lambda/n} \left( \tau^{\widetilde{X}^{(n)}}\left( \{j_{\ell}(n), j_{\ell}(n) + 1\}\right) > \tau^{\widetilde{X}}\left( (I^{\delta'}_n)^c\right) \right) = \frac{R^{\lambda/n}\left(j_{\ell}(n) + 1, x\right)}{R^{\lambda/n}\left(j_{\ell}(n) + 1, j_{\ell}(n)+\lfloor n\delta'\rfloor +2\right)}.
\end{equation*}
 As we work under the coupling of Section \ref{SectionCopuledSpaces}, we have that, almost surely,
\begin{equation*}
	\limsup_{n \to \infty} \sup_{x \in \left[j_n(\ell) + 1, j_n(\ell) +1+ n \delta\right]} \frac{R^{\lambda/n}\left(x, j_{\ell}(n) + 1\right)}{R^{\lambda/n}\left(j_{\ell}(n) + 1, j_{\ell}(n)+1+\lfloor n\delta'\rfloor +1\right)} \le C_1 \frac{\Delta^{+, \alpha_0}_{x_\ell}(\delta)}{\Delta^{+, \alpha_0}_{x_\ell}(\delta'/2)}.
\end{equation*}
Following symmetric arguments for $x \in \left[j_n(\ell) -n\delta, j_n(\ell) \right]$ (taking into account the asymmetry of the jumps of $S^{\alpha_0,\lambda}$), we conclude the proof.
\end{proof}

The lemma below shows that a walk started inside $I_n^\delta$ and reflected at the boundary of $I_n^{\delta'}$ hits $\{j_{\ell}(n), j_{\ell}(n) + 1\}$ quickly. We will denote $\widetilde{E}^{\omega, \lambda/n}_{x, |I_n^{\delta'}|}$ the expectation associated with this reflected random walk, started at $x$.
\begin{lemma}\label{LemmaHittingLnFast}
There exists a positive constant $C_2 > 0$ such that
\begin{align*}
&\limsup_{n \to \infty} \sup_{x \in I_n^\delta} \widetilde{E}^{\omega, \lambda/n}_{x, |I_n^{\delta'}|} \left[ \tau^{\widetilde{X}}(x, \{j_{\ell}(n), j_{\ell}(n) + 1\}) \right] \\
&\le C_2 \left(\Delta^{+, \alpha_0}_{x_\ell}(\delta) + \Delta^{-, \alpha_0}_{x_\ell}(2\delta) \right) \left( \Delta^{+, \alpha_\infty}_{x_\ell}(\delta') + \Delta^{-, \alpha_\infty}_{x_\ell}(2\delta') \right) b_n.
\end{align*}
\end{lemma}

\begin{proof}
As we consider the reflected random walk, the network it evolves on is finite, and thus we can then apply the commute time identity (\cite[Prop.~10.7]{LPBook} and \cite[Cor.~2.21]{LPBook}) and get that, for any $x \in \{j_n(\ell) + 1, j_n(\ell) + n \delta\}$
	\begin{align*}
		\widetilde{E}^{\omega, \lambda/n}_{x, |I_n^{\delta'}|} \left[ \tau^{\widetilde{X}}(x, \{j_{\ell}(n), j_{\ell}(n) + 1\}) \right] \le R^{\lambda/n} \left( x, j_{\ell}(n) + 1 \right) \sum_{i = j_{\ell}(n) + 1}^{j_n(\ell) + n \delta'} c^{\lambda/n} (\{i, i+1\}) .
	\end{align*}
As we work under the coupling of Section \ref{SectionCopuledSpaces}, we obtain that
\begin{align*}
	&\limsup_{n \to \infty} \sup_{x \in \{j_n(\ell) + 1, j_n(\ell) + n \delta\}} R^{\lambda/n} \left( x, j_{\ell}(n) + 1 \right) \sum_{i = j_{\ell}(n) + 1}^{j_n(\ell) + n \delta'} c^{\lambda/n} (\{i, i+1\}) \\
	&\le C_2 d_{n, \infty} d_{n, 0} \Delta^{+, \alpha_\infty}_{x_\ell}(\delta') \Delta^{+, \alpha_0}_{x_\ell}(\delta).
\end{align*}
A symmetric argument for $x \in \{j_n(\ell) - \delta n, j_n(\ell)\}$ completes the proof.
\end{proof}

\begin{lemma}\label{LemmaExitSlowly}
	A walk started inside $\{j_{\ell}(n), j_{\ell}(n) + 1\}$ exits $I_n^{\delta'}$ slowly enough, because it spends a lot of time around its starting point. More precisely, for all $\eta>0$, we have that
	\[
	\liminf_{n \to \infty} \min_{x \in\{ x_\ell^{(n)}, x_\ell^{(n)} + 1/n\}} \widetilde{P}_x^{\omega, \lambda/n}\left(\widetilde{X}_{[0, \eta]}^{(n)} \in \left[x_\ell - \delta', x_\ell + \delta' \right] \right) \ge \widetilde{P}^{\omega, \lambda}_{x_\ell} \left( \left\| \widetilde{Z}_{[0, \eta]} - x_\ell \right\|_\infty > \frac{\delta'}{2} \right).
	\]
	Here $X_{[0, t]}$ denotes the whole trajectory of the process $X$ between time $0$ and $t$.
\end{lemma}

\begin{proof}
	The key step of the proof is recalling from Proposition~\ref{prop:QuenchedWeakConvergence} that $\widetilde{\mathbf{P}}$ almost surely, on the coupling of Section~\ref{SectionCopuledSpaces} we have that for $\beta_n \in \{ x_\ell^{(n)}, x_\ell^{(n)} + 1/n\}$
	\begin{equation*}
		\widetilde{P}_{\beta_n}^{\omega, \lambda/n}\left( \left(n^{-1} \widetilde{X}_{t b_n} \right)_{t \ge 0} \in \cdot \right),
	\end{equation*}
	converges, weakly as a probability measures on $D([0, \infty), \R)$, to the law of $(\widetilde{Z}^{\lambda}_t)_{t \ge 0}$ started from $x_\ell$. To establish this fact, one just needs to notice that $x_\ell$ is almost-surely a continuity point for the resistance metric. Since $\widetilde{Z}^{\lambda}_t$ is almost-surely continuous, using Skorohod's representation theorem, we can couple the processes so that $P_{\beta_n}^{\omega, \lambda/n}$ almost surely
	\begin{equation*}
		\sup_{t \in [0, \eta(\delta)]} \left\| \widetilde{X}_{t}^{(n)} - \widetilde{Z}^{\lambda}_{t} \right\|_\infty \to 0, \quad \text{as } n \to \infty.
	\end{equation*}
	Then, we can write that, for $\beta_n \in \left\{ x_\ell^{(n)}, x_\ell^{(n)} + 1/n\right\}$
	\begin{align*}
		\liminf_{n \to \infty} \widetilde{P}_{\beta_n}^{\omega, \lambda/n}&\left(\widetilde{X}_{[0, \eta]}^{(n)} \in \left[x_\ell - \delta', x_\ell + \delta' \right] \right) \ge \\ & \quad \quad  \liminf_{n \to \infty} \widetilde{P}_{\beta_n, x_\ell}^{\omega, \lambda/n}\left( \left\|\widetilde{X}_{[0, \eta]}^{(n)} - \widetilde{Z}^{\lambda}_{[0, \eta]} \right\|_\infty\le \frac{\delta'}{3}, \left\|\widetilde{Z}^{\lambda}_{[0, \eta]} - x_\ell \right\|_\infty \le \frac{\delta'}{2} \right).
	\end{align*}
Using the coupling, we get that the first event in the last probability happens almost-surely for all $n$ large enough. So we get the following bound
\begin{align*}
\liminf_{n \to \infty} \widetilde{P}_{\beta_n}^{\omega, \lambda/n}\left(\widetilde{X}_{[0, \eta]}^{(n)} \in \left[x_\ell - \delta', x_\ell + \delta' \right] \right) \ge \widetilde{P}^{\omega, \lambda}_{x_\ell} \left( \left\| \widetilde{Z}_{[0, \eta]} - x_\ell \right\|_\infty \le \frac{\delta'}{2} \right),
\end{align*}
as desired.
\end{proof}

\begin{lemma}\label{LemmaInvariantMeasure}
When the random walk is inside $I_n^{\delta'}$, then it spends almost all the time on the set $\{j_{\ell}(n), j_{\ell}(n) + 1\}$. More precisely,  there exists a positive constant $C_4 > 0$ such
\[\limsup_{n \to \infty} \sup_{s > 0} \max_{\widehat{\beta}_n \in \{j_{\ell}(n), j_{\ell}(n) + 1\}} \widetilde{P}^{\omega, \lambda/n}_{\widehat{\beta}_{n}, |I_n^{\delta'}|} \left( \widetilde{X}_s \not\in \{j_{\ell}(n), j_{\ell}(n) + 1\}\right) \le C_4 \frac{\left( \Delta^{+, \alpha_\infty}_{x_\ell}(\delta') + \Delta^{-, \alpha_\infty}_{x_\ell}(2\delta') \right)}{\nu^{\alpha_\infty} (x_\ell) e^{-2\lambda K}}.\]
\end{lemma}

\begin{proof}
Start by observing that
\begin{equation*}
	\pi(x) = \frac{c^{\lambda/n}(x)}{c^{\lambda/n} (\{ j_\ell(n), j_\ell(n) + 1\})}
\end{equation*}
is an invariant measure for the walk and $\pi(\widehat{\beta}_n) \ge 1$. Hence, for either value of $\widehat{\beta}_n$, we have that, for all $s > 0$
	\begin{equation}\label{eq1}
		\widetilde{P}^{\omega, \lambda/n}_{\widehat{\beta}_{n}, |I_n^{\delta'}|} \left( \widetilde{X}_s \not\in \{j_{\ell}(n), j_{\ell}(n) + 1\}\right) \le \frac{1}{c^{\lambda/n} (\{ j_\ell(n), j_\ell(n) + 1\})} \sum_{x \in I_n^{\delta'} \backslash \{j_\ell(n), j_\ell(n) + 1\}} c^{\lambda/n}(x).
	\end{equation}
The coupling of Section \ref{SectionCopuledSpaces} implies that
\begin{equation}\label{eq2}
	\frac{1}{d_{n, \infty}} c^{\lambda/n} (\{ j_\ell(n), j_\ell(n) + 1\}) \stackrel{n \to \infty}{\to} \nu^{\alpha_\infty} (x_\ell) e^{-\lambda x_\ell},
\end{equation}
and
\begin{equation}\label{eq3}
\limsup_{n \to \infty} \frac{1}{d_{n, \infty}} \sum_{x \in I_n^{\delta'} \backslash \{j_\ell(n), j_\ell(n) + 1\}} c^{\lambda/n}(x) \le C_4 \left( \Delta^{+, \alpha_\infty}_{x_\ell}(\delta') + \Delta^{-, \alpha_\infty}_{x_\ell}(2\delta') \right).
\end{equation}
One can conclude the proof by inserting \eqref{eq2} and \eqref{eq3} into \eqref{eq1}.
\end{proof}

\begin{proof}[Proof of Lemma \ref{Lemma:ConditionalProbHitting}]
Recall that $(x_\ell, w_\ell)$ is an atom of $\nu^{\alpha_\infty}$ and $(x_\ell^{(n)},w_\ell^{(n)})$ is the sequence converging to $(x_\ell, w_\ell)$, as provided by Proposition~\ref{PropositionPoissonPoint}. Let us fix the notation $A_n \coloneqq \{x_\ell^{(n)}, x_\ell^{(n)} + 1/n\}$, and let $\tau^* \coloneqq \inf \{s \ge 0 : \, \widetilde{X}_s^{(n)} \in A_n\}$. Fix  $\delta' > \delta$ and observe that, using Markov's property,
\begin{align}
	\widetilde{P}^{\omega, \lambda/n}\left(\widetilde{X}^{(n)}_1 \not \in A_n \Big| \big| \widetilde{X}^{(n)}_{1 - \eta} - x_{\ell}^{(n)} \big| < \delta \right) &\le \sup_{x \in I_n^\delta} \widetilde{P}_x^{\omega, \lambda/n}\left(\widetilde{X}^{(n)}_{\eta} \not \in A_n\right) \nonumber\\
	&\le \sup_{x \in I_n^\delta} \widetilde{P}^{\omega, \lambda/n}_x\left(\tau^* > \eta \right) + \sup_{x \in I_n^\delta} \widetilde{P}^{\omega, \lambda/n}_x\left(\widetilde{X}^{(n)}_{\eta} \not \in A_n, \tau^* \le \eta \right).\nonumber
\end{align}
Let $\theta_t$ be the canonical time shift by time $t$. Then the last term in the sum can be bounded from above by
\begin{align*}
	\lefteqn{2\sup_{x \in I_n^\delta} \max_{y \in A_n} \widetilde{P}^{\omega, \lambda/n}_x\left(\tau^* \le \eta, \widetilde{X}^{(n)}_{\tau^*} = y, \widetilde{X}^{(n)}_{\eta - \tau^*} \circ \theta_{\tau^*} \not \in A_n \right) }\\
&= 2\sup_{x \in I_n^\delta} \max_{y \in A_n} \widetilde{E}^{\omega, \lambda/n}_x\left[ \mathds{1}_{\left\{\tau^* \le \eta, \widetilde{X}^{(n)}_{\tau^*} = y \right\}} \left.P_y^{\omega, \lambda/n}\left( \widetilde{X}^{(n)}_{\eta - t} \not \in A_n \right)\right|_{t=\tau^*} \right]
\end{align*}
Hence, by applying standard computations, we get that for $C>0$
\begin{align}
	\widetilde{P}^{\omega, \lambda/n}\left(\widetilde{X}^{(n)}_1 \in A_n \Big| \big| \widetilde{X}^{(n)}_{1 - \eta} - x_{\ell}^{(n)} \big| < \delta \right) &\ge 1 - C\bigg(\sup_{x \in I_n^\delta} \widetilde{P}_x^{\omega, \lambda/n}\left( \tau^* >\eta \right)\nonumber\\
	& + \max_{y \in A_n} \widetilde{P}_y^{\omega, \lambda/n}\left( \left\| \widetilde{X}_{[0, \eta]}^{(n)} - x_\ell \right\|_\infty > \delta' \right) \nonumber \\
	& + \max_{y \in A_n} \sup_{s > 0} \widetilde{P}_{y, |I_n^{\delta'}|}^{\omega, \lambda/n}\left( \widetilde{X}_{s}^{(n)} \not \in  \left\{x_\ell^{(n)}, x_\ell^{(n)} + 1/n\right\} \right) \bigg) \nonumber,
\end{align}
where we recall that $\widetilde{P}_{y, |I_n^{\delta'}|}^{\omega, \lambda/n}$ denotes the measure associated with the random walk reflected on the boundary of $I_n^{\delta'}$. In order to conclude, let us fix $\delta'\in(0,1)$ and $\varepsilon>0$. First, by Lemma~\ref{LemmaInvariantMeasure}, together with the fact that $\nu^{\alpha_\infty}(x_\ell) > 0$ and that the subordinator is right-continuous with left limits, we have that we can choose $\delta'>0$ small enough so that
\begin{equation}
	\limsup_{n \to \infty} \max_{y \in A_n} \sup_{s > 0} \widetilde{P}_{y, |I_n^{\delta'}|}^{\omega, \lambda/n}\left( \widetilde{X}_{s}^{(n)} \not \in  \left\{x_\ell^{(n)}, x_\ell^{(n)} + 1/n\right\} \right) \le \varepsilon. \nonumber
\end{equation}
Second, by Lemma~\ref{LemmaExitSlowly} and Lemma~\ref{lemmaExitIntervalDiffusion}, we can choose $\eta>0$ small enough such that
\[
	\limsup_{n \to \infty} \max_{y \in A_n} \widetilde{P}_y^{\omega, \lambda/n}\left( \left\| \widetilde{X}_{[0, \eta]}^{(n)} - x_\ell \right\|_\infty > \delta' \right) \le \widetilde{P}^{\omega, \lambda}_{x_\ell} \left( \left\| \widetilde{Z}_{[0, \eta]} - x_\ell \right\|_\infty > \frac{\delta'}{2} \right) \le \varepsilon.\]
Finally, using Lemma~\ref{LemmaHittingTheMiddle}, and Lemma~\ref{LemmaHittingLnFast}, we can choose $\delta>0$ small enough, depending on $\eta$ and $\delta'$, such that
\begin{equation*}
	\limsup_{n \to \infty} \sup_{x \in I_n^\delta} \widetilde{P}_x^{\omega, \lambda/n}\left( \tau^* >\eta \right) \le \varepsilon.
\end{equation*}
Overall, we have shown that
\begin{equation*}
	\liminf_{n \to \infty} \widetilde{P}^{\omega, \lambda/n}\left(\widetilde{X}^{(n)}_1 \in \{x_\ell^{(n)}, x_\ell^{(n)} + 1/n\} \Big| \big| \widetilde{X}^{(n)}_{1 - \eta} - x_{\ell}^{(n)} \big| < \delta \right) \ge 1 - 4C \varepsilon,
\end{equation*}
which is enough to conclude the proof, since $\varepsilon$ can be chosen arbitrarily small by taking the appropriate $\delta$ and $\eta(\delta)$.
\end{proof}

\section{Proof of the aging results}\label{sec:agingproof}

In Section~\ref{SectionAgingRW} we prove Theorem~\ref{TheoremGap} and Proposition~\ref{theoremAging}. In Section~\ref{SectionAgingRWT}, we prove the first part of Proposition \ref{theoremSubAging}, which is the aging statement. For this purpose, we first prove the result restricted to the box $[-K, K]$, in Proposition~\ref{PropositionInsideKAging}, and then extend it to the whole process. Using Proposition~\ref{PropostionWallsPointProcess}  and Proposition~\ref{PropositionMarginalWithTraps} is crucial is this section.

\subsection{Aging under \eqref{RW}}

Again, we work under the coupling of Section \ref{SectionCopuledSpaces}. Let $\mathbb{P}^{\lambda/n, K}$ be the annealed law of the random walk reflected at the boundary of the box $[-Kn, Kn]$ and $\mathbb{P}^{\lambda, K}$ that of the corresponding diffusion reflected at the boundary of $[-K, K]$, with $K \in \mathbb{N}$. Recall the definition \eqref{EquationDefGaps} of $\mathrm{Gap}^{\lambda}_{n}(t)$ and $\mathrm{Gap}^{\lambda}(t)$. Under $\mathbb{P}^{\lambda/n, K}$, we assign a deterministic value $C^{*} \in \mathbb{R}_{+}$ to $\mathrm{Gap}^{\lambda}_{n}(t)$ (respectively $\mathrm{Gap}^{\lambda}(t)$) if $\widebar{X}_{a_n} = Kn$ (respectively $\widebar{Z}^{\lambda}_1 = K$).

Our main goal here is to prove Theorem \ref{TheoremGap} but, before that, we need to prove the convergence of the relevant point processes. Under the coupling of Section~\ref{SectionCopuledSpaces}, let us define $v_i^{(n)} \coloneqq P^{\omega, \lambda/n, K}(\widebar{X}_{a_n} = i/n)$ and, for an atom $x_\ell\ge0$ of $\nu^{\alpha_0}$, $v_\ell \coloneqq P^{\omega, \lambda, K}(\widebar{Z}_1 = x_\ell)$. Note that
\begin{align*}
	\widebar{\rho}^{(n)}(1) &= \sum_{i/n \in [0, K]} \delta_{i/n} v_i^{(n)}, \\
	\widebar{\rho}(1) &= \sum_{\ell: x_\ell \in [0, K]} \delta_{x_{\ell}} v_{\ell},
\end{align*}
where $x_\ell$ are the locations of the atoms of the measure $\widebar{\rho}(1)$ (which is purely atomic, see Lemma~\ref{lemmaMaximumDisc} below), together with the convention $x_0 = K$. Set for simplicity $r^{(n)}(\{i, i+1\}) = d_{n, 0}^{-1} r^{\lambda/n}(\{i, i+1\})$. Let us define
\[	\begin{split}
		\widebar{\pi}^{(n)}(dx) &\coloneqq \sum_{i/n \in [0, K]} \delta_{r^{(n)}(\{i, i+1\})} v_i^{(n)}, \\
		\widebar{\pi}(dx) &\coloneqq \sum_{\ell: x_\ell \in [0, K]} \delta_{e^{-\lambda x_\ell}\nu^{\alpha_0}(x_\ell)} v_\ell,
	\end{split}\]
where we attach the special value $ r^{(n)}(\{Kn, Kn + 1\}) = e^{-\lambda K} \nu^{\alpha_0}(x_0) = C^*$.

\begin{lemma}\label{LemmaConvergenceWallMeas}
	Under Assumption~\eqref{RW}, for almost every realisations of the environment, as $n \to \infty$
	\begin{equation*}
		\widebar{\pi}^{(n)}(dx) \to \widebar{\pi}(dx),
	\end{equation*}
	both vaguely and in the point process sense.
\end{lemma}

\begin{proof}
By Lemma~\ref{lemmaMaximumDisc} we get that every atom $(e^{-\lambda x_\ell}\nu^{\alpha_0}(x_\ell), v_\ell)$ of $\widebar{\pi}$ corresponds to an atom $(x_\ell, w_\ell)$ of $\nu^{\alpha_0}$ (restricted to $[0, K]$) defined in \eqref{eqn:PoissonPointMeasures}, with an extra atom at $K$. By Proposition~\ref{PropostionWallsPointProcess} we immediately have that Condition~\ref{Condition1} is satisfied by the measures $\widebar{\pi}^{(n)}$ and $\widebar{\pi}$. Note that this is also true for the special atom at $K$. We just need to prove vague convergence. Let $I_{\delta} \coloneqq \{i \colon i/n \in [0, K]\text{ and } v_{i}^{(n)} > \delta \}$, then for any continuous and non-negative function on $[-K, K]$, we have that
\begin{align*}
	\int f(x) \widebar{\pi}^{(n)} (dx) &= \sum_{i \in I_{\delta}} v_{i}^{(n)} f\left( r^{(n)}({i, i + 1}) \right) + \sum_{i \not\in I_{\delta}} v_{i}^{(n)} f\left( r^{(n)}({i, i + 1}) \right).
\end{align*}
As $\{\ell: v_\ell>\delta\}$ is almost-surely finite,  the observation above concerning Condition~\ref{Condition1} yields
\begin{equation*}
	\lim_{n \to \infty} \sum_{i \in I_{\delta}} v_{i}^{(n)} f\left( r^{(n)}({i, i + 1}) \right) = \sum_{\ell: v_\ell>\delta} v_{\ell} f\left( \nu^{\alpha_0}(x_\ell) e^{- x_\ell}\right)
\end{equation*}
for Lebesgue almost-every $\delta>0$. The right-hand side is monotone and bounded by $\|f\|_\infty$, thus as $\delta \to 0$ it converges to $\int f(x) \widebar{\pi}(dx)$. Analogously, we obtain that
\begin{align*}
	\limsup_{n\to \infty} \sum_{i \not\in I_{\delta}} v_{i}^{(n)} f\left( r^{(n)}({i, i + 1}) \right) &\le \|f\|_\infty \left( 1 - \liminf_{n \to \infty} \sum_{i \in I_{\delta}} v_{i}^{(n)}  \right) \\
	&\le \|f\|_\infty \left( 1 - \sum_{\ell: v_\ell>\delta} v_{\ell} \right).
\end{align*}
Since $\widebar{\rho}(1)$ is almost-surely a purely atomic measure, we conclude the proof by taking the limit as $\delta\to 0$.
\end{proof}

Now, we are ready to prove the convergence of the gap processes.

\begin{proof}[Proof of Theorem~\ref{TheoremGap}] \label{SectionAgingRW}
	 We aim to prove that for every bounded $f : \mathbb{R}_+ \to \mathbb{R}$
	\begin{equation}\label{eqn:GapConvergence}
		\lim_{n \to \infty} \mathbb{E}^{\lambda/n}\left[ f \left( \mathrm{Gap}^{\lambda}_{n}(1) \right) \right] = \mathbb{E}^{\lambda}\left[ f \left( \mathrm{Gap}^{\lambda}(1) \right) \right],
	\end{equation}
Because it is easier to work with the reflected processes, let us start by observing that
\begin{align*}
	\mathbb{E}^{\lambda/n}\left[ f \left( \mathrm{Gap}^{\lambda}_{n}(1) \right) \mathds{1}_{ \left\{\tau^{X}_{-Kn} \wedge \tau^{X}_{Kn} > a_n\right\} } \right] &= \mathbb{E}^{\lambda/n, K}\left[ f \left( \mathrm{Gap}^{\lambda}_{n}(1) \right) \mathds{1}_{ \left\{\tau_{-Kn} \wedge \tau_{Kn} > a_n\right\} } \right], \\
	\mathbb{E}^{\lambda}\left[ f \left( \mathrm{Gap}^{\lambda}(1) \right) \mathds{1}_{ \left\{ \tau^{Z^{\lambda}}_{-K} \wedge \tau^{Z^{\lambda}}_{K} > 1\right\} }  \right] &= \mathbb{E}^{\lambda, K}\left[ f \left( \mathrm{Gap}^{\lambda}(1) \right) \mathds{1}_{ \left\{ \tau^{Z^{\lambda}}_{-K} \wedge \tau^{Z^{\lambda}}_{K} > 1\right\} } \right].
\end{align*}
The above yields that
\begin{align}\label{danse1}
	\left| \mathbb{E}^{\lambda/n}\left[ f \left( \mathrm{Gap}^{\lambda}_{n}(1) \right) \right] - \mathbb{E}^{\lambda/n, K}\left[ f \left( \mathrm{Gap}^{\lambda}_{n}(1) \right)\right] \right| &\le \left\| f \right\|_{\infty} \mathbb{P}^{\lambda/n} \left( \tau^{X}_{-Kn} \wedge \tau^{X}_{Kn} \le a_n \right), \\ \label{danse2}
	\left|\mathbb{E}^{\lambda}\left[ f \left( \mathrm{Gap}^{\lambda}(1) \right) \right] - \mathbb{E}^{\lambda, K}\left[ f \left( \mathrm{Gap}^{\lambda}(1) \right) \right] \right| &\le \left\| f \right\|_{\infty} \mathbb{P}^{\lambda} \left( \tau^{Z}_{-K} \wedge \tau^{Z}_{K} \le 1 \right).
\end{align}
Using Lemma~\ref{Lemma:HittingKComparison}, we can bound the $\limsup$ of the probability on the right-hand side of \eqref{danse1} as follows:
\begin{equation}\label{eqn:RecallBoundHittingK}
	\limsup_{n} \,\,\mathbb{P}^{\lambda/n, K}\left( \tau_{Kn}^{X} \wedge \tau_{-Kn}^{X} \le a_n \right) \le \mathbb{P}^{\lambda, K}\left( \tau_{K - 1}^{Z} \wedge \tau_{-K + 1}^{Z} \le 2 \right).
\end{equation}
By \cite[Lemma 5.3]{MottModel}, we have that the right-hand side of \eqref{danse2} and \eqref{eqn:RecallBoundHittingK} converge to $0$ as $K$ goes to infinity. Hence, in order to prove \eqref{eqn:GapConvergence}, we only need to prove that the $\mathbb{P}^{\lambda/n, K}$-law of    $\mathrm{Gap}^{\lambda}_{n}(1) $ converges towards the $\mathbb{P}^{\lambda, K}$-law of $\mathrm{Gap}^{\lambda}(1)$.

One can observe that
\begin{align*}
	\mathbb{P}^{\lambda/n, K}\left( \mathrm{Gap}^{\lambda}_{n}(1)\le u \right) &=  \mathbf{E}\left[ P^{\omega, \lambda/n, K}\left( r^{(n)}\left(\widebar{X}_{a_n},\widebar{X}_{a_n} + 1\right)\le u \right)\right]\\
	&= 1 - \mathbf{E}\left[ P^{\omega, \lambda/n, K}\left( r^{(n)}\left(\widebar{X}_{a_n}, \widebar{X}_{a_n} + 1 \right) > u \right) \right]\\
	& = 1 - \mathbf{E}\left[ \sum_{r^{(n)}(\{i, i+1\}) > u} v_i^{(n)} \right].
\end{align*}
Under the coupling, using Lemma~\ref{LemmaConvergenceWallMeas} and the fact that $\{\ell: \nu^{\alpha_0}(x_\ell)>u\}$ is almost surely finite, we obtain that
\begin{equation*}
	\lim_{n\to \infty} \sum_{r^{(n)}(i, i+1) > u} v_i^{(n)} = \sum_{\nu^{\alpha_0}(x_{\ell}) e^{-\lambda x_\ell} > u} v_\ell,
\end{equation*}
as long as $u \ne \nu^{\alpha_0}(x_{\ell}) e^{-\lambda x_\ell}, \forall \ell$, which is true for Lebesgue almost-every $u$. We conclude the proof by applying the dominated convergence theorem.
\end{proof}

Finally, we check our main aging result under assumption \eqref{RW}.

\begin{proof}[Proof of Proposition~\ref{theoremAging}]
We prove that the aging statement holds in the box $[-K, K]$, i.e.\ we show that
	\begin{equation}\label{eqn:newAgingInsideK}
		\lim_{n \to \infty} \Prob^{\lambda/n, K}\left( \widebar{X}_{a_n} = \widebar{X}_{[a_n,h a_n]} \right) = \theta(h) \coloneqq \Prob^{\lambda, K} \left(\widebar{Z}_{1} = \widebar{Z}_{[1, h]} \right).
	\end{equation}
The extension to the whole space follows from Lemma~\ref{Lemma:HittingKComparison} and \cite[Lemma 5.3]{MottModel} by repeating the steps of the proof of Theorem~\ref{TheoremGap}; we omit the details.

We work under the coupling of Section \ref{SectionCopuledSpaces}. From the $J_1$-convergence of $X^{(n)} \coloneqq (n^{-1}X_{ta_n})_{t\geq 0}$ to $Z^\lambda$ of Proposition \ref{prop:QuenchedWeakConvergence} (and the continuity of the limiting process), we have
\begin{equation}\label{rerere}
	E_0^{\omega,\lambda/n, K}\left[f\left(X_1^{(n)}\right)g\left(\widebar{X}_1^{(n)}\right)\right]\rightarrow
	E_0^{\omega,\lambda, K}\left[f\left(Z_1\right)g\left(\widebar{Z}_1\right)\right]
\end{equation}
for all continuous functions $f$ and $g$ on $[-K,K]$. Taking $g=g_\delta$, as defined by setting
\[g_\delta(x)\coloneqq (1-2\delta^{-1}|x-x_\ell|)_+,\]
where $x_\ell$ is the location of a discontinuity of $S^{\alpha_0}$, the right-hand side of \eqref{rerere} converges as $\delta\rightarrow0$ to
\[ E_0^{\omega,\lambda, K}\left[f\left(Z_1\right)\mathds{1}_{\{\widebar{Z}_1=x_\ell\}}\right].\]
As for the left-hand side of \eqref{rerere}, recall the definition of $j_{\ell}(n)$ given below \eqref{eqn:eqnProofPointProcess}, as soon as $n$ is large enough, we have that $|x_\ell-j_\ell(n)/n|\le 1/n \le \delta/2$, and therefore we obtain
\begin{eqnarray*}
	\lefteqn{\left|E_0^{\omega,\lambda/n, K}\left[f\left(X_1^{(n)}\right)g_\delta\left(\widebar{X}_1^{(n)}\right)\right]-
		E_0^{\omega,\lambda/n, K}\left[f\left(X_1^{(n)}\right)\mathds{1}_{\{\widebar{X}_1^{(n)}=j_\ell(n)/n\}}\right]\right|}\\
	&\leq&\|f\|_\infty\left( 2\delta^{-1}|x_\ell-j_\ell(n)/n|+P_0^{\omega,\lambda/n, K}\left(\widebar{X}_{a_n}\in[j_\ell(n)-\delta n,j_\ell(n)+\delta n]\backslash\{j_\ell(n)\}\right)\right),
\end{eqnarray*}
which, by Lemmas \ref{lemmaEstimate1RWRW} and \ref{lemmaEstimate2RWRW}, converges to 0 as $n\rightarrow \infty$ and then $\delta\rightarrow 0$. In particular, it follows that
\[E_0^{\omega,\lambda/n, K}\left[f\left(X_1^{(n)}\right)\mathds{1}_{\{\widebar{X}_1^{(n)}=j_\ell(n)/n\}}\right]\rightarrow
E_0^{\omega,\lambda, K}\left[f\left(Z_1\right)\mathds{1}_{\{\widebar{Z}_1=x_\ell\}}\right].\]
Combining this with Proposition \ref{PropostionWallsPointProcess} and the strict positivity of the limiting probabilities in \eqref{eqn:WallsPP} (as is confirmed by Lemma \ref{LemmaMaximumOnEveryWall}), this yields that
\[E_0^{\omega,\lambda/n, K}\left[f\left(X_1^{(n)}\right)\:\vline\:\widebar{X}_1^{(n)}=j_\ell(n)/n\right]\rightarrow
E_0^{\omega,\lambda, K}\left[f\left(Z_1\right)\:\vline\:\widebar{Z}_1=x_\ell\right].\]
Since the continuous function $f$ was arbitrary, this implies that if we define $\mu_{n,j_\ell(n)}$ to be the law of $X^{(n)}_1$ conditioned on $\widebar{X}_1^{(n)}=j_\ell(n)/n$, and $\mu_{x_\ell}$ to be the law of $Z^\lambda_1$ conditioned on $\widebar{Z}^\lambda_1=x_\ell$, then
\[\mu_{n,j_\ell(n)} \rightarrow\mu_{x_\ell},\]
weakly as probability measures on $\mathbb{R}$.

Now, from the conclusion of the previous paragraph, we obtain from Skorohod's representation theorem the existence of random variables $A_n$, $A$ built on the same probability space so that: $A_n\sim \mu_{n,j_\ell(n)} $, $A\sim\mu_{x_\ell}$, and $A_n\rightarrow A$, almost-surely. Noting from the proof of \cite[Lemma 5.4]{MottModel} that $A$ is almost-surely not at a discontinuity of $S^{\alpha_0}$, Proposition~\ref{PropostionWallsPointProcess} implies that we further have that
\[P_{A_n}^{\omega,\lambda/n, K}\left(\widebar{X}_{h-1}^{(n)}=j_\ell(n)/n\right)\rightarrow P_{A}^{\omega,\lambda, K}\left(\widebar{Z}_{h-1}=x_\ell\right),\]
almost-surely. Taking expectations of the above limit (and again applying Proposition~\ref{PropostionWallsPointProcess})), we conclude that
\[P^{\omega,\lambda/n, K}_0\left(\widebar{X}_1^{(n)}=j_\ell(n)/n=\widebar{X}_{[1,h]}^{(n)}\right)\rightarrow
P^{\omega,\lambda, K}_0\left(\widebar{Z}_1=x_\ell=\widebar{Z}_{[1,h]}\right).\]

Finally, recall that $v_\ell = P^{\omega, \lambda, K}(\widebar{Z}_1 = x_\ell)$, since $\{\ell:\:v_\ell>\delta\}$ is a finite set for each $\delta>0$, we deduce from the above conclusion that
\[\sum_{\ell:\:v_\ell>\delta}P^{\omega,\lambda/n, K}_0\left(\widebar{X}_1^{(n)}=j_\ell(n)/n=\widebar{X}_{[1,h]}^{(n)}\right)\rightarrow
\sum_{\ell:\:v_\ell>\delta}P^{\omega,\lambda, K}_0\left(\widebar{Z}_1=x_\ell=\widebar{Z}_{[1,h]}\right).\]
Clearly, by Lemma \ref{lemmaMaximumDisc}, the right-hand side here satisfies
\[\sum_{\ell:\:v_\ell>\delta}P^{\omega,\lambda, K}_0\left(\widebar{Z}_1=x_\ell=\widebar{Z}_{[1,h]}\right)\rightarrow P^{\omega,\lambda, K}_0\left(\widebar{Z}_1=\widebar{Z}_{[1,h]}\right),\]
as $\delta\rightarrow 0$. Moreover, as for the left-hand side, we have
\begin{eqnarray*}
	\lefteqn{\left|P^{\omega,\lambda/n, K}_0\left(\widebar{X}_1^{(n)}=\widebar{X}_{[1,h]}^{(n)}\right)-\sum_{\ell:\:v_\ell>\delta}P^{\omega,\lambda/n, K}_0\left(\widebar{X}_1^{(n)}=j_\ell(n)/n=\widebar{X}_{[1,h]}^{(n)}\right)\right|}\\
	&\leq&\sum_{i\not\in\{j_\ell(n):\:v_\ell>\delta\}}P^{\omega,\lambda/n, K}_0\left(\widebar{X}_1^{(n)}=i/n\right)\\
	&=&1-\sum_{\ell:\:v_\ell>\delta}P^{\omega,\lambda/n, K}_0\left(\widebar{X}_1^{(n)}=j_\ell(n)/n\right)\\
	&\rightarrow&1-\sum_{\ell:\:v_\ell>\delta}v_\ell
\end{eqnarray*}
as $n\rightarrow\infty$, where we have applied Proposition \ref{PropostionWallsPointProcess} to deduce the limit. Moreover, again appealing to Lemma \ref{lemmaMaximumDisc}, we see that the final expression converges to 0 as $\delta\rightarrow0$. This is enough to complete the proof of the result under the quenched measure when the coupling of environments is in place. Taking expectations with respect to the environment law yields the annealed result \eqref{eqn:newAgingInsideK}, as desired.
\end{proof}

\subsection{Aging under \eqref{RWT}} \label{SectionAgingRWT}

We still work under the coupling of Section \ref{SectionCopuledSpaces}. Let $\widetilde{\mathbb{P}}^{\lambda/n, K}$ be the annealed law of the random walk reflected at the boundary of the box $[-Kn, Kn]$ and $\widetilde{\mathbb{P}}^{\lambda, K}$ that of the corresponding diffusion reflected at the boundary of $[-K, K]$, with $K \in \mathbb{N}$.

\begin{proposition} \label{PropositionInsideKAging}
	Under the assumption \eqref{RWT}, with $\alpha_0, \alpha_\infty \in (0, 1)$, we have that, for all  $h > 1$,
\[
		\lim_{n \to \infty} \widetilde{\mathbb{P}}^{\lambda/n, K}\left( \big| \widetilde{X}_{b_n} - \widetilde{X}_{h b_n} \big| \le 1 \right) = \widetilde{\theta}^{K}(h) \coloneqq \widetilde{\mathbb{P}}^{\lambda, K} \left(\widetilde{Z}_{1} = \widetilde{Z}_{h} \right).\]
\end{proposition}

\begin{proof}[Proof of Proposition~\ref{PropositionInsideKAging}]
 Let us start by recalling some notation: $\widetilde{\rho}^{(n)}(1)$ denotes the quenched marginal law of $\widetilde{X}^{(n)}_1$ and $\widetilde{\rho}(1)$ is the quenched marginal law of $\widetilde{Z}^{\lambda}_1$. As $\widetilde{\rho}(1)$ is purely atomic, we can safely define the countable collection $\mathcal{A}_K$ of atoms $(x_\ell, v_\ell)$ such that $x_\ell\in[-K,K]$, $v_\ell \coloneqq P^{\omega, \lambda, K}(\widetilde{Z}_1 = x_\ell)$ and
\begin{equation}\label{eqn:TwoMarginalsWithTraps}
		\widetilde{\rho}(1) = \sum_{(x_\ell,v_\ell)\in\mathcal{A}_K} \delta_{x_\ell} v_\ell.
\end{equation}
We also define $v_i^{(n)} \coloneqq P^{\omega, \lambda, K}(\widetilde{X}_{b_n} = i/n)$, so that
\begin{equation}\label{eqn:TwoMarginalsWithTraps2}
		\widetilde{\rho}^{(n)}(1) = \sum_{i/n \in [-K, K]} \delta_{i/n} v^{(n)}_{i}.
\end{equation}
In the rest of the proof, we will drop the $K$  for notational simplicity, but we are still working on the environments and processes restricted to the finite boxes $[-Kn, Kn]$ and $[-K, K]$. Recall the definition \eqref{def:Tn} of the event $\mathcal{T}_n$. Let us also define
\[	\begin{split}
		u_{i, j}^{(n)} &\coloneqq P^{\omega, \lambda/n} \left( \widetilde{X}^{(n)}_{1 + h} = j/n | \widetilde{X}^{(n)}_{1} = i/n \right), \\
		u_{i, j} &\coloneqq P^{\omega, \lambda} \left( \widetilde{Z}_{1 + h} = x_j | \widetilde{Z}_{1} = x_i \right).
	\end{split}\]
By Proposition~\ref{PropositionMarginalWithTraps},  for every atom $(x_{\ell}, v_{\ell})\in\mathcal{A}_K$ of $\widetilde{\rho}(1)$ there exists a function $j_{\ell}(n)$ such that
\begin{equation*}
	v_{\ell} = \widetilde{P}^{\omega, \lambda}\left( \widetilde{Z}_1 = x_{\ell} \right) = \lim_{n \to \infty} P^{\omega, \lambda/n}\left( \widetilde{X}^{(n)}_1 \in \left\{ \frac{j_{\ell}(n)}{n}, \frac{j_{\ell}(n)}{n}+ 1/n \right\} \right).
\end{equation*}
Moreover, by Proposition~\ref{PropositionMarginalWithTraps}, $(x_\ell,v_\ell)$ is an atom of $\nu^{\alpha_\infty}$, hence it is not an atom of $\nu^{\alpha_0}$ by independence. Thus, for all $\ell$ and all $k$, we have that, almost surely,
\begin{align*}
	\left( u^{(n)}_{j_{\ell}(n), j_{k}(n)} + u^{(n)}_{j_{\ell}(n), j_{k}(n) + 1} \right) &\stackrel{n\to \infty}{\to} u_{\ell, k} \\
	\left( u^{(n)}_{j_{\ell}(n) + 1, j_{k}(n)} + u^{(n)}_{j_{\ell}(n) +1, j_{k}(n) + 1} \right) &\stackrel{n\to \infty}{\to} u_{\ell, k}
\end{align*}
Now, we aim to show that
\begin{equation}\label{goal00}
\widetilde{\mathbf{E}}\left[\sum_{i} v^{(n)}_{i} \left( u_{i, i}^{(n)} + u_{i, i + 1}^{(n)} + u_{i, i - 1}^{(n)} \right)\right]\stackrel{n\to\infty}{\longrightarrow}	\widetilde{\mathbf{E}}\left[\sum_{\ell} v_{\ell} u_{\ell, \ell}\right],
\end{equation}
which would imply the result. Let us denote $\widebar{u}_i^{(n)} \coloneqq \left( u_{i, i}^{(n)} + u_{i, i + 1}^{(n)} + u_{i, i - 1}^{(n)} \right)$. Using the observations above it is straightforward to notice that
\begin{equation}\label{eqn:LiminfAgingWithTraps}
	\begin{split}
		\sum_{\ell} v_{\ell} u_{\ell, \ell} &= \lim_{n \to \infty} \sum_{\ell} \left(v^{(n)}_{j_{\ell}(n)} \widebar{u}^{(n)}_{j_{\ell}(n)} + v^{(n)}_{j_{\ell}(n)+1} \widebar{u}^{(n)}_{j_{\ell}(n) + 1}\right) \\
		&\le \liminf_{n \to \infty} \sum_{i} v^{(n)}_{i} \left( u_{i, i}^{(n)} + u_{i, i + 1}^{(n)} + u_{i, i - 1}^{(n)} \right).
	\end{split}
\end{equation}
Recalling the notation from the statement of Lemma \ref{LemmaGoodSeparationIndep}, let us define the following sets
\begin{equation}\label{events}
	\begin{split}
		A_n(\delta) &\coloneqq \left\{  j \in T_n^{\alpha_\infty}:\: \exists \, |i - j| \le 1 \text{ such that } (v_j^{(n)} + v_i^{(n)})> \delta, \, \widebar{u}^{(n)}_j> \delta \right\}, \\
		B_n(\delta) &\coloneqq \left\{j : \widebar{u}_j^{(n)} \le \delta\right\}, \\
		A^v_n(\delta) &\coloneqq \left\{  j \in T_n^{\alpha_\infty}:\: \exists \, |i - j| \le 1 \text{ with } (v_j^{(n)} + v_i^{(n)})> \delta \right\}.
	\end{split}
\end{equation}
Notice that $A_n(\delta)\cup B_n(\delta)\cup\left(A_n^v(\delta)\right)^c$ contains all the indices. On the event $\mathcal{T}_n$, thanks to Proposition~\ref{PropositionMarginalWithTraps} and the fact that the number of terms in $A_n(\delta)$ is finite almost-surely, we get that, for Lebesgue almost-every $\delta$,
\begin{equation}\label{eqn:BoundSupAgingWithTraps1}
	\limsup_{n \to \infty} \sum_{i \in A_n(\delta)} v_i^{(n)} \widebar{u}_i^{(n)} \le \sum_{(x_\ell,v_\ell) \in A(\delta)} v_\ell u_{\ell, \ell},
\end{equation}
where $A(\delta)$ is the set of atoms $(x_\ell,v_\ell)\in\mathcal{A}_K$ such that $v_\ell>\delta$. Moreover
\[	\limsup_{n \to \infty} \sum_{i \in B_n(\delta)} v_i^{(n)} \widebar{u}_i^{(n)} \le \delta \limsup_{n \to \infty} \sum_{i \in B_n(\delta)} v_i \le \delta.\]
Finally, let us denote by $C_n(\delta)$ the complement of $A_n^{v}(\delta)$, then
\begin{equation}\label{eqn:BoundSupAgingWithTraps3}
	\begin{split}
		\limsup_{n \to \infty} \sum_{i \in C_n(\delta)} v_i^{(n)} \widebar{u}_i^{(n)} \le \limsup_{n \to \infty} \sum_{i \in C_n(\delta)} v^{(n)}_i \le 1 - \liminf_{n \to \infty} \sum_{i \in A^{v}_n(\delta)} v^{(n)}_i.
	\end{split}
\end{equation}
But, we also have that
\begin{equation*}
	\liminf_{n \to \infty} \sum_{i \in A^{v}_n(\delta)} v^{(n)}_i \ge \sum_{i: v_i > \delta} v_i,
\end{equation*}
and this last sum converges to $1$ as $\delta \to 0$ since $\widetilde{\rho}$ is purely atomic. By putting together \eqref{eqn:BoundSupAgingWithTraps1}-\eqref{eqn:BoundSupAgingWithTraps3} we get
\begin{equation}\label{eqn:LimsupAgingWithTraps}
	\limsup_{n \to \infty} \sum_{i} v^{(n)}_i \widebar{u_i} \le \sum_{\ell \in A(\delta)} v_i u_{i, i} + \delta + \left( 1 - \sum_{i: v_i > \delta} v_i \right).
\end{equation}
The bound from above follows by taking the limit as $\delta \to 0$. Using \eqref{eqn:LiminfAgingWithTraps} and  \eqref{eqn:LimsupAgingWithTraps}, one can prove \eqref{goal00} by applying  the dominated convergence theorem. This concludes the proof.
\end{proof}

 We now have all the tools and are able to prove the aging part of Theorem~\ref{theoremSubAging}.

\begin{proof}[Proof of Theorem~\ref{theoremSubAging}, Part I]
	It is immediate to notice that, for $n$ large,
	\begin{equation*}
		\begin{split}
			\widetilde{\mathbb{P}}^{\lambda/n, K}\left( \big| \widetilde{X}_{b_n} - \widetilde{X}_{h b_n} \big| \le 1, \,\, \tau_{Kn}^{X} \wedge \tau_{-Kn}^{X} > 2hb_n\right) = \widetilde{\mathbb{P}}^{\lambda/n}\left( \big| \widetilde{X}_{b_n} - \widetilde{X}_{h b_n} \big| \le 1, \,\, \tau_{Kn}^{X} \wedge \tau_{-Kn}^{X} > 2hb_n\right),
		\end{split}
	\end{equation*}
	and analogously
	\begin{equation*}
		\widetilde{\mathbb{P}}^{\lambda, K} \left(\widetilde{Z}_{1} = \widetilde{Z}_{h}, \,\, \tau_{K}^{Z} \wedge \tau_{-K}^{Z} > 2h \right) = \widetilde{\mathbb{P}}^{\lambda} \left(\widetilde{Z}_{1} = \widetilde{Z}_{h}, \,\, \tau_{K}^{Z} \wedge \tau_{-K}^{Z} > 2h \right).
	\end{equation*}
	So, by Proposition~\ref{PropositionInsideKAging}, we have, for every $K \in \mathbb{N}$ and every $h \ge 1$,
	\begin{align}
		\limsup_{n \to \infty}\,\, &\widetilde{\mathbb{P}}^{\lambda/n}\left( \big| \widetilde{X}_{b_n} - \widetilde{X}_{h b_n} \big| \le 1 \right) \nonumber \\
		& \le \widetilde{\mathbb{P}}^{\lambda} \left(\widetilde{Z}_{1} = \widetilde{Z}_{h} \right) + \widetilde{\mathbb{P}}^{\lambda}\left( \tau_{K}^{Z} \wedge \tau_{-K}^{Z} \le 2h \right) + \limsup_{n} \,\,\widetilde{\mathbb{P}}^{\lambda/n}\left( \tau_{Kn}^{X} \wedge \tau_{-Kn}^{X} \le 2hb_n \right).\nonumber
	\end{align}
By \eqref{eqn:2HittingTimesComparedRWT} in Lemma~\ref{Lemma:HittingKComparison} we have that
\begin{equation*}
	\limsup_{n} \,\,\widetilde{\mathbb{P}}^{\lambda/n}\left( \tau_{Kn}^{X} \wedge \tau_{-Kn}^{X} \le 2hb_n \right) \le \widetilde{\mathbb{P}}^{\lambda}\left( \tau_{K - 1}^{Z} \wedge \tau_{-K + 1}^{Z} \le 2h +1 \right).
\end{equation*}
Thus we have that
	\begin{align}
		\limsup_{n \to \infty}\,\,\widetilde{\mathbb{P}}^{\lambda/n}\left( \big| \widetilde{X}_{b_n} - \widetilde{X}_{h b_n} \big| \le 1 \right) \le \widetilde{\mathbb{P}}^{\lambda} \left(\widetilde{Z}_{1} = \widetilde{Z}_{h} \right) + 2\widetilde{\mathbb{P}}^{\lambda}\left( \tau_{K - 1}^{Z} \wedge \tau_{-K + 1}^{Z} \le 2h +1 \right).\nonumber
	\end{align}
	By reasoning in the same way, one can obtain
	\begin{equation*}
		\liminf_{n \to \infty}\,\, \widetilde{\mathbb{P}}^{\lambda/n}\left( \big| \widetilde{X}_{b_n} - \widetilde{X}_{h b_n} \big| \le 1 \right) \ge \widetilde{\mathbb{P}}^{\lambda}\left(\widetilde{Z}_{1} = \widetilde{Z}_{h} \right) - 2\widetilde{\mathbb{P}}^{\lambda}\left( \tau_{K - 1}^{Z} \wedge \tau_{-K + 1}^{Z} \le 2h +1\right).
	\end{equation*}
The conclusion follows by taking $K$ to infinity and applying Lemma~\ref{LemmaExitSlowlyKBox}.
\end{proof}

\section{Proof of the sub-aging result under the assumption \eqref{RWT}}\label{sec:subagingproof}

At the end of this section, we prove the second part of Theorem~\ref{theoremSubAging}. We will need several tools before being able to prove the main result. The crucial step is the following proposition; its proof is the core of this section. We still work under the coupling of Section \ref{SectionCopuledSpaces}. Recall that $\widetilde{\mathbb{P}}^{\lambda/n, K}$ is the annealed law of the random walk reflected at the boundary of the box $[-Kn, Kn]$ and $\widetilde{\mathbb{P}}^{\lambda, K}$ that of the corresponding diffusion reflected at the boundary of $[-K, K]$, with $K \in \mathbb{N}$.

\begin{proposition} \label{PropositionInsideKSubaging}
	Under the assumption \eqref{RWT}, with $\alpha_0, \alpha_\infty \in (0, 1)$, we have that, for all  $h > 0$,
	\[		\lim_{n \to \infty} \widetilde{\mathbb{P}}^{\lambda/n, K}\left( \big| \widetilde{X}_{b_n + s_1 d_{n, \infty}} - \widetilde{X}_{b_n +s_2 d_{n, \infty}} \big| \le 1, \,\, \forall s_1, s_2 \in [0, h] \right) = \widebar{\theta}(h) \coloneqq \widetilde{\EVal}^{\lambda, K}\left[ e^{- h\frac{A^0 + A^2}{2A^1}}\right],\]
	where $A^0, A^1, A^2$ are such that $A^1 \stackrel{(\mathrm{d})}{=} \nu^{\alpha_\infty}(\widetilde{Z}^{\lambda}_1)$ and $A^0, A^2$ are distributed as independent conductances under $\widetilde{\mathbf{P}}$, independent of $\widetilde{Z}^{\lambda}$. In this statement $\nu^{\alpha_\infty}$ is restricted to $[-K, K]$ and $\widetilde{Z}^{\lambda}_1$ is reflected at the boundary.
\end{proposition}

We postpone the proof of the result as we will first need to establish several preliminary results. Let us assume the construction of Section~\ref{SectionCopuledSpaces}. Recall the notation defined in \eqref{eqn:TwoMarginalsWithTraps} and \eqref{eqn:TwoMarginalsWithTraps2}, and note that  each $x_\ell$ is both an atom of $\widetilde{\rho}(1)$ with weight $v_\ell$ and an atom of $\nu^{\alpha_\infty}$ with weight $\nu^{\alpha_\infty}(x_\ell)$.  Also, recall that there exists $j_\ell(n)$ such that $j_\ell(n)/n$ converges to $x_\ell$ and for which the masses of the relevant discrete measures converge (see the proof of the following result for details). To ease the notation, we will drop the superscript $K$ in the proofs but we will still work with the restricted processes.

Set for simplicity $c^{(n)}(i) = d_{n, \infty}^{-1} c(i)$. Under Assumption~\eqref{RWT}, let us define the following two measures
\[	\begin{split}
		\pi^{(n)}(dx) &\coloneqq \sum_{i/n \in [-K, K]} \delta_{c^{(n)}(i)} v_i^{(n)}, \\
		\pi(dx) &\coloneqq \sum_{(x_\ell,v_\ell)\in\mathcal{A}_K} \delta_{\nu^{\alpha_\infty}(x_\ell)} v_\ell.
	\end{split}\]

\begin{lemma}\label{LemmaConvergenceDepthMeas}
	Under Assumption~\eqref{RWT}, for almost every realisations of the environment, as $n \to \infty$
	\begin{equation*}
		\pi^{(n)}(dx) \stackrel{v}{\to} \pi(dx).
	\end{equation*}
	Moreover, for any atom $(\nu^{\alpha_\infty}(x_\ell), v_\ell) \in \pi$ there exists an index $j_\ell(n)$ such that, almost-surely, as $n \to \infty$
	\begin{align*}
		c^{(n)}\left( j_\ell(n) \right) &\to \nu^{\alpha_\infty}(x_{\ell}),\\
		c^{(n)}\left( j_\ell(n) + 1 \right) &\to \nu^{\alpha_\infty}(x_{\ell}), \\
		v^{(n)}_{j_\ell(n)} + v^{(n)}_{j_\ell(n) + 1} &\to v_\ell.
	\end{align*}
\end{lemma}

\begin{proof}
	Let us start by proving the second part of the lemma. By Proposition  \ref{PropositionMarginalWithTraps}, we have that for each atom $(x_{\ell}, v_\ell)$, there exists a function $j_{\ell}(n)$ such that
	\begin{align*}
		d_{n, \infty}^{-1} c\left( \{j_\ell(n), j_\ell(n) + 1\} \right) &\to \nu^{\alpha_\infty}(x_\ell),\\
		v^{(n)}_{j_\ell(n)} + v^{(n)}_{j_\ell(n) + 1} &\to v_\ell.
	\end{align*}
	On the event $\mathcal{T}_n$ of Lemma~\ref{LemmaGoodSeparationIndep} we have that, almost-surely,
	\begin{equation*}
		d_{n, \infty}^{-1}\left(c\left( \{j_\ell(n) + 1, j_\ell(n) + 2\} \right) + c\left( \{j_\ell(n) - 1, j_\ell(n)\} \right) \right) \to 0,
	\end{equation*}
	which implies the second part of the lemma. Let us now prove vague convergence. Recall the definition \eqref{events} of the set $A^{v}_n(\delta)$. For $f$ a continuous and non-negative function on $[-K, K]$, we have that
	\begin{align*}
		\int f(x) \pi^{(n)} (dx) &= \sum_{i \in A_n^{v}(\delta)} v_{i}^{(n)} f\left( c^{(n)}(i) \right) + \sum_{i \not\in A_n^{v}(\delta)} v_{i}^{(n)} f\left( c^{(n)}(i) \right).
	\end{align*}
	As $\{\ell: v_\ell>\delta\}$ is almost-surely finite,  the observation above yields
	\begin{equation*}
		\lim_{n \to \infty} \sum_{i \in A_n^{v}(\delta)} v_{i}^{(n)} f\left( c^{(n)}(i) \right) = \sum_{\ell: v_\ell>\delta} v_{\ell} f\left( \nu^{\alpha_\infty}(x_\ell)\right)
	\end{equation*}
for Lebesgue almost-every $\delta$. The right hand side is monotone and bounded by $\|f\|_\infty$, thus as $\delta \to 0$ it converges to $\int f(x) \pi(dx)$. Analogously, we obtain that
	\begin{align*}
		\limsup_{n\to \infty} \sum_{i \not\in A_n^{v}(\delta)} v_{i}^{(n)} f\left( c^{(n)}(i) \right) &\le \|f\|_\infty \left( 1 - \liminf_{n \to \infty} \sum_{i \in A_n^{v}(\delta)} v_{i}^{(n)}  \right) \\
		&\le \|f\|_\infty \left( 1 - \sum_{\ell: v_\ell>\delta} v_{\ell} \right).
	\end{align*}
	We conclude by taking the limit as $\delta\to 0$ since $\widetilde{\rho}(1)$ is almost surely a purely atomic measure.
\end{proof}

\begin{proposition} Recall the definition \eqref{InvariantMeasure}, then, for all points of continuity $u$ of the right-hand side,
\begin{equation*}
	\lim_{n \to \infty} \widetilde{\mathbf{E}}\left[ \widetilde{P}^{\omega, \lambda/n, K}\left( d_{n, \infty}^{-1} c\left(\widetilde{X}_{b_n} \right) \le u \right)\right] = \widetilde{\mathbf{E}}\left[ \widetilde{P}^{\omega, \lambda, K}\left( \nu^{\alpha_\infty}\left(\widetilde{Z}_1\right) \le u \right) \right].
\end{equation*}
\end{proposition}

\begin{proof} Note that
\begin{align*}
	\widetilde{\mathbf{E}}\left[ \widetilde{P}^{\omega, \lambda/n}\left( c^{(n)}\left(\widetilde{X}_{b_n}\right)\le u \right)\right] &= 1 - \widetilde{\mathbf{E}}\left[ \widetilde{P}^{\omega, \lambda/n}\left( c^{(n)}\left(\widetilde{X}_{b_n}\right) > u \right) \right]\\
	& = 1 - \widetilde{\mathbf{E}}\left[ \sum_{c^{(n)}(i) > u} v_i^{(n)} \right].
\end{align*}
However on our coupling, by Lemma~\ref{LemmaConvergenceDepthMeas} and using that $\{\ell: \nu^{\alpha_\infty}(x_\ell)>u\}$ is almost-surely finite, we obtain that
\begin{equation*}
	\lim_{n\to \infty} \sum_{c^{(n)}(i) > u} v_i^{(n)} = \sum_{\nu^{\alpha_\infty}(x_{\ell}) > u} v_\ell,
\end{equation*}
as long as $u \ne \nu^{\alpha_\infty}(x_{\ell}), \forall \ell$, which is true for Lebesgue almost-every $u$. Note that this is true otherwise the vague convergence of Lemma~\ref{LemmaConvergenceDepthMeas} would be violated. We conclude the proof by applying the dominated convergence theorem.
\end{proof}

Let us introduce the following set:
\[	N_n^{\alpha_\infty} \coloneqq \left\{i: i/n \in [-K, K], \,\, i \not\in T_n^{\alpha_\infty}, \,\, \exists j \in T_n^{\alpha_\infty} \text{ such that } |i-j| = 1\right\}.\]
By construction, the conductances $\{c(\{i, i + 1\})\}_{i \in N_n^{\alpha_\infty}}$ are i.i.d.\ and distributed as $c(\{0, 1\})$ conditional on $c(\{0, 1\}) \le d_{n, \infty}^{1 - \widehat{\delta}}$. Let us also define a family of i.i.d.\ random variables $\{\widehat{c}_i^{(n)}\}_{i \in \Z}$ distributed like $c(\{0, 1\})$ conditional on $c(\{0, 1\}) \le d_{n, \infty}^{1 - \widehat{\delta}}$. The next lemma guarantees that re-sampling the conductances in $N_n^{\alpha_\infty}$ does not affect the almost-sure convergence of Proposition~\ref{PropositionPoissonPoint}.

\begin{lemma}\label{LemmaNeighboursSmall}
The following limits hold almost-surely
\begin{equation}\label{eqn:NeighSmall1}
	\frac{1}{d_{n, \infty}}\sum_{i \in N_n^{\alpha_\infty}} c(\{i, i + 1\}) \stackrel{n \to \infty}{\to} 0,  \qquad  \frac{1}{d_{n, 0}}\sum_{i \in N_n^{\alpha_\infty}} \frac{1}{c(\{i, i + 1\})} \stackrel{n \to \infty}{\to} 0
\end{equation}
and
\begin{equation}\label{eqn:NeighSmall2}
	\frac{1}{d_{n, \infty}}\sum_{i \in [-2K\ceil{n^{3/4}}, 2K\ceil{n^{3/4}}]} \widehat{c}_i^{(n)} \stackrel{n \to \infty}{\to} 0, \qquad \frac{1}{d_{n, 0}}\sum_{i \in [-2K\ceil{n^{3/4}}, 2K\ceil{n^{3/4}}]} \frac{1}{\widehat{c}_i^{(n)}} \stackrel{n \to \infty}{\to} 0.
\end{equation}
\end{lemma}
\begin{proof}
Firstly, let us focus on the terms in \eqref{eqn:NeighSmall1}. On the event $\mathcal{T}_n$, defined at \eqref{def:Tn}, we get that almost-surely
\begin{equation*}
	\sum_{i \in N_n^{\alpha_\infty}} c(\{i, i + 1\}) \le \sum_{i \not \in T_n^{\alpha_\infty}} c(\{i, i + 1\}) \quad \textnormal{and} \quad\sum_{i \in N_n^{\alpha_\infty}} \frac{1}{c(\{i, i + 1\})} \le \sum_{i \not \in T_n^{\alpha_0}} \frac{1}{c(\{i, i + 1\})},
\end{equation*}
for all $n$ large enough. This observation makes the proof symmetric for the conductances and the resistances. Let us just present the first one. One may also note that for all $\delta > 0$ and all $n$ large enough the indices $i \not \in T_n^{\alpha_0}$ are contained in $I^{(n), \alpha_\infty}_0 \backslash I^{(n), \alpha_\infty}_\delta$ (restricted to $[-K, K]$), as defined in \eqref{eqn:SetLargeIncrements}. Using equations \eqref{eqn:SmallDiscrete}, \eqref{eqn:SmallAtomsAreSmall} and \eqref{eqn:VerySmallDiscrete} we get that for all $\varepsilon>0$ and all $n$ large enough
\begin{equation*}
	\frac{1}{d_{n, \infty}} \sum_{i \not \in T_n^{\alpha_\infty}} c(\{i, i + 1\}) \le \varepsilon,
\end{equation*}
which concludes the proof of \eqref{eqn:NeighSmall1}.
	
Let us now prove \eqref{eqn:NeighSmall2}. First, we dominate the sums that appear in the statements with the sum of conductances (respectively resistances) that are not conditioned. This is needed because we aim to use the monotonicity trick already used in the proof of Lemma~\ref{LemmaGoodSeparationIndep}. The conditioning creates a problem in this case because as $n$ increases the conductances have more room to be large.
For every $\widehat{c}_{i}^{(n)}$, we can find a coupling with a $\widetilde{c}_{i}$ which is distributed as a standard conductance and the coupling is such that $\widetilde{c}_{i} \ge\widehat{c}_{i}^{(n)}$. In particular, we can suppose,  for all $n$ and $i$,
\[	\widetilde{c}_{i} = B_{i}^{(n)} \widehat{c}_{i}^{(n)} + (1 - B_{i}^{(n)}) \widebar{c}_{i}^{(n)} \ge \widehat{c}_{i}^{(n)},\]
where $\widebar{c}_{i}^{(n)}$ is distributed as $c(\{0, 1\})$ conditional on $c(\{0, 1\}) > d_{n, \infty}^{1 - \widehat{\delta}}$, and $B_{i}^{(n)}$ is a Bernoulli random variable (independent of $\widehat{c}_{i}^{(n)}$ and $\widebar{c}_{i}^{(n)}$) with parameter $\mathbf{P}(c(\{0, 1\}) \le d_{n, \infty}^{1 - \widehat{\delta}})$. Thus, in order to prove \eqref{eqn:NeighSmall2}, it will suffice to show that
\begin{equation*}
	\lim_{n} \frac{1}{d_{n, \infty}}\sum_{j = 1}^{4K\ceil{n^{3/4}}} \widetilde{c}_{j} = 0, \quad \mathbf{P}\text{-a.s..}
\end{equation*}
For the sum of the inverses appearing in the statement, we can follow the same procedure, defining a family $\{\widetilde{r}_i\}_{i = 1}^{4K\ceil{n^{3/4}}}$, where $\widetilde{r}_1$ is distributed as $r(\{0, 1\})$ conditional on $r(\{0, 1\}) \ge 1$; this conditioning is necessary because $1/\widehat{c}_1^{(n)}$ is conditioned on $\widehat{c}_1^{(n)} \le d_{n, \infty}^{1 - \widehat{\delta}}$. Let us define the events
\[		\widehat{C}_n^{\alpha_\infty} = \left\{ \sum_{j = 1}^{4K\ceil{n^{3/4}}} \widetilde{c}_{i} \le d_{n, \infty}^{1 - \widehat{\delta}/2} \right\}\text{ and } \widehat{C}_n^{\alpha_0} = \left\{ \sum_{j = 1}^{4K\ceil{n^{3/4}}} \widetilde{r}_{i} \le d_{n, 0}^{1 - \widehat{\delta}/2}   \right\}.\]
We want to apply the Fuk-Nagaev inequality \cite[Theorem 5.1]{BergerFukNagaev}, which gives the following property of random variables with regularly varying tails. Let $\mathscr{S}(m) \coloneqq \sum_{i = 1}^m X_i$ and $\mathscr{M}(m) \coloneqq \max_{i \in \{1,\cdots,m\}} X_i$, then there exists a constant $c>0$ such that, for all $y \le x$,
\begin{equation}\label{FN}
	P \left( \mathscr{S}(m) > x, \, \mathscr{M}(m) < y \right) {\le} \left( c m \frac{y}{x} L(y) y^{-\gamma}\right)^{x/y}.
\end{equation}
Let us use \eqref{FN} with $x = d_{n, \infty}^{1 - \widehat{\delta}/2}$, $y = d_{n, \infty}^{1 - \widehat{\delta}}$ and $m = 4K n^{3/4}$, and recall that a slowly varying function is eventually smaller than any polynomial. This implies that there exists $\nu>0$ such that, for $n$ large enough,
\begin{equation*}
	\widetilde{\mathbf{P}}\left( (\widehat{C}_n^{\alpha_\infty})^c \right) \le n^{-\nu} \text{ and }
	\widetilde{\mathbf{P}}\left( (\widehat{C}_n^{\alpha_0})^c \right) \le n^{-\nu}.
\end{equation*}
We can define the event
\begin{equation*}
	\widehat{\mathcal{C}}_\ell^{\alpha_\infty} \coloneqq \sum_{j = 1}^{4K\ceil{(2\ell)^{3/4}}} \widetilde{c}_{i} \le d_{\ell, \infty}^{1 - \widehat{\delta}/2},
\end{equation*}
and applying again \eqref{FN} implies that $\mathbf{P}(\widehat{\mathcal{C}}_\ell^{\alpha_\infty}) \ge 1 - \ell^{-\nu}$. We can now apply the monotonicity trick already used in the proof of Lemma~\ref{LemmaGoodSeparationIndep} by noticing that $ (\widehat{C}_n^{\alpha_\infty})^c \subseteq (\widehat{\mathcal{C}}_\ell^{\alpha_\infty})^c$ for all $n = \ell, \dots, 2\ell$, and get that $\widehat{C}_n^{\alpha_\infty}$ happens almost-surely for all $n$ large enough. The proof for the event $\widehat{C}_n^{\alpha_0}$ follows the same lines. We wish to highlight the fact that the application of the Fuk-Nagaev inequality is not affected by the conditioning of the $\widetilde{r}_i$ because it simply multiplies the tail probability by a constant.
\end{proof}

For simplicity, let us set the notation $c_i = c(\{i, i+1\})$ in what follows. Let us denote by $\widehat{\omega}$ the environment induced by substituting the variables $\{c_i\}_{i \in N_n^{\alpha_\infty}}$ with the conductances $\{\widehat{c}_i\}_{i \in \Z}$ (and $\widehat{\mathbf{P}}$ its law). Note that the distribution of $\widehat{\omega}$ is the same as the one of $\omega$. Using Lemma~\ref{LemmaNeighboursSmall}, we could replicate the procedure of Section~\ref{SectionCopuledSpaces} and get that Proposition~\ref{prop:QuenchedWeakConvergence} and Proposition~\ref{PropositionMarginalWithTraps} would still hold. Crucially, we also obtain that, for any atom $(x_{\ell}, v_{\ell}) \in \text{supp }\widetilde{\rho}(1)$, that, almost surely
\begin{equation}\label{eqn:BothEnvirnomentsConverge}
	\begin{split}
		\widetilde{P}^{\omega, \lambda}\left( \widetilde{Z}_{1} = x_{\ell} \right) &=\lim_{n \to \infty} \widetilde{P}^{\omega, \lambda/n}\left( \widetilde{X}_{b_n} \in \{j_{\ell}(n), j_{\ell}(n) + 1\} \right) \\
		&= \lim_{n \to \infty} \widetilde{P}^{\widehat{\omega}, \lambda/n}\left( \widetilde{X}_{b_n} \in \{j_{\ell}(n), j_{\ell}(n) + 1\} \right).
	\end{split}
\end{equation}
Recall $A^{v}(\delta) = \{\ell: v_\ell > \delta\}$, let $A^{v}_n(\delta)$ be its discrete counterpart in the environment $\omega$ (see \eqref{eqn:BoundSupAgingWithTraps3}) and $A^{\widehat{v}}_n(\delta)$ the one in the environment $\widehat{\omega}$.

Let us define the random variable $T \coloneqq \inf\{t \ge 0 : |\mathrm{range}(\widetilde{X}_t)| > 2\}$. Recall $\theta_t$ is the canonical time shift by $t$, then
\[	T^{(n)} \coloneqq \frac{1}{d_{n, \infty}} T \circ \theta_{b_n}.\]
Let us fix some further notation, let $\mathcal{L}_X(\cdot)$ be the Laplace transform of the random variable $X$. In particular, we recall that, for $\xi>0$,
\[\mathcal{L}_{\exp\left(\lambda \right)}\left(  \xi \right)=\frac{\lambda}{\lambda+\xi},\]
where $\exp(\lambda)$ here denotes an exponential random variable with parameter $\lambda$. Furthermore, let $d\widetilde{\mathbf{P}}(\cdot)$ be the measure associated with the distribution function $\widetilde{\mathbf{P}}(c_0 \le t)$. Finally, let $d\widetilde{F}$ be the measure associated with the distribution function $\widetilde{\mathbb{P}}^{K}(\nu^{\alpha_\infty}(\widetilde{Z}^{\lambda}_1) \le t)$. The following proposition aims to show that the distribution of the random variable $T^{(n)}$ converges to an exponential distribution of parameter with the correct parameter (in the sense of Proposition~\ref{PropositionInsideKSubaging}).

\begin{proposition}\label{PropositionLaplaceTransform}
	For every $\xi>0$,
	\begin{equation*}
		\widetilde{\mathbb{E}}^{\lambda/n, K}\left[ e^{- \xi T^{(n)}} \right] \stackrel{n\to \infty}{\to} \widetilde{\mathbb{E}}^{\lambda, K} \left[ \mathcal{L}_{\exp\left( \frac{A^0 + A^2}{2 A^1} \right)}\left(  \xi \right)\right],
	\end{equation*}
where $A_0$, $A_1$ and $A_2$ are defined as in Proposition~\ref{PropositionInsideKSubaging}. More explicitly, we have
\begin{equation*}
	\widetilde{\mathbb{E}}^{\lambda, K} \left[ \mathcal{L}_{\exp\left( \frac{A^0 + A^2}{2 A^1} \right)}\left(  \xi \right)\right] = \int_{0}^{\infty}  \int_{0}^\infty\int_{0}^\infty\int_{0}^\infty e^{-\xi s} \frac{t_1 + t_2}{2 u} e^{-s \frac{t_1 + t_2}{2 u}} ds d\widetilde{\mathbf{P}}(t_1) d\widetilde{\mathbf{P}}(t_2) d\widetilde{F}(u).
\end{equation*}
\end{proposition}

Before proving this result we need to show that the random variable $T^{(n)}$ is well-approximated by an exponential random variable whose parameter depends on the discrete environment.

\begin{lemma}\label{LemmaQuenchedLaplaceTransform}
For every $\xi>0$, for all $\delta>0$, for almost all realisations of $\omega$ we have that
\begin{flalign*}
\widehat{\mathbf{E}}&\left[ \sum_{\ell \in A^{v}(\delta)} \bigg| \widehat{v}^{(n)}_{j_{\ell}(n)} \widetilde{E}^{\widehat{\omega}, \lambda/n}_{j_\ell(n)} \left[ e^{-\xi \frac{T}{d_{n, \infty}}} \right] + \widehat{v}^{(n)}_{j_{\ell}(n) + 1} \widetilde{E}^{\widehat{\omega}, \lambda/n}_{j_\ell(n) + 1} \left[ e^{-\xi \frac{T}{d_{n, \infty}}} \right]- v_{\ell}\mathcal{L}_{\exp\left(\frac{\widehat{c}_{j_\ell(n) - 1} + \widehat{c}_{j_\ell(n) + 1}}{2\nu^{\alpha_\infty}(x_\ell)} \right)}\left(  \xi \right)\bigg| \Big| \omega \right]
\end{flalign*}
converges to 0 as $n\rightarrow\infty$.
\end{lemma}

\begin{proof} Note that, for all $n$ large enough, $j_{\ell}(n) \in T_n^{\alpha_\infty}$, which implies $\{j_{\ell}(n) - 1, j_{\ell}(n) + 1\} \in N_n^{\alpha_\infty}$, Then $\widehat{c}_{i} \le d_{n, \infty}^{1 - \widehat{\delta}}$ for $i = \{j_{\ell}(n) - 1, j_{\ell}(n) + 1\}$. Furthermore, we have that $\widehat{c}_{j_{\ell}(n) - 1}, \widehat{c}_{j_{\ell}(n) + 1}$ are independent of $\omega$. We already know that
\begin{equation}\label{eqn:FactsBeginningProof}
	\widehat{v}_{j_{\ell}(n)} + \widehat{v}_{j_{\ell}(n) + 1}\to v_\ell \quad \text{and} \quad d_{n, \infty}^{-1}\widehat{c}_{j_{\ell}(n)} = d_{n, \infty}^{-1} c_{j_{\ell}(n)} \to \nu^{\alpha_\infty}(x_{\ell}).
\end{equation}
For simplicity of notation, let us set $j_\ell(n) = 0$, the general case being an easy adaptation. Let us define $T^{+} \coloneqq \inf\{t \ge 0: \widetilde{X}_t \not \in \{0,  1\}\}$ and  $p_n^*(0) = \widetilde{P}^{\widehat{\omega}, \lambda/n}_{0}(\tau_{-1} < \tau_{1})$. Observe that, under the quenched law,
\[\widetilde{E}_x^{\widehat{\omega}, ^{\lambda/n}} \left[ e^{-\xi \frac{T^{+}}{d_{n, \infty}}} \right]- p_n^*(x)\le 	\widetilde{E}^{\widehat{\omega}, ^{\lambda/n}}_x \left[ e^{-\xi \frac{T}{d_{n, \infty}}} \right] \le  \widetilde{E}_x^{\widehat{\omega}, ^{\lambda/n}} \left[ e^{-\xi \frac{T^{+}}{d_{n, \infty}}} \right]+ p_n^*(x)\]
Start by noting that
\begin{equation*}
	p_n^*(0) \le n^{-\widehat{\delta}} \to 0 \quad \mathrm{as} \,\, n \to \infty.
\end{equation*}
It remains to control the Laplace transform of $T^+$. Define
\begin{equation*}
p_1 = \frac{\widehat{c}^{\lambda/n}_{-1}}{\widehat{c}^{\lambda/n}_{-1} + c^{\lambda/n}_0}\text{,  } p_{2} = \frac{\widehat{c}^{\lambda/n}_{1}}{\widehat{c}^{\lambda/n}_{1} + c^{\lambda/n}_{0}}\text{ and } p_n = 1 - (1 - p_1)(1-p_2).
 \end{equation*}
Under the measure $\widetilde{P}^{\widehat{\omega}, \lambda/n}_{0}$, we have that
\begin{equation*}
	\sum_{i = 1}^{Y^* } \left( e_{2i - 1} + e_{2i} \right) \preccurlyeq T^{+} \preccurlyeq \left(\sum_{i = 1}^{Y^* }
	 \left(e_{2i - 1} + e_{2i}\right) + e_0 \right),
\end{equation*}
where $\preccurlyeq$ denotes stochastic domination, $\{e_i\}_{i \ge 0}$ is a family of i.i.d.~exponential random variables of mean $1$ (independent of everything else),  $Y_1$ and $Y_2$ are geometric random variables of parameters $p_1$ and $p_2$ respectively, and $Y^*= \min\{Y_1, Y_2\} \sim \mathrm{Geom}(p_n)$. All these geometric random variables take values in $\{0,1,2,\dots\}$.
We can discard $e_0$ in the sum as $d_{n, \infty}^{-1} e_0 \to 0$ in probability. A straightforward computation yields that
\begin{equation*}
	p_n  =\frac{c^{\lambda/n}_0(\widehat{c}^{\lambda/n}_{-1} + \widehat{c}^{\lambda/n}_{1})}{(c^{\lambda/n}_0 + \widehat{c}^{\lambda/n}_{-1})(c^{\lambda/n}_0 + \widehat{c}^{\lambda/n}_{1})} + \frac{\widehat{c}^{\lambda/n}_{-1} \widehat{c}^{\lambda/n}_{1}}{(c^{\lambda/n}_0 + \widehat{c}^{\lambda/n}_{-1})(c^{\lambda/n}_0 + \widehat{c}^{\lambda/n}_{1})}.
\end{equation*}
Furthermore, by using the exact form of the Laplace transform of a geometric sum of i.i.d.~random variables we get that
\begin{equation*}
	\widetilde{E}_{0}^{\widehat{\omega}, \lambda/n} \left[ e^{-\xi \frac{T^{+}}{d_{n, \infty}}} \right] = \frac{p_n d_{n, \infty} \left(\frac{\xi + d_{n, \infty}}{d_{n, \infty}}\right)^2}{p_n d_{n, \infty} + 2 \xi + \xi^2 d_{n, \infty}^{-1}} + o^{(n)}(1).
\end{equation*}
We now focus on showing that
\begin{equation} \label{eqn:AlmostExponentialLaplaceTransform}
	\left|\frac{p_n d_{n, \infty} \left(\frac{\lambda + d_{n, \infty}}{d_{n, \infty}}\right)^2}{p_n d_{n, \infty} + 2 \lambda + \lambda^2 d_{n, \infty}^{-1}} - \mathcal{L}_{\exp\left(\frac{\widehat{c}_{ - 1} + \widehat{c}_{1}}{2\nu^{\alpha_\infty}(x_\ell)} \right)}\left(  \xi \right) \right| \stackrel{n \to \infty}{\to} 0.
\end{equation}
Let us make the following observations, using the facts which were stated in \eqref{eqn:FactsBeginningProof} and immediately above that:
\begin{enumerate}
	\item $ \displaystyle \left|p_n d_{n, \infty} -  \frac{\widehat{c}_{- 1} + \widehat{c}_{1}}{\nu^{\alpha_\infty}(x_\ell)} \right| = \left(\widehat{c}_{- 1} + \widehat{c}_{ + 1}\right) o^{(n)}(1)$,
	\item $\displaystyle \left| \left(\frac{\xi + d_{n, \infty}}{d_{n, \infty}}\right)^2 - 1 \right| = o^{(n)}(n^{-1})$,
	\item $\displaystyle \left| p_n d_{n, \infty} + 2 \xi + \xi^2 d_{n, \infty}^{-1} - \frac{\widehat{c}_{- 1} + \widehat{c}_{ + 1}}{\nu^{\alpha_\infty}(x_\ell)} - 2 \xi \right| = \left(\widehat{c}_{- 1} + \widehat{c}_{ + 1}\right) o^{(n)}(1)$.
\end{enumerate}
By plugging these estimates in to the left-hand side of \eqref{eqn:AlmostExponentialLaplaceTransform} we get that it is bounded from above by
\begin{equation*}
	\frac{\left( \frac{\widehat{c}_{- 1} + \widehat{c}_{ 1}}{\nu^{\alpha_\infty}(x_\ell)} + 2\xi \right) o^{(n)}(1)}{\left( \frac{\widehat{c}_{- 1} + \widehat{c}_{ 1}}{\nu^{\alpha_\infty}(x_\ell)} + 2\xi \right)^2 (1 - o^{(n)}(1))},
\end{equation*}
and this quantity goes to $0$ as $n \to \infty$ for all $\xi>0$. Observe that, thus far we showed that
\begin{equation*}
	\limsup_{n} \left| \widetilde{E}^{\widehat{\omega}, \lambda/n}_{0} \left[ e^{-\xi \frac{T}{d_{n, \infty}}} \right] - \mathcal{L}_{\exp\left(\frac{\widehat{c}_{ - 1} + \widehat{c}_{ 1}}{2\nu^{\alpha_\infty}(x_\ell)} \right)}\left(  \xi \right)\right| = 0.
\end{equation*}
By mirroring this argument we also get that
\begin{equation*}
	\limsup_{n} \left| \widetilde{E}^{\widehat{\omega}, \lambda/n}_{1} \left[ e^{-\xi \frac{T}{d_{n, \infty}}} \right] - \mathcal{L}_{\exp\left(\frac{\widehat{c}_{ - 1} + \widehat{c}_{ 1}}{2\nu^{\alpha_\infty}(x_\ell)} \right)}\left(  \xi \right)\right| = 0.
\end{equation*}
This, together with $\widehat{v}_{j_{\ell}(n)} + \widehat{v}_{j_{\ell}(n) + 1}\to v_\ell$, is enough to conclude that each term in the sum over $A^{v}(\delta)$ goes to $0$ almost surely. However, since we also know that $A^{v}(\delta)$ is almost-surely finite, this observation and an application of the dominated convergence theorem finishes the proof.
\end{proof}

\begin{lemma}\label{LemmaWeakCVNeighbours}
For all $\delta>0$, conditional on $\omega$, if we define $A^{v}(\delta) = \{\ell: v_\ell > \delta\}$, the collection
	\begin{equation*}
		\left(\{\widehat{c}_{j_\ell(n)-1}, \widehat{c}_{j_\ell(n) + 1} \}\right)_{\ell \in A^{v}(\delta) }
	\end{equation*}
converges in distribution to an i.i.d.\ collection of random variables distributed as
$\{c_0, c_1\}$, where $c_0, c_1$ are distributed as two independent conductances under $\widetilde{\mathbf{P}}$.
\end{lemma}

\begin{proof} For almost all realisations of $\omega$, we have that, for $n$ large enough, $\{j_\ell(n)-1, j_\ell(n) + 1 \} \in N_n^{\alpha_\infty}$ for all $\ell \in A^{v}(\delta)$, and moreover the relevant pairs are disjoint. Hence, independently, each of the pairs $\{\widehat{c}_{j_\ell(n)-1}, \widehat{c}_{j_\ell(n) + 1} \}$ are independent, with distribution of the conductance $c(\{0,1\})$ conditioned on being no greater than $d_{n,\infty}^{1-\hat{\delta}}$. Since the event in the latter conditioning has probability converging to one, the result readily follows.
\end{proof}

\begin{proof}[Proof of Proposition~\ref{PropositionLaplaceTransform}]
Let us recall the definition of the following set, which is measurable with respect to $\omega$,
\begin{equation*}
	A_n^{v}(\delta) = \left\{j \in T_n^{\alpha_\infty}:\:\exists \,\, |i - j| = 1 \text{ such that } \left(v_i^{(n)} + v_j^{(n)}\right) > \delta \right\},
\end{equation*}
and define analogously $A_n^{\widehat{v}}(\delta)$ in the environment $\widehat{\omega}$. Using \eqref{eqn:BothEnvirnomentsConverge}, we have that there exists a $n_0$ such that for all $n \ge n_0$ the indices appearing in $A_n^{v}(\delta)$ and $A_n^{\widehat{v}}(\delta)$ are the same for all $n \ge n_0$, and in particular, they are the $j_{\ell}(n), j_{\ell}(n) + 1$ corresponding to the atoms of $A^{v}(\delta) = \{\ell: v_\ell > \delta\}$. By the Markov property, we can write
 \begin{align*}
	\widetilde{\mathbb{E}}^{\lambda/n}\left[ e^{- \xi T^{(n)}} \right]  = \widehat{\mathbf{E}}\left[ \sum_{i/n \in [-K, K]} \widehat{v}^{(n)}_{i} \widetilde{E}^{\widehat{\omega}, \lambda/n}_i \left[ e^{-\xi \frac{T}{d_{n, \infty}}} \right] \right].
\end{align*}
Let us split the sum as follows:
\begin{equation}\label{eqn:TwoPiecesSubAging}
	 \widehat{\mathbf{E}}\left[ \sum_{i/n \in A_n^{\widehat{v}}(\delta)} \widehat{v}^{(n)}_{i} \widetilde{E}^{\widehat{\omega}, \lambda/n}_i \left[ e^{-\xi \frac{T}{d_{n, \infty}}} \right] \right] + \widehat{\mathbf{E}}\left[ \sum_{i/n \not\in A_n^{\widehat{v}}(\delta)} \widehat{v}^{(n)}_{i} \widetilde{E}^{\widehat{\omega}, \lambda/n}_i \left[ e^{-\xi \frac{T}{d_{n, \infty}}} \right] \right].
\end{equation}
The contribution of the second term can be estimated as
\begin{equation*}
	\limsup_{n \to \infty} \widehat{\mathbf{E}}\left[ \sum_{i/n \not\in A_n^{\widehat{v}}(\delta)} \widehat{v}^{(n)}_{i} \widetilde{E}^{\widehat{\omega}, \lambda/n}_i \left[ e^{-\xi \frac{T}{d_{n, \infty}}} \right] \right]  \le \limsup_{n \to \infty} \widehat{\mathbf{E}}\left[ \sum_{i/n \not\in A_n^{\widehat{v}}(\delta)} \widehat{v}^{(n)}_{i} \right] \le \left( 1 - \widetilde{\mathbf{E}} \left[ \sum_{\ell \in A^{v}(\delta)} v_\ell \right] \right).
\end{equation*}
Let us define
\begin{equation*}
R_1(\delta)=1 - \widetilde{\mathbf{E}} \left[ \sum_{\ell \in A^{v}(\delta)} v_\ell \right],
\end{equation*}
and note that $R_1(\delta)$ converges to $0$ as $\delta$ goes to $0$, using the  dominated convergence theorem and the fact that the marginals of $\widetilde{Z}^{\lambda}$ are almost-surely purely atomic. For the first part of \eqref{eqn:TwoPiecesSubAging}, we can rewrite it as
\begin{align*}
	\widehat{\mathbf{E}}\left[ \sum_{i/n \in A_n^{\widehat{v}}(\delta)} \widehat{v}^{(n)}_{i}\widetilde{E}^{\widehat{\omega}, \lambda/n}_i \left[ e^{-\xi \frac{T}{d_{n, \infty}}} \right]\right] &= \widehat{\mathbf{E}}\left[ \widehat{\mathbf{E}}\left[ \sum_{i/n \in A_n^{\widehat{v}}(\delta)} \widehat{v}^{(n)}_{i} \widetilde{E}^{\widehat{\omega}, \lambda/n}_i \left[ e^{-\xi \frac{T}{d_{n, \infty}}} \right] \Big| \omega \right]\right] \\
	& = \widetilde{\mathbf{E}}\left[ \widehat{\mathbf{E}}\left[ \sum_{i/n \in A_n^{\widehat{v}}(\delta)} \widehat{v}^{(n)}_{i} \widetilde{E}^{\widehat{\omega}, \lambda/n}_i \left[ e^{-\xi \frac{T}{d_{n, \infty}}} \right] \Big|\,\, \omega \right]\right].
\end{align*}
The last term, by the observation at the beginning of the proof, can be re-written as the expectation of the sum appearing in Lemma~\ref{LemmaQuenchedLaplaceTransform}. Then, applying said lemma, we can write it as
\begin{flalign*}
	\limsup_{n \to \infty}\widetilde{\mathbf{E}}&\left[ \widehat{\mathbf{E}}\left[ \sum_{i/n \in A_n^{\widehat{v}}(\delta)} \widehat{v}^{(n)}_{i} \widetilde{E}^{\widehat{\omega}, \lambda/n}_i \left[ e^{-\xi \frac{T}{d_{n, \infty}}} \right] \Big|\,\, \omega \right]\right] \\
	&= \limsup_{n \to \infty} \widetilde{\mathbf{E}}\left[ \sum_{\ell \in A^v(\delta)} \widehat{\mathbf{E}}\left[ \widehat{v}^{(n)}_{j_{\ell}(n)} \widetilde{E}^{\widehat{\omega}, \lambda/n}_{j_\ell(n)} \left[ e^{-\xi \frac{T}{d_{n, \infty}}} \right] + \widehat{v}^{(n)}_{j_{\ell}(n) + 1} \widetilde{E}^{\widehat{\omega}, \lambda/n}_{j_\ell(n) + 1} \left[ e^{-\xi \frac{T}{d_{n, \infty}}} \right] \Big| \omega \right] \right]  \\
	&\le \limsup_{n \to \infty} \widetilde{\mathbf{E}}\left[ \sum_{\ell \in A^v(\delta)} v_{\ell} \widehat{\mathbf{E}}\left[ \mathcal{L}_{\exp\left(\frac{\widehat{c}_{j_\ell(n) - 1} + \widehat{c}_{j_\ell(n) + 1}}{2\nu^{\alpha_\infty}(x_\ell)} \right)} \Big| \omega \right] \right].
\end{flalign*}
Using Lemma~\ref{LemmaWeakCVNeighbours} and the boundedness of the function inside the expectation, we obtain
\begin{equation*}
	 \limsup_{n \to \infty} \widetilde{\mathbf{E}}\left[ \sum_{\ell \in A^v(\delta)} v_{\ell} \widehat{\mathbf{E}}\left[ \mathcal{L}_{\exp\left(\frac{\widehat{c}_{j_\ell(n) - 1} + \widehat{c}_{j_\ell(n) + 1}}{2\nu^{\alpha_\infty}(x_\ell)} \right)}\Big| \omega \right] \right]=\widetilde{\mathbf{E}}\left[ \sum_{\ell \in A^v(\delta)} v_{\ell} \widehat{\mathbf{E}}\left[ \mathcal{L}_{\exp\left(\frac{A^0+A^1}{2\nu^{\alpha_\infty}(x_\ell)} \right)}\Big| \omega \right] \right].
\end{equation*}
Plugging back in the sum the terms $\ell \not \in A^{v}(\delta)$ we get, by the dominated convergence theorem,
\begin{align*}
	\limsup_{n \to \infty} \widetilde{\mathbb{E}}^{\lambda/n}\left[ e^{- \xi T^{(n)}} \right] \le \widetilde{\mathbb{E}}^{\lambda} \left[ \mathcal{L}_{\exp\left( \frac{A^0 + A^2}{2 A^1} \right)}\left(  \xi \right)\right] + 2R_1(\delta).
\end{align*}
By reasoning analogously one gets
\begin{align*}
	\liminf_{n \to \infty} \widetilde{\mathbb{E}}^{\lambda/n}\left[ e^{- \xi T^{(n)}} \right] \ge \widetilde{\mathbb{E}}^{\lambda} \left[ \mathcal{L}_{\exp\left( \frac{A^0 + A^2}{2 A^1} \right)}\left(  \xi \right)\right] - 2R_1(\delta).
\end{align*}
Which is enough to conclude the proof, since $\delta$ is arbitrary and $\lim_{\delta \to 0} R_1(\delta) = 0$.
\end{proof}

\begin{proof}[Proof of Proposition~\ref{PropositionInsideKSubaging}]
	Note that we can write
	\begin{align*}
		\widetilde{\mathbb{P}}^{\lambda/n, K}\left( \big| \widetilde{X}_{b_n + s_1 d_{n, \infty}} - \widetilde{X}_{b_n +s_2 d_{n, \infty}} \big| \le 1, \,\, \forall s_1, s_2 \in [0, h] \right) = \widetilde{\mathbb{P}}^{\lambda/n, K} \left( T^{(n)} \ge h \right).
	\end{align*}
Using Proposition~\ref{PropositionLaplaceTransform} and L\'{e}vy's continuity Theorem (see \cite[Theorem~5.3]{Kallenberg}), we get that
\begin{equation*}
	\lim_{n \to \infty} \widetilde{\mathbb{P}}^{\lambda/n, K} \left( T^{(n)} \ge h \right) = \widetilde{\EVal}^K\left[ e^{- h\frac{A^0 + A^2}{2A^1}}\right],
\end{equation*}
also due to the shape of the upper tail of the distribution of an exponential random variables.
\end{proof}

\begin{proof}[Proof of Theorem~\ref{theoremSubAging}, Part II]
	The proof is very similar to the one of Part I above. We can once again notice that
	\begin{equation*}
		\begin{split}
			\widetilde{\mathbb{P}}^{\lambda/n}&\left( \big| \widetilde{X}_{b_n + s_1 d_{n, \infty}} - \widetilde{X}_{b_n +s_2 d_{n, \infty}} \big| \le 1, \,\, \forall s_1, s_2 \in [0, h], \,\, \tau_{Kn}^{X} \wedge \tau_{-Kn}^{X} > 2b_n \right) \\
			& = \widetilde{\mathbb{P}}^{\lambda/n, K}\left( \big| \widetilde{X}_{b_n + s_1 d_{n, \infty}} - \widetilde{X}_{b_n +s_2 d_{n, \infty}} \big| \le 1, \,\, \forall s_1, s_2 \in [0, h], \,\, \tau_{Kn}^{X} \wedge \tau_{-Kn}^{X} > 2b_n \right),
		\end{split}
	\end{equation*}
and analogously
\begin{equation*}
	\widetilde{\EVal}^{\lambda}\left[ e^{- h\frac{A^0 + A^2}{2A^1}}\mathds{1}_{\left\{\tau_{K}^{Z} \wedge \tau_{-K}^{Z} > 2\right\}}\right] = \widetilde{\EVal}^{\lambda, K}\left[ e^{- h\frac{A^0 + A^2}{2A^1}} \mathds{1}_{\left\{\tau_{K}^{Z} \wedge \tau_{-K}^{Z} > 2 \right\} }\right].
\end{equation*}
So, by Proposition~\ref{PropositionInsideKSubaging}, we have that, for all $K \in \mathbb{N}$ and $h>0$,
\begin{align}
	\limsup_{n \to \infty}\,\, &\widetilde{\mathbb{P}}^{\lambda/n}\left(\big| \widetilde{X}_{b_n + s_1 d_{n, \infty}} - \widetilde{X}_{b_n +s_2 d_{n, \infty}} \big| \le 1, \,\, \forall s_1, s_2 \in [0, h] \right) \nonumber \\
	& \le \widetilde{\EVal}^{\lambda}\left[ e^{- h\frac{A^0 + A^2}{2A^1}}\right] + \widetilde{\mathbb{P}}^\lambda \left( \tau_{K}^{Z} \wedge \tau_{-K}^{Z} \le 2 \right) + \limsup_{n} \,\,\widetilde{\mathbb{P}}^{\lambda/n}\left( \tau_{Kn}^{X} \wedge \tau_{-Kn}^{X} \le 2b_n \right).\nonumber
\end{align}
By applying \eqref{eqn:2HittingTimesComparedRWT}, we get that
\begin{align}
	\limsup_{n \to \infty}\,\, &\widetilde{\mathbb{P}}^{\lambda/n}\left( \big| \widetilde{X}_{b_n + s_1 d_{n, \infty}} - \widetilde{X}_{b_n +s_2 d_{n, \infty}} \big| \le 1, \,\, \forall s_1, s_2 \in [0, h] \right) \nonumber\\
	& \le \widetilde{\EVal}^{\lambda}\left[ e^{- h\frac{A^0 + A^2}{2A^1}}\right] + 2\widetilde{\mathbb{P}}^{\lambda}\left( \tau_{K - 1}^{Z} \wedge \tau_{-K + 1}^{Z} \le 2 \right). \nonumber
\end{align}
Reasoning in the same way, one can also get
\begin{align}
	\liminf_{n \to \infty}\,\, &\widetilde{\mathbb{P}}^{\lambda/n}\left( \big| \widetilde{X}_{b_n + s_1 d_{n, \infty}} - \widetilde{X}_{b_n +s_2 d_{n, \infty}} \big| \le 1, \,\, \forall s_1, s_2 \in [0, h] \right) \nonumber\\
	& \ge \widetilde{\EVal}^\lambda\left[ e^{- h\frac{A^0 + A^2}{2A^1}}\right] - 2\widetilde{\mathbb{P}}^\lambda\left( \tau_{K - 1}^{Z} \wedge \tau_{-K + 1}^{Z} \le 2 \right). \nonumber
\end{align}
Applying Lemma~\ref{LemmaExitSlowlyKBox} finishes the proof.
\end{proof}

\section{Some results on the limiting processes}\label{sec:reslimitprocess}

In this section, we collect together a number of useful results for the limiting processes $Z^\lambda$ and $\widetilde{Z}^{\lambda}$. We start by showing that, at a fixed time, $Z^\lambda$ is $\mathbb{P}^{\lambda}$-a.s.\ not located at its maximum.

\begin{lemma}\label{lemmaZNotAtMaximum}
Recall that $Z^{\lambda}_t$ was defined in \eqref{eqn:eqnDefinitionZ}. For this process, it holds that
\[\Prob^{\lambda}\left( Z_t<\widebar{Z}_t\right) = 1,\qquad \forall t>0.\]
\end{lemma}
\begin{proof} In the proof, we drop the $\lambda$ superscript of $\Prob^{\lambda}$, $Z^{\lambda}$ (and other objects) for simplicity. Let $B$ and $S \coloneqq S^{\alpha_0}$ be as in Section \ref{sec:limits}, and introduce the $\sigma$-finite measure $\nu(dx)=e^{2\lambda x}dx$ on $\mathbb{R}$. Set
\[\mathbb{Q}\coloneqq\mathbf{P}_B\times \mathbf{P}_S\times \nu,\]
where we write $\mathbf{P}_B$ for the law of $B$ and $\mathbf{P}_S$ for the law of $S$. Then, if $(B,S,X)$ is ``chosen'' according to $\mathbb{Q}$, we write $B^X=B+S(X)$, i.e.\ conditional on $(S,X)$, $B^X$ is the standard Brownian motion started from $S(X)$. We also define $H^X$ from $B^X$ analogously to the definition of $H^{\lambda}$ at \eqref{eqn:eqnDefinitionH}, and set
\[Z^X \coloneqq S^{-1}\left(B^X_{H^X_\cdot}\right).\]
Similarly to the proof of \cite[Lemma 5.4]{MottModel}, we observe that, for $\mathbf{P}_S$-a.e.\ realisation of $S$, under $\mathbb{Q}(\cdot\: |\: S)$, the process $Y^X \coloneqq B^X_{H^X_\cdot}$ is the Markov process naturally associated with the resistance metric measure space
\[\left( \overline{S(\mathbb{R})}\backslash\{S_{\infty}\},d,\mu\right),\]
where $S_{\infty} \coloneqq \lim_{t\rightarrow \infty}S(t)$, $d$ is the Euclidean metric and $\mu \coloneqq \mu^\lambda$, as defined at \eqref{SmoothSpeedMeasure}, started from $S(X)$. In particular, we highlight that, by \cite[Theorem 10.4]{KigQ}, $Y^X$ admits a jointly continuous, symmetric transition density $(p^Y(x,y))_{x,y\in \overline{S(\mathbb{R})}\backslash\{S_{\infty}\},\: t>0}$ with respect to $\mu$ (that does not depend on $X$). Again following the proof of \cite[Lemma 5.4]{MottModel}, one can use these facts and that
\begin{align*}
S:\mathbb{R}\backslash D\rightarrow \overline{S(\mathbb{R})}\backslash \left(\bigcup_{s\in D} \{S(s^{-}),S(s) \}\cup \left\{S_{\infty}\right\}\right),
\end{align*}
where $D$ is the set of discontinuities of $S$, is a homeomorphism to check that: for $\mathbf{P}_S$-a.e.\ realisation of $S$ and $X\not\in D$, it holds that $Z^X$ is $\mathbb{Q}(\cdot\: | \:S,X)$-a.s.\ continuous and, under $\mathbb{Q}(\cdot\:| \:S,X)$, is Markov with symmetric transition density
\[\left(p^Z(u,v)\right)_{u,v\in\mathbb{R},\:t>0}\coloneqq \left(p^Y(S(u),S(v))\mathds{1}_{u,v\not\in D}\right)_{u,v\in\mathbb{R},\:t>0}\]
with respect to $\nu$. As an easy consequence of the latter property, we note it $\mathbf{P}_S$-a.s.\ holds that, for any continuous, bounded functions of compact support $f_0,f_1,\dots,f_k$ and times $0=t_0<t_1<\dots t_k=T\in (0,\infty)$,
\begin{eqnarray*}
\mathbb{Q}\left(\prod_{i=0}^k f_i\left(Z^X_{t_i}\right)\:\vline\: S\right)   &=& \int\nu^{\otimes(k+1)}(dx_0dx_1\dots dx_k)f_0(x_0)\prod_{i=1}^kp_{t_i-t_{i-1}}^Z(x_{i-1},x_i) f_i(x_i) \\
   &=& \int\nu^{\otimes(k+1)}(dx_0dx_1\dots dx_k)f_k(x_k)\prod_{i=1}^kp_{(T-t_{i-1})-(T-t_{i})}^Z(x_{i},x_{i-1}) f_{i-1}(x_{i-1})\\
   &=&\mathbb{Q}\left(\prod_{i=0}^k f_i\left(Z^X_{T-t_i}\right)\:\vline\: S\right).
\end{eqnarray*}
Together with the continuity of $Z^X$, this implies that, $\mathbf{P}_S$-a.s., under $\mathbb{Q}(\cdot\:\vline\: S)$,
\begin{equation}\label{zrev}
\left(Z^X_t\right)_{t=0}^T\buildrel{d}\over{=}\left(Z^X_{T-t}\right)_{t=0}^T.
\end{equation}

Next, we claim that
\begin{equation}\label{zsup}
\mathbb{Q}\left(Z^X_0=\overline{Z^X}_T\right)=0,\qquad \forall T>0.
\end{equation}
To prove this, we start by fixing a typical realisation of $S$ and $X$, and a sequence $\delta_i\downarrow X$ in $D$. Let us denote, to ease the notation, $S_{\delta_{i}} \coloneqq S(\delta_i)$. For the elements in the sequence $\delta_i$, we will check that $\mathbb{Q}(\cdot\:\vline\:S,X)$-a.s., the hitting time
\[\tau^{B^X}(S_{\delta_i})\coloneqq\inf\left\{t>0:\:B^X_t=S(\delta_i)\right\}\]
is equal to $H^X_{t_i}$ for some (random) time $t_i>0$. Indeed, since the local time of Brownian motion is strictly positive on an open interval of its starting point for any positive time and also $\mu([S_{\delta_i},S_{\delta_i}+\varepsilon])>0$ for any $\varepsilon>0$, it $\mathbb{Q}(\cdot\:\vline\:S,X)$-a.s.\ holds that
\[\int L^X_{\tau^{B^X}(S_{\delta_i})+t}(x)\mu(dx)-\int L^X_{\tau^{B^X}(S_{\delta_i})}(x)\mu(dx)>0,\qquad \forall t>0,\]
where $L^X$ is the local time of $B^X$. In particular $\tau^{B^X}(S_{\delta_i})$ is a point of strict increase for the continuous additive functional $t\mapsto \int L^X_{t}(x)\mu(dx)$, and it must therefore fall into the image of $H^X$. Applying the same argument at $t=0$, one can verify that $t=0$ also falls into the image of $H^X$. Hence we find that $H^X_0=0$ and $\tau^{B^X}(S_{\delta_i})=H^X_{t_i}$ for some $t_i>0$, as required. Consequently, it $\mathbb{Q}(\cdot\:\vline\:S,X)$-a.s.\ holds that
\[Z^X_{t_i}=S^{-1}\left(B^X_{H^X_{t_i}}\right)=S^{-1}\left(B^X_{\tau^{B^X}(S_{\delta_i})}\right)=S^{-1}\left(S_{\delta_i}\right)=\delta_i,\]
which in turn implies that
\[\overline{Z^X}_{t_i}>Z^X_0.\]
Furthermore, note that $H_{t_i}=\tau^{B^X}(S_{\delta_i})\downarrow0$. Since $H^X$ is strictly increasing (by the continuity of $t\mapsto \int L^X_{t}(x)\mu(dx)$), it follows that $t_i\downarrow 0$. In conjunction with the previous displayed equation, this leads to the conclusion that, $\mathbb{Q}(\cdot\:\vline\:S,X)$-a.s.,
\[\overline{Z^X}_{t}>Z^X_0,\qquad \forall t>0,\]
and the result at \eqref{zsup} follows.

Combining \eqref{zrev} and \eqref{zsup}, we conclude that $\mathbb{Q}(Z^X_t=\overline{Z^X}_t)=0$ for any $t>0$. Thus, denoting by $\mathbb{P}$ the measure $\mathbf{P}_B\times \mathbf{P}_S$, defining $Z$ as in \eqref{eqn:eqnDefinitionZ}, and writing $\theta_x$ for the usual shift on $\mathbb{R}$,
\begin{eqnarray*}
0&=&\int \nu(dx)\int \mathbf{P}_S(ds)\mathbb{Q}\left(Z^X_t=\overline{Z^X}_t\:\vline\:S=s,\:X=x\right)\\
&=&\int \nu(dx)\int \mathbf{P}_S(ds)\mathbb{P}\left(Z_t=\overline{Z}_t\:\vline\:S=s\circ \theta_x-s_x\right)\\
&=&\int \nu(dx) \mathbb{P}\left(Z_t=\overline{Z}_t\right),
\end{eqnarray*}
where the second inequality is a simple consequence of the construction of $Z^X$, and the third is a result of the stationarity of the distribution of the increments of $S$ under spatial shifts. Hence we obtain that $\mathbb{P}(Z_t=\overline{Z}_t)=0$, which completes the proof.
\end{proof}

We next show that the maximum of $Z^\lambda$ is located at a discontinuity of $S^{\alpha_0}$. Note we also drop lambda superscripts in the proof of the following result.

\begin{lemma}\label{lemmaMaximumDisc}
Recall that $Z^{\lambda}_t$ was defined in \eqref{eqn:eqnDefinitionZ} and let us define the set $D \coloneqq \{ v \in \R:\: S^{\alpha_0}(v) \ne S^{\alpha_0}(v^-) \}$. For each fixed $t>0$, it holds that
\begin{equation}\label{statement}
\Prob^\lambda\left( \widebar{Z}_t \in D \right) =\Prob^{\lambda}\left( \widebar{(B_{H^{\lambda}})}_t \in \bigcup_{v\in D}\{S^{\alpha_0,\lambda}(v^-)\} \right) =\Prob^\lambda \left( (\widebar{B})_{H^{\lambda}_t} \in \bigcup_{v \in D} \left\{ \left(S^{\alpha_0, \lambda}(v^-), S^{\alpha_0, \lambda}(v) \right) \right\}  \right) = 1.
\end{equation}
\end{lemma}
\begin{proof} We drop the $\lambda$ superscript of $\mathbb{P}^{\lambda}$, $Z^{\lambda}$, $H^{\lambda}$ $S^{\alpha_0, \lambda}$ for simplicity. Let $B$ and $S \coloneqq S^{\alpha_0}$ be as in Section \ref{sec:limits}. We first claim that if $Y \coloneqq B_{H_\cdot}$, then it $\mathbb{P}$-a.s.\ holds that
\begin{equation}\label{ybar}
\overline{Y}_t = \sup\left\{B_s:\:s\leq H_t\right\}\cap\overline{S(\mathbb{R})},\qquad \forall t>0.
\end{equation}
Since $\mu$ has support $\overline{S(\mathbb{R})}$, it follows from its construction that the process $Y$ takes values in $\overline{S(\mathbb{R})}$, $\mathbb{P}$-a.s. Thus it is straightforward to check that $\overline{Y}_t$ is bounded above by the right-hand side of \eqref{ybar}. Conversely, for every $u\in D$, by applying the same argument as used in the proof of Lemma \ref{lemmaZNotAtMaximum}, it can be checked that, $\mathbb{P}$-a.s., if $\tau^B(S(u))\leq H_t$ (where $\tau^B(x)$ is the hitting time of $x$ by $B$), then $\tau^B(S(u))=H_s$ for some $s\leq t$. In particular, this implies that $S(u)\leq \overline{Y}_t$. Hence, $\mathbb{P}$-a.s.\ we have that
\[\sup\left\{B_s:\:s\leq H_t\right\}\cap\overline{S(D)}\leq \overline{Y}_t,\qquad \forall t>0.\]
Since one has that $\overline{S(D)}=\overline{S(\mathbb{R})}$ and taking closure does not affect the supremum, we thus obtain that $\overline{Y}_t$ is also bounded below by the right-hand side of \eqref{ybar}, and thus the proof of the claim is complete.

Next, fix $t>0$, and recall from Lemma \ref{lemmaZNotAtMaximum} that $Z_t<\overline{Z}_t$, $\mathbb{P}$-a.s. Since $Z$ is a continuous process, it follows that, $\mathbb{P}$-a.s., there exists a $\delta>0$ such that $Z_s<\overline{Z}_t$ for all $s\in [t-\delta,t]$. Using the monotonicity of $S^{-1}$, one can also check the analogous claim for $Y$. In particular, we $\mathbb{P}$-a.s.\ have that there exists a $\delta>0$ such that
\[\bar{Y}_t=\bar{Y}_s=\sup\left\{B_r:\:r\leq H_s\right\}\cap\overline{S(\mathbb{R})},\qquad \forall s\in[t-\delta,t].\]
Now, the function $s\mapsto H_s$ is strictly increasing (by the continuity of $r\mapsto \int L^X_{r}(x)\mu(dx)$), and so one may further deduce that $[H_{t-\delta},H_t]$ contains a rational number, $q$ say. Combining these observations, we obtain that $\mathbb{P}$-a.s.\ there exists a rational number $q\in [0, H_t]$ such that
\begin{equation}\label{yexp}
\overline{Y}_t=\sup[0,\overline{B}_q]\cap\overline{S(\mathbb{R})}.
\end{equation}
Since the set $\overline{S(\mathbb{R})}$ has zero Lebesgue measure (see \cite[Corollary II.20]{BertoinLevy}) and the supremum of a Brownian motion at a fixed time has a distribution that is absolutely continuous with respect to Lebesgue measure (see \cite[pg.~102]{KaratzasShreve}), it $\mathbb{P}$-a.s.\ holds that
\[\overline{B}_q\in (0,\infty)\backslash \overline{S(\mathbb{R})}=\left(S_\infty,\infty\right) \cup \bigcup_{u \in {D}:\:u>0} \left(S(u^-), S(u) \right) ,\]
where we again write $S_{\infty} \coloneqq \lim_{t\rightarrow \infty}S(t)$. If $\overline{B}_q\in (S_\infty,\infty)$, then \eqref{yexp} implies that $\overline{Y}_t=S_\infty$, and from this it follows that $\overline{Z}_t=\infty$. However, this cannot happen, since $Z$ is a conservative process (as established in \cite[Lemma 5.3]{MottModel}). Hence $\overline{B}_q\in (S(u^-), S(u))$ for some $u\in D$, and together with \eqref{yexp} this yields that $\overline{Y}_t=S(u^-)$. Consequently, $\overline{Z}_t=u\in D$, which verifies that the left-hand probability at \eqref{statement} is equal to 1. Moreover, from the same argument, we see that $(\overline{B})_{H_t}=\overline{B}_q\in (S(u^-), S(u))$ for some $u\in D$, as required to complete the proof.
\end{proof}

We can also show that the maximum of $Z^\lambda_{t}$ has positive probability of being located at any discontinuity point of $S^{\alpha_0}$ on the right of its starting point.

\begin{lemma}\label{LemmaMaximumOnEveryWall}
	Recall that $Z^{\lambda}_t$ was defined in \eqref{eqn:eqnDefinitionZ} and the set $D \coloneqq \{ v \in \R:\: S^{\alpha_0}(v) \ne S^{\alpha_0}(v^-) \}$. For each fixed $t>0$, it holds that $\mathbb{P}_{0}$-a.s.\ the locations of the atoms of the distribution of $\widebar{Z}^{\lambda}_t$ contain the locations of the atoms of $D \cap (0, \infty)$.
\end{lemma}

\begin{proof}
	We drop the $\lambda$ superscript of $\mathbb{P}^{\lambda}$, $Z^{\lambda}$, $H^{\lambda}$ $S^{\alpha_0, \lambda}$, $\mu^\lambda$ for simplicity. Let $B$ and $S \coloneqq S^{\alpha_0}$ be as in Section \ref{sec:limits} and set $Y = B_{H_\cdot}$. Fix any $u \in D \cap (0, \infty)$, we know that $\left(S(u^-), S(u) \right)$ has positive Lebesgue measure and that $S(u^-) > 0$. Firstly, by basic properties of Brownian motion, we have for any $q>0$ that
	\begin{equation*}
		\mathbb{P}\left( \widebar{B}_{q} \in \left( S(u^{-}), S(u) \right), B_{q} \in \left( 0, S(u^{-}),\:q \in [0, H_{t}] \right) \,\, | \,\, S \right) > 0.
	\end{equation*}
(In particular, in time $q$, the Brownian motion can hit the interval $( S(u^{-}), S(u))$ and return to the left of $S(u^-)$ without hitting $S_u$, whilst placing arbitrarily small local time on $\overline{S(\mathbb{R})}$.) Hence, we are left to prove that with positive probability the maximum of the process does not exceed $S(u)$ in the remaining time. A Brownian motion $B$ started from $S(u^{-})$ can accumulate, with positive probability, arbitrary large local time in the set $\left[ 0, S(u^{-}) \right]$ before leaving the interval $(-1, r)$, where $r \coloneqq (S(u) + S(u^{-}))/2$. More precisely, let $\sigma \coloneqq \inf\{t> 0: B_{t} \not \in (-1, r)\}$ and let $L^{B}_{t}[x, y] \coloneqq \inf_{z \in [x, y]} L_{t}^{B}(z)$, then, for all $s>0$, by the Ray-Knight theorem,
	\begin{equation*}
		\mathbb{P}_{S(u^{-})}\left( L_{\sigma}^{B}[0, S(u^{-})] > s \,\, | \,\, S \right) > 0.
	\end{equation*}
	Furthermore, by the fact that subordinators are almost surely increasing, it must be the case that $\mu([0, S(u^{-})]) > 0$. Hence, we deduce that
	\begin{equation*}
		\mathbb{P}_{S(u^{-})}\left( \sigma > H_{t}\,\,| \,\, S \right) > 0.
	\end{equation*}
	Then, by the Markov property,
	\begin{align*}
		\lefteqn{\mathbb{P}\left( \widebar{Y}_{t} = S(u^{-}) \,\,| \,\, S \right)}\\
& \ge \mathbb{P}\left( \widebar{B}_{q} \in \left( S(u^{-}), S(u) \right), B_{q} \in \left( 0, S(u^{-}) \right),\:q \in [0, H_{t}] \,\, | \,\, S \right) \mathbb{P}_{S(u^{-})}\left( \sigma > H_{t} \,\, | \,\, S \right)\\&>0,
	\end{align*}
and we conclude by taking expectations.
\end{proof}

Now, turning to the model with traps, we show that $\widetilde{Z}^\lambda$ is $\widetilde{\Prob}$-a.s.\ located in the discontinuities of $S^{\alpha_\infty}$.

\begin{lemma}\label{lem73}
Recall that $\widetilde{Z}^{\lambda}_t$ was defined at \eqref{eqn:eqnDefinitionZTilde} and let us define the set $D_{\infty} \coloneqq \{ v \in \R:\: S^{\alpha_\infty}(v) \ne S^{\alpha_\infty}(v^-) \}$. It holds that
\[\widetilde{\Prob}^{\lambda}\left(  \widetilde{Z}_t\in D_\infty \right) = 1,\qquad \forall t>0.\]
\end{lemma}
\begin{proof} Since $B_{\widetilde{H}^{\lambda}_\cdot}$ is the time change of Brownian motion by the measure $\widetilde{\mu}^{\lambda}$, which is a purely atomic measure, with atoms at $(S^{\alpha_0, \lambda}_u)_{u\in D_\infty}$, it must hold for each fixed $t>0$ that, $\widetilde{\Prob}$-a.s.,
\[B_{\widetilde{H}^{\lambda}_t}\in \left\{S^{\alpha_0, \lambda}(u):\:u\in D_\infty\right\}.\]
The result follows.
\end{proof}

From the previous result, we can check that $\widetilde{Z}^\lambda$ is likely to be found at exactly the same location at two nearby times.

\begin{lemma}\label{lemmaleftcont}
Recall that $\widetilde{Z}^{\lambda}_t$ was defined at \eqref{eqn:eqnDefinitionZTilde}. It holds that
\[\lim_{\varepsilon \downarrow 0}\widetilde{\Prob}^{\lambda}\left( \widetilde{Z}_{1-\varepsilon}  = \widetilde{Z}_{1}  \right) = 1.\]
\end{lemma}
\begin{proof} We drop the $\lambda$ superscript of $\widetilde{\Prob}^{\lambda}$, $\widetilde{Z}^{\lambda}$, $\widetilde{H}^{\lambda}$, $S^{\alpha_0, \lambda}$, $S^{\alpha_\infty, \lambda}$, $\widetilde{\mu}^{\lambda}$ for simplicity, set also $S_{u}^{\alpha_0} \coloneqq S^{\alpha_0}(u)$. By the scaling properties of the subordinators $S^{\alpha_0}$, $S^{\alpha_\infty}$ and the Brownian motion $B$, it is possible to check that, for any constant $c>0$,
\[\left(\widetilde{Z}_{ct}\right)_{t\geq0}\buildrel{(\mathrm{d})}\over{=}\left(c^{\alpha_0\alpha_\infty/(\alpha_0+\alpha_\infty)}\widetilde{Z}_{t}\right)_{t\geq0}.\]
Hence the statement of the lemma is equivalent to the following limit:
\[\lim_{\varepsilon \downarrow 0}\widetilde{\Prob}\left( \widetilde{Z}_{1}  = \widetilde{Z}_{1+\varepsilon} \right) = 1.\]

Now, from (the proof of) Lemma \ref{lem73}, we have that $B_{\widetilde{H}_1}=S_{u}^{\alpha_0}$ for some $u\in D_\infty$. Moreover, since a Brownian motion accumulates a mean $\eta/2$ exponential amount of local time at its starting point before exiting an ball of radius $\eta$ around it (cf.\ \cite[(3.189)]{MR}), it holds that the time-changed process $B_{\widetilde{H}_\cdot}$ takes at least a mean $\eta/2\tilde{\mu}(\{S_{u}^{\alpha_0}\})$ exponential time after time $1$ to escape the interval $[S_{u}^{\alpha_0}-\eta,S_{u}^{\alpha_0}+\eta]$. Thus
\[\widetilde{\mathbb{P}}\left(B_{\widetilde{H}_{1+\varepsilon}}\neq S_{u}^{\alpha_0}\:\vline\:S^{\alpha_0},\:S^{\alpha_\infty},\:B_{\widetilde{H}_1}=S_{u}^{\alpha_0}\right)\leq 1-e^{-2\varepsilon\widetilde{\mu}(\{S_{u}^{\alpha_0}\})/\eta}+\widetilde{\mathbb{P}}\left(B_{\widetilde{H}^\eta_{1+\varepsilon}}\neq S_{u}^{\alpha_0}\:\vline\:S,\:B_{\widetilde{H}^\eta_1}=S_{u}^{\alpha_0}\right),\]
where $B_{\widetilde{H}^\eta_\cdot}$ corresponds to the Brownian motion time-changed by the restricted measure $\widetilde{\mu}^\eta\coloneqq \widetilde{\mu}(\cdot\cap[S^{\alpha_0}_u-\eta,S^{\alpha_0}_u+\eta])$. For the probability on the right-hand side above, we have the following estimate from \cite[Lemma 2.5]{FIN}:
\[\widetilde{\mathbb{P}}\left(B_{\widetilde{H}^\eta_{1+\varepsilon}}\neq S_u^{\alpha_0}\:\vline\:S^{\alpha_0},\:S^{\alpha_\infty},\:B_{\widetilde{H}^\eta_1}=S_u^{\alpha_0}\right)\leq \frac{\widetilde{\mu}\left([S^{\alpha_0}_u-\eta,S^{\alpha_0}_u+\eta]\backslash\{S^{\alpha_0}_u\}\right)}{\widetilde{\mu}\left([S^{\alpha_0}_u-\eta,S^{\alpha_0}_u+\eta]\right)}.\]
Hence we conclude that
\begin{eqnarray*}
\lefteqn{\widetilde{\mathbb{P}}\left(B_{\widetilde{H}_{1+\varepsilon}}\neq B_{\widetilde{H}_1}\right)}\\
&= &\widetilde{\mathbb{E}}\left[\sum_{u\in D_\infty}\widetilde{\mathbb{P}}\left(B_{\widetilde{H}_{1+\varepsilon}}\neq S_u^{\alpha_0}\:\vline\:S^{\alpha_0},\:S^{\alpha_\infty},\:B_{\widetilde{H}_1}=S_u^{\alpha_0}\right)\widetilde{\mathbb{P}}\left(B_{\widetilde{H}_1}=S_u^{\alpha_0}\:\vline\: S^{\alpha_0},\:S^{\alpha_\infty}\right)\right]\\
&\leq &\widetilde{\mathbb{E}}\left[\sum_{u\in D_\infty}\left(1-e^{-2\varepsilon\widetilde{\mu}(\{S^{\alpha_0}_u\})/\eta}+\frac{\widetilde{\mu}\left([S^{\alpha_0}_u-\eta,S^{\alpha_0}_u+\eta]\backslash \{S^{\alpha_0}_u\}\right)}{\widetilde{\mu}\left([S^{\alpha_0}_u-\eta,S^{\alpha_0}_u+\eta]\right)}\right)\widetilde{\mathbb{P}}\left(B_{\widetilde{H}_1}=S_u^{\alpha_0}\:\vline\: S^{\alpha_0},\:S^{\alpha_\infty}\right)\right].
\end{eqnarray*}
Taking limits as $\varepsilon\rightarrow 0$, this implies
\[\limsup_{\varepsilon\rightarrow0}\widetilde{\mathbb{P}}\left(B_{\widetilde{H}_{1+\varepsilon}}\neq B_{\widetilde{H}_1}\right)\leq \widetilde{\mathbb{E}}\left[\sum_{u\in D_\infty}\frac{\widetilde{\mu}\left([S^{\alpha_0}_u-\eta,S^{\alpha_0}_u+\eta]\backslash \{S^{\alpha_0}_u\}\right)}{\widetilde{\mu}\left([S^{\alpha_0}_u-\eta,S^{\alpha_0}_u+\eta]\right)}\widetilde{\mathbb{P}}\left(B_{\widetilde{H}_1}=S_u^{\alpha_0}\:\vline\: S^{\alpha_0},\:S^{\alpha_\infty}\right)\right].\]
Moreover, since $\widetilde{\mathbb{P}}$-a.s., for any $u\in D_\infty$, as $\eta\rightarrow 0$, $\widetilde{\mu}([S^{\alpha_0}_u-\eta,S^{\alpha_0}_u+\eta]\backslash \{S^{\alpha_0}_u\})\rightarrow 0$ and $\widetilde{\mu}([S^{\alpha_0}_u-\eta,S^{\alpha_0}_u+\eta])\rightarrow \widetilde{\mu}(\{S^{\alpha_0}_u\})>0$, we have from the dominated convergence theorem that the right-hand side here converges to zero as $\eta\rightarrow 0$. Thus
\begin{equation}\label{eqn:ProbContinuity}
	\lim_{\varepsilon\rightarrow0}\widetilde{\mathbb{P}}\left(B_{\widetilde{H}_{1+\varepsilon}}= B_{\widetilde{H}_1}\right)=1,
\end{equation}
and the result follows.
\end{proof}

A weaker version of the previous result is the following, which is a simple consequence of the continuity of $\widetilde{Z}^\lambda$ (see, for example, \cite[Corollary 3.1]{Ogura}).

\begin{lemma}\label{lemmaExitIntervalDiffusion}
Recall that $\widetilde{Z}^{\lambda}_t$ was defined at \eqref{eqn:eqnDefinitionZTilde}. For all $\delta>0$, it holds that
	\[\lim_{\varepsilon \downarrow 0}\widetilde{\Prob}^{\lambda}\left( \sup_{t \in [1-\varepsilon, 1]} \left|\widetilde{Z}_{1} - \widetilde{Z}_{t}\right| \le \delta \right) = 1.\]
\end{lemma}

\begin{lemma}
	Recall that $\widetilde{Z}^{\lambda}_t$ was defined at \eqref{eqn:eqnDefinitionZTilde} and let us define the set $D_{\infty} \coloneqq \{ v \in \R:\: S^{\alpha_\infty}(v) \ne S^{\alpha_\infty}(v^-) \}$. For any starting point $x \in \mathbb{R}$ and any $t>0$ the locations of the atoms $D_{\infty}$ are contained in the locations of the atoms of the marginal $\widetilde{Z}^{\lambda}_t$ started from $x$.
\end{lemma}
\begin{proof}
We drop the $\lambda$ superscript of $\widetilde{\Prob}^{\lambda}$, $\widetilde{Z}^{\lambda}$, $\widetilde{H}^{\lambda}$, $S^{\alpha_0, \lambda}$, $S^{\alpha_\infty, \lambda}$, $\widetilde{\mu}^{\lambda}$ for simplicity. Note that the statement is equivalent to prove that the atoms of the marginal of $B_{\widetilde{H}_t}$, started from any $y \in \overline{S^{\alpha_0}(\mathbb{R})}$, contain the atoms of the speed measure $\widetilde{\mu}$.
	
By \eqref{eqn:ProbContinuity}, for all $z_{0}$ such that $\widetilde{\mu}(z_{0}) > 0$, we have that the probability
\begin{equation*}
	\widetilde{\mathbb{P}}\left( B_{\widetilde{H}_t} = z_{0} \:\vline\:S^{\alpha_0},\:S^{\alpha_\infty},\:B_{\widetilde{H}_0}=y\right),
\end{equation*}
is continuous in $t$. Furthermore, this continuity extends to all $z \in \mathbb{R}$ as Lemma~\ref{lem73} guarantees that the other points always have $0$ probability. From the definition at \eqref{eqn:HeavyTailedSpeedMeasure}, for all $t'>t>0$ we can write
\begin{align}
	\lefteqn{\int_{t}^{t'} \widetilde{\mathbb{P}}\left( B_{\widetilde{H}_s} = z_{0} \:\vline\:S^{\alpha_0},\:S^{\alpha_\infty},\:B_{\widetilde{H}_0}=y\right)ds}\nonumber\\
& = \widetilde{\mathbb{E}}\left[ L_{\widetilde{H}_{t'}}^{B} (z_{0} - y) - L_{\widetilde{H}_{t}}^{B} (z_{0} - y) \:\vline\:S^{\alpha_0},\:S^{\alpha_\infty},\:B_{\widetilde{H}_0}=y\right] \mu(z_0).\label{ProbCrossingPos}
\end{align}
We claim that the expectation on the right-hand side is strictly positive. Indeed, it is not hard to check that if $\widetilde{H}_t$ is strictly increasing in $t$ for all starting points $B_{\widetilde{H}_0}=y$. Hence, between times $\widetilde{H}_{t}$ and $\widetilde{H}_{t'}$ the Brownian motion has a positive probability to cross $z_{0} - y$ and hence will strictly increase its local time there so that the right-hand side of \eqref{ProbCrossingPos} is strictly positive.

Thus, for all $t>0$ there exists $0<s<t$ such that $\widetilde{\mathbb{P}}( B_{\widetilde{H}_s} = z_{0} \:\vline\:S^{\alpha_0},\:S^{\alpha_\infty},\:B_{\widetilde{H}_0}=y)$ is strictly positive. Hence, by continuity of these probabilities and the Markov property, we get that
\begin{equation*}
		 \widetilde{\mathbb{P}}\left( B_{\widetilde{H}_s} = z_{0} \:\vline\:S^{\alpha_0},\:S^{\alpha_\infty},\:B_{\widetilde{H}_0}=y\right) \widetilde{\mathbb{P}}\left( B_{\widetilde{H}_{t-s}} = z_{0} \:\vline\:S^{\alpha_0},\:S^{\alpha_\infty},\:B_{\widetilde{H}_0}=z_0\right) > 0,
\end{equation*}
which is a lower bound for $\widetilde{\mathbb{P}}( B_{\widetilde{H}_t} = z_{0} \:\vline\:S^{\alpha_0},\:S^{\alpha_\infty},\:B_{\widetilde{H}_0}=y)$. We conclude by taking expectations.
\end{proof}

Finally, we check the conservativeness of $\widetilde{Z}^{\lambda}$.

\begin{lemma}\label{LemmaExitSlowlyKBox}
	For any $u > 0$,
	\begin{equation*}
		\limsup_{K \to \infty}\,\, \widetilde{\Prob}^{\lambda}\left( \tau_{K}^{\widetilde{Z}} \wedge \tau_{-K}^{\widetilde{Z}} \le u \right) = 0.
	\end{equation*}
\end{lemma}
\begin{proof}
The proof is similar to the one of \cite[Lemma 5.3]{MottModel}. We aim to prove that for all $\varepsilon> 0$ we can find $K$ large enough such that
\begin{equation*}
	\widetilde{\Prob}^{\lambda}\left( \tau_{K}^{\widetilde{Z}} \wedge \tau_{-K}^{\widetilde{Z}} \le u \right) \le \varepsilon.
\end{equation*}
Let us start proving this for the symmetric process $Z = Z^{0}$. By symmetry, we can reduce the problem into proving that
\begin{equation*}
   	\widetilde{\Prob}\left( \tau^{\widetilde{Z}}_{K} \le u \right) \le \frac{\varepsilon}{2}.
\end{equation*}
Let us define
\begin{align*}
   	M_{K} &\coloneqq \sup_{t \in [1, K-1]} \left(S^{\alpha_\infty}(t) - S^{\alpha_\infty}(t^{-})\right), \\
   	T_{K} &\coloneqq \argmax_{t \in [1, K-1]} \left(S^{\alpha_\infty}(t) - S^{\alpha_\infty}(t^{-})\right).
\end{align*}
The number of jumps of $S^{\alpha_\infty}$ in the interval $[1,K-1]$ of size larger than $m$ has a Poisson distribution with mean $m^{-\alpha_\infty} (K - 2)$. Hence, for all $m$, we can find $K_0$ such that, for all $K \ge K_0$,
\[\widetilde{\mathbf{P}}\left(M_{K} < m\right)=e^{-m^{-\alpha_\infty}(K-2)} \le \frac{\varepsilon}{4}.\]
Moreover, using the independence of the two subordinators $S^{\alpha_0}$ and $S^{\alpha_\infty}$ (as well as the fact that the former is a strictly increasing function), we can find $\eta\in(0,1)$ (independent of $m$ and $K$) so that
\[\widetilde{\mathbf{P}}\left(\min\{S^{\alpha_0}(T_K+1)-S^{\alpha_0}(T_K),S^{\alpha_0}(T_K^-)-S^{\alpha_0}(T_K-1)\}< \eta\right)\leq \frac{\varepsilon}{4}.\]
Now, on the event that both $M_{K} \geq  m$ and $\min\{S^{\alpha_0}(T_K+1)-S^{\alpha_0}(T_K),S^{\alpha_0}(T_K^-)-S^{\alpha_0}(T_K-1)\}\geq \eta$, one has by arguing as in the proof of Lemma \ref{lemmaleftcont} that from the hitting time of $T_K$ by $\widetilde{Z}^\lambda$ to the time it exits a ball of radius 1 around this, at least an exponential, mean $\eta m/2$, amount of time must pass. In particular, we get that
\[\widetilde{\Prob}\left( \tau^{\widetilde{Z}}_{K} \le u \right)\leq \frac{\varepsilon}{2}+
\widetilde{\Prob}\left( \exp\left( \frac{2}{\eta m} \right) \le u \right)=  \frac{\varepsilon}{2}+1-e^{-2/\eta m}.\]
Thus, by choosing first $\eta$ (which we recall could be chosen independent of $m$ and $K$), and then $m$ and then $K$ suitably large, one can ensure this bound is smaller than $\varepsilon$, as desired.

For the version with vanishing bias $\lambda > 0$ we notice that, by standard properties of Brownian motion
\begin{equation*}
	\widetilde{\Prob}^{\lambda}\left( \tau_{K}^{\widetilde{Z}} \wedge \tau_{-K}^{\widetilde{Z}} = \tau_{-K}^{\widetilde{Z}} \big| S^{\alpha_0, \lambda} \right) = \frac{\left| S^{\alpha_0, \lambda}(- K)^{-1} \right|}{\left| S^{\alpha_0, \lambda}(- K)^{-1} \right| + \left| S^{\alpha_0, \lambda}(K)^{-1} \right|}.
\end{equation*}
Thanks to the fact that $\left(S^{\alpha_0, \lambda}(t)\right)_{t \in \mathbb{R}}$ is bounded for $t \to \infty$ and unbounded for $t \to - \infty$ one can choose $K$ large enough such that
\begin{equation*}
	\widetilde{\Prob}^{\lambda}\left( \tau_{K}^{\widetilde{Z}} \wedge \tau_{-K}^{\widetilde{Z}} = \tau_{-K}^{\widetilde{Z}} \right) \le \frac{\varepsilon}{2}.
\end{equation*}
Then re-running the proof as in the symmetric case we get that, for every $u,\varepsilon>0$,
\begin{equation*}
	\widetilde{\Prob}^{\lambda}\left( \tau_{K}^{\widetilde{Z}} \wedge \tau_{-K}^{\widetilde{Z}} \le u \right) \le \varepsilon
\end{equation*}
for large $K$. (Note that, the one adaptation required is that one should consider the event $\min\{S^{\alpha_0,\lambda}(T_K+1)-S^{\alpha_0,\lambda}(T_K),S^{\alpha_0,\lambda}(T_K^-)-S^{\alpha_0,\lambda}(T_K-1)\}\geq \eta e^{-2\lambda T_K}$, which still has a probability that is independent of $K$.) This is enough to conclude the proof.
\end{proof}

\appendix

\section{Note on $\mathrm{\mathbf{J_1}}$ convergence} \label{sec:AppendixJ_1}

In the proofs of Propositions~\ref{PropositionRWEnvConvergence} and \ref{PropositionRWTEnvConvergence}, we claimed that the convergence of $\widebar{\nu}^{\alpha_0, (n)}$ towards $\widebar{\nu}^{\alpha_0}$ in the sense of Proposition~\ref{PropositionPoissonPoint} guarantees that, under the same coupling, almost-surely, recalling \eqref{EquationEffctiveResistance},
\[
S^{\alpha_0, \lambda/n, (n)}(t) \stackrel{J_1}{\to} S^{\alpha_0, \lambda}(t),
\]
and the same holds for the $\alpha_\infty$ process. Let us now show that the first statement of Lemma~\ref{LemmaCorollaryEnvironment} holds. In particular, we will show that, almost surely,
\[
S^{\alpha_0, \lambda/n, (n)}(t) \stackrel{J_1}{\to} q^{-1/\alpha_0}S^{\alpha_0, \lambda/q}(qt),
\]
where $q \coloneqq (1 - p) = \widetilde{\mathbf{P}}(r(\{0, 1\}) > 1)$. This is enough, as a change of variable in \eqref{eqn:eqnTwoSidedLevy} and the self-similarity of subordinators guarantee that
\[
\left(S^{\alpha_0, \lambda}(t)\right)_{t \ge 0} \stackrel{(\mathrm{d})}{=} \left(q^{-1/\alpha_0}S^{\alpha_0, \lambda/q}(qt)\right)_{t \ge 0}.
\]
Note that, for our purpose, it is enough to prove almost sure convergence to a version of the process. Hence, in Lemma~\ref{LemmaCorollaryEnvironment} and in what follows, for simplicity, we denote the rescaled process $S^{\alpha_0, \lambda}(t)$.

\begin{proof}[Proof of Lemma~\ref{LemmaCorollaryEnvironment}] We use the notation defined in Section~\ref{SectionCopuledSpaces}. Note that $d^*_{n, 0}/d_{n, 0} \to q^{-1/\alpha_0}$, so let us discard that part. Concerning notation, let us set, for $x \in [0, 1]$, $B^{(n)}(x) \coloneqq \floor{nx} - N^{(n)}(x)$ and use the shorthand $B^{(n)} = B^{(n)}(1)$. Furthermore let us introduce the function $\widetilde{h}_n(x) \coloneqq B^{(n)}(x)/n$ and its right continuous inverse $\widetilde{h}^{-1}_n$. One can check that, for $i$ such that $b_i=0$, $\widetilde{h}_n(i/n) = (i -i^*)/n$ and $\widetilde{h}^{-1}_n((i -i^*)/n)-\tfrac{1}{n} = i/n$, where $i^*$ is defined in \eqref{def:x*}. By the functional law of large numbers we have that $\widetilde{h}_n (x) \to q x$ uniformly in $[0, 1]$ and consequently $\widetilde{h}_n^{-1} (s) \to s/q$.
	
The definition of the $J_1$ metric on $[0, 1]$ is the following
\begin{equation}\label{eqn:Skorohod}
	d_{J_1}\left(f, g \right) \coloneqq \inf_{\xi \in \Xi} \left( \sup_{t \in [0, 1]} \left| f \circ \xi(t) - g(t) \right| + \sup_{t \in [0, 1]} \left| \xi(t) - t \right|  \right),
\end{equation}
where $\Xi$ is the set of continuous, strictly increasing functions that map $[0, 1]$ onto itself (which necessarily admit continuous and strictly increasing inverses). Thus we need to prove that for every $\varepsilon > 0$ there exists $n_0$ and $\xi_n \in \Xi$ such that, for all $n \ge n_0$,
\begin{equation}\label{eqn:J_1Small}
	\sup_{t \in [0, 1]} \left| S^{\alpha_0, \lambda/n, (n)}(t) - q^{-1/\alpha_0}S^{\alpha_0, \lambda/q}(q\xi_n(t)) \right| + \sup_{t \in [0, 1]} \left| \xi_n(t) - t \right| \le \varepsilon.
\end{equation}
Recall the definition of the set $I^{(n), \alpha_\infty}_\delta$ from \eqref{eqn:SetLargeIncrements}, and consider analogously $I^{(n), \alpha_0}_\delta$. Furthermore, let
\begin{equation*}
	I^{\alpha_0, q}_\delta \coloneqq \left\{ t \in [0, q] \colon S^{\alpha_0}(t) - S^{\alpha_0}(t^{-}) > \delta \right\}.
\end{equation*}
Let us notice the following three claims hold for any $\delta>0$.
\begin{itemize}
	\item For all $n$ large enough, almost-surely, by the vague and point process convergence of Proposition~\ref{PropositionPoissonPoint}, the two sets $I^{(n), \alpha_0}_\delta$ and $I^{\alpha_0, q}_\delta$ will have matching atoms, in the sense that for every $x_j \in I^{\alpha_0, q}_\delta$ one can find $i \in I^{(n), \alpha_0}_\delta$ such that
	\begin{equation*}
		\widetilde{h}_n\left( \frac{i}{n}\right) = \frac{i - i^*}{n} \to x_j,\qquad g_{n}^{\alpha_0} \left( S^{\alpha_0}\left(\frac{i - i^*+1}{n} \right) - S^{\alpha_0}\left(\frac{i - i^*}{n} \right) \right) \to S^{\alpha_0}(x_j) - S^{\alpha_0}(x_j^{-}).
	\end{equation*}
    \item For all $n$ large enough the number of atoms in $I^{(n), \alpha_0}_\delta$ is the same as $|I^{\alpha_0, q}_\delta|$ and both are almost-surely finite.
    \item For all $n$ large enough the set of points $\{\widetilde{h}_n^{-1}(x_j) \colon x_j \in I^{\alpha_0, q}_\delta\} \cup \{0, 1\}$ is well defined in the sense that the map $\widetilde{h}_n^{-1}$ is injective on $I^{\alpha_0, q}_\delta$ and does not map any point to $\{0, 1\}$.
\end{itemize}
Thus we define $\xi_{n}(t)$ to be the inverse of the linear interpolation of the points $\{(1/q) x_j, \widetilde{h}_n^{-1}(x_j)\}$ for $x_j \in I^{\alpha_0, q}_\delta$ and with the convention that $0$ and $1$ are mapped onto themselves. The observations above imply that $\xi_n(\cdot)$ is well defined and is inside $\Xi$ for all $n$ large enough. This choice implies that for all $\varepsilon>0$, all $\delta>0$ and all $n$ large enough
\[\sup_{t \in [0, 1]} \left| \sum_{\substack{i = 0 \\ i \in I^{(n), \alpha_0}_\delta}}^{\floor{nt} - 1} e^{-\frac{2\lambda i}{n} }g_{n}^{\alpha_0} \left( S^{\alpha_0}\left(\tfrac{i - i^*+1}{n} \right) - S^{\alpha_0}\left(\tfrac{i - i^*}{n} \right) \right) - \sum_{\substack{j: \frac{x_j}{q} \le \xi_n(t),\\ x_j \in I^{\alpha_0, q}_\delta}}  e^{-\frac{2x_{j}}{q}} \left(S^{\alpha_0}(x_j) - S^{\alpha_0}(x_j^{-})\right) \right| \le \frac{\varepsilon}{4}.\]
Basically, $\xi_n$ exactly matches the location of the discontinuities that are larger than $\delta$. Furthermore, notice that by the third bullet point above and the fact that $\widetilde{h}_n(x)$ converges uniformly towards $q x$, for all $n$ large enough
\[\sup_{t \in [0, 1]} \left| \xi_n(t) - t \right| \le \frac{\varepsilon}{4}.\]
Moreover, the law of large numbers and the results in \eqref{eqn:SmallDiscrete}, \eqref{eqn:SmallAtomsAreSmall} and \eqref{eqn:VerySmallDiscrete} (re-phrased for the $\alpha_{0}$ process) imply that one can always pick $\delta$ small enough such that, for all $n$ large
\[	\sum_{\substack{i = 0 \\ i \not \in I^{(n), \alpha_0}_\delta \backslash I^{(n), \alpha_0}_0}}^{\floor{n} - 1} g_{n}^{\alpha_0} \left( S^{\alpha_0}\left(\frac{i - i^*+1}{n} \right) - S^{\alpha_0}\left(\frac{i - i^*}{n} \right) \right) + \frac{1}{d_{n, 0}}\sum_{i/n \in [0, 1], \, b_i = 1} \widebar{r}(\{i, i+1\})  \le \frac{\varepsilon}{4}, \]
	\[\sum_{x_j \not \in I^{\alpha_0, q}_\delta} S^{\alpha_0}\left(x_j \right) - S^{\alpha_0}\left(x_j^{-}\right) \le \frac{\varepsilon}{4}.\]
The last sum is taken over all discontinuity points of the stable subordinator, which is a pure jump process. This concludes the proof, as we have shown that \eqref{eqn:J_1Small} holds as $e^{-2\lambda x} \le 1$ for $x \in [0, 1]$. On this last observation, note that for the general case $[-K, K]$ for $K$ fixed we can bound $e^{-2\lambda x}$ with the constant $e^{2 \lambda K}$, the result follows as $\varepsilon$ above can be chosen arbitrarily small.
\end{proof}

\begin{lemma}\label{Lemma:HittingKComparison}
	For every $t>0$ we have that
	\[		\limsup_{n} \,\,\mathbb{P}^{\lambda/n, K}\left( \tau_{Kn}^{X} \wedge \tau_{-Kn}^{X} \le ta_n \right) \le \mathbb{P}^{\lambda, K}\left( \tau_{K - 1}^{Z} \wedge \tau_{-K + 1}^{Z} \le t + 1 \right).	\]
	and that
	\begin{equation}\label{eqn:2HittingTimesComparedRWT}
		\limsup_{n} \,\,\widetilde{\mathbb{P}}^{\lambda/n, K}\left( \tau_{Kn}^{\widetilde{X}} \wedge \tau_{-Kn}^{\widetilde{X}} \le tb_n \right) \le \widetilde{\mathbb{P}}^{\lambda, K}\left( \tau_{K - 1}^{\widetilde{Z}} \wedge \tau_{-K + 1}^{\widetilde{Z}} \le t + 1 \right).
	\end{equation}
\end{lemma}

\begin{proof}
	The proof goes in the same way as the proof of \cite[Lemma~5.2]{MottModel}, being a simple consequence of $J_1$-convergence.
\end{proof}

%DC - Note that Cerny appears in wrong place - adjust in final version.
\printbibliography
\end{document}